\documentclass[12pt,amstex]{amsart}

\usepackage{xr}
\usepackage[shortlabels]{enumitem}
\usepackage[english]{babel}
\usepackage{amscd, amsfonts, amsmath, amssymb, amsthm, epsfig, float, geometry, graphicx, pstricks, pst-node, txfonts, enumitem}

\usepackage[all]{xy}
\usepackage{tikz}
\usepackage{multicol}

\usepackage{hyperref}
\hypersetup{
	colorlinks,
	citecolor=black,
	filecolor=black,
	linkcolor=black,
	urlcolor=black
}

\numberwithin{equation}{section}

\theoremstyle{definition}
\newtheorem{thm}{Theorem}[section]
\newtheorem{lem}[thm]{Lemma}
\newtheorem{defn}[thm]{Definition}
\newtheorem{rem}[thm]{Remark}
\newtheorem{prop}[thm]{Proposition}
\newtheorem{ques}[thm]{Question}
\newtheorem{coro}[thm]{Corollary}
\newtheorem{ex}[thm]{Example}

\newtheorem{conv}[thm]{Convention}

\def \C {\mathbb{C}}
\def \E {\mathbb{E}}
\def \F {\mathbb{F}}

\def \L {\bold{L}}

\def \N {\mathbb{N}}
\def \O {\mathcal{O}}
\def \P {\bold{P}}
\def \Q {\mathbb{Q}}
\def \R {\mathbb{R}}
\def \T {\mathbb{T}}
\def \V {\mathbb{F}_{p}^{d}}

\def \Z {\mathbb{Z}}

\def \+ {\hat{+}}
\def \- {\hat{-}}

\def \b {\bold{b}}

\def \d {\delta}
\def \e {\epsilon}

\def \h {\bold{h}}

\def \l {\lambda}

\def \w {\bold{w}}
\def \x {\bold{x}}
\def \y {\bold{y}}

\def \ca {\curvearrowright}

\def \CP {\{\text{pure, nice, independent, consistent}\}}
\def \da {\downharpoonleft}

\def \err {\text{err}}

\def \Gow {\Box}

\def \Lip {\text{Lip}}

\def \poly {\text{poly}}
\def \pp {\perp_{M}}

\def \rank {\text{rank}}

\def \sp {\text{span}}
\def \str {\text{str}}

\def \uni {\text{uni}}

\title[Spherical higher order Fourier analysis over finite fields IV]{Spherical higher order Fourier analysis over finite fields IV: an application to the Geometric Ramsey Conjecture}
\author{Wenbo Sun}
\address[Wenbo Sun]{Department of Mathematics, Virginia Tech, 225 Stanger Street, Blacksburg, VA, 24061, USA}
\email{swenbo@vt.edu}

\thanks{The author was partially supported by the NSF Grant DMS-2247331}


\begin{document}

\maketitle

\begin{abstract}
  This paper is the fourth and the last part of the series \emph{Spherical higher order Fourier analysis over finite fields}, aiming to develop the higher order Fourier analysis method along spheres over finite fields, and to solve the Geometric Ramsey Conjecture in the finite field setting.
  
  In this paper, we proof a conjecture of Graham on the Remsey properties for spherical configurations in the finite field setting.
To be more precise, we
  show that for any spherical configuration $X$ of $\mathbb{F}_{p}^{d}$ of complexity at most $C$ with $d$ being sufficiently large with respect to $C$ and $\vert X\vert$, and for some prime $p$ being sufficiently large with respect to $C$, $\vert X\vert$ and $\epsilon>0$, any set $E\subseteq \mathbb{F}_{p}^{d}$ with $\vert E\vert>\epsilon p^{d}$ contains at least $\gg_{C,\epsilon,\vert X\vert}p^{(k+1)d-(k+1)k/2}$ congruent copies of $X$, where $k$ is the dimension of $\text{span}_{\mathbb{F}_{p}}(X-X)$. The novelty of our approach is that we avoid the use of harmonic analysis, and replace it by the theory of spherical higher order Fourier analysis developed in previous parts of the series.
\end{abstract}

\tableofcontents

\section{Introduction}

 \subsection{Geometric Ramsey Theorem over finite fields}  	 
This paper is the forth and last part of the series \emph{Spherical higher order Fourier analysis over finite fields} \cite{SunA,SunB,SunC}. The purpose of this paper is to use the tools developed in \cite{SunA,SunB,SunC} to solve the Geometric Ramsey Conjecture in the finite field setting.
 Geometric Ramsey theory is a topic first introduced in a series of papers by Erd\"os et al.  \cite{E73a,E73b,E73c}, which studies geometric patterns that can not be destroyed by partitioning Euclidean space into finitely many parts. We say that a finite set $X$ is \emph{geometrically Ramsey} if for every $r\in\N_{+}$, there exists a dimension $d=d(r,X)$ such that every $r$-coloring of $\mathbb{R}^{d}$ contains a monochromatic and congruent copy of $\lambda X$ for sufficiently large $\lambda>0$.  
 It was shown in \cite{E73a} that geometrically Ramsey sets must be \emph{spherical}, meaning that $E$ is contained in the surface of a sphere. This led to a longstanding problem in combinatorics  called the \emph{Geometric Ramsey Conjecture} made by Graham \cite{Gra94}, saying that a finite set $E\subseteq \R^{d}$ is geometrically Ramsey if and only if it is spherical.

   This conjecture is far from being settled and only partial results are known, including vertices of $k$-dimensional cubes \cite{E73a}, non-degenerating simplices \cite{Bou86,FR90} (see also \cite{LMb}), products of simplices \cite{LM18}, trapezoids \cite{Kriz92}, regular polygons and polyhedra \cite{Kriz91}.

Since  vector spaces over finite fields are useful models to study many problems in combinatorics and number theory, it is natural to ask whether Graham's conjecture holds in this setting. We start with some definitions.
For $x=(x_{1},\dots,x_{d}),y=(y_{1},\dots,y_{d})\in\V$, define the \emph{dot product} by $x\cdot y:=x_{1}y_{1}+\dots+x_{d}y_{d}\in\F_{p}$ and let $\vert x\vert^{2}:=x\cdot x$.
 For convenience we call an ordered finite subset of $\V$ to be a \emph{configuration} in $\V$.
	Let
$X=\{x_{j}\colon j\in J\}\subseteq\F_{p}^{d}$ and $Y=\{y_{j}\colon j\in J\}\subseteq \F_{p}^{d}$ be two configurations in $\V$.
We say that $Y$ is \emph{congruent} to $X$ if there exist an isometry $U$ on $\V$ (i.e. $U\colon \V\to\V$ is a linear transformation such that $\vert U(x)\vert^{2}=\vert x\vert^{2}$ for all $x\in\V$) and $z\in\V$ such that $y_{j}=z+U(x_{j})$ for all $j\in J$.
We say that $X\subseteq \V$ is \emph{spherical} if there exist $u_{0}\in\V$ and $r\in\F_{p}$ such that $\vert u-u_{0}\vert^{2}=r$ for all $u\in X$. 

Instead of studying the Geometric Ramsey Conjecture in the finite field setting, we formulate a density version of it as follows: for any spherical configuration $X\subseteq \V$ and any subset $E\subseteq \V$ with $\vert E\vert>\e p^{d}$, $E$ contains a  congruent copy of $X$ if $p$ is sufficiently large depending on $d$ and $X$. We remark that the Geometric Ramsey Conjecture in the finite field setting fails for non-spherical configurations by Lemma 17 of \cite{LMP19}.

In this paper, we provide an affirmative answer to the density version of the Geometric Ramsey Conjecture in the finite field setting for general spherical configurations.
For any $X\subseteq\F_{p}^{d}$, let $k:=\dim(\sp_{\F_{p}}(X-X))$, $s:=\vert X\vert-k\geq 1$, and denote \begin{equation}\label{4:finald}
\begin{split}
 d_{0}(X):=
 \max\{(2s+12)(15s+423),4k^{2}+12k+7\}.
\end{split}
\end{equation}

\begin{thm}[Density version of the Geometric Ramsey Theorem for finite fields]\label{4:mainmain}
Let $d\in\N_{+}$, $C,\e>0$, $p$ be a prime and $X\subseteq\F_{p}^{d}$ be a spherical configuration of complexity at most $C$ (see Definition \ref{4:ccgf}). Denote $k:=\dim(\sp_{\F_{p}}(X-X))$, $s:=\vert X\vert-k\geq 1$ and suppose that $d\geq d_{0}(X)$. There exists $p_{0}=p_{0}(C,\e,k,s)\in\N$ and $\d=\d(C,\e,k,s)>0$ such that if $p>p_{0}$, then  every set $E\subseteq\V$ with $\vert E\vert>\e p^{d}$  contains at least $\d p^{(k+1)d-(k+1)k/2}$  congruent copies of $X$. 
\end{thm}	

We remark that the quantity $p^{(k+1)d-(k+1)k/2}$ arises from the observation that a  random subset $E$ of $\V$ contains approximately $\e^{\vert X\vert} p^{(k+1)d-(k+1)k/2}$ isometric  copies of $X$ (see Section \ref{4:s:msms} for the definition), where  $k$ is the dimension of $\sp_{\F_{p}}(X-X)$. This is because there are $k(k+1)/2$ quadratic relations given by distances between the points in the generating set of $X$, which determines the locations of all other points in $X$, and each vertex is contained in $E$ with probability $\e$ (see also page 2 of \cite{LMP19}).

The case  $s=1$ of  Theorem \ref{4:mainmain}, where $X$ is a simplex was proved by Parshall \cite{Par17} (which can be viewed as an extension Bourgain's work \cite{Bou86}).  For the case $s=2$, 
  Theorem \ref{4:mainmain} was proved by Lyall, Magyar and Parshall \cite{LMP19}  
   under the additional assumption that $X$ is \emph{non-degenerate} (meaning that $X-X$ is a non-isotropic subspace of $\V$).\footnote{The main result Theorem 1 in \cite{LMP19} was stated only for isometric copies, but its proof also applies to the case of congruent copies.}  It is worth noting that the quantity $\d$ in the aforementioned results  was independent of the complexity of $X$, whereas in Theorem \ref{4:mainmain}, unfortunately we were unable to drop the dependency on the complexity of the configuration. So it is natural to ask:
  
  \begin{ques}
      Can we remove the dependence of $p_{0}$ and $\d$ on $C$ in Theorem \ref{4:mainmain}?
  \end{ques}

The lower bound of $d$ we obtain in Theorem \ref{4:mainmain} is rather coarse. So it is also natural to ask:

  \begin{ques}
      What is the optimal lower bound for $d$ in Theorem \ref{4:mainmain}?
  \end{ques}

\subsection{Strategy of the proof}
The study of the Geometric Ramsey Conjecture in literature uses methods from classical harmonic analysis. Relying on meticulous estimates on Fourier coefficients, such methods proved to be powerful when $s\leq 2$ or when the configuration enjoys some extra symmetric properties. However, the harmonic analysis methods   no longer work when $s\geq 3$. This is because the Fourier coefficient is  insufficient to describe the spherical averages (see for example (\ref{4:Vdc0})) pertaining to this problem. 
 
  In this paper, instead of using harmonic analysis, we use the approach of the higher order Fourier analysis developed in the previous parts of the series \cite{SunA,SunB,SunC}. 
  In Section \ref{4:s:3rr}, we provide some basic algebraic properties which will be used in later sections. In Section \ref{4:s:4rr}, by reducing Theorem \ref{4:mainmain} to Theorem \ref{4:mainmain2}, we show that it suffices to consider the Geometric Ramsey Conjecture for subsets of a sphere. 
  
  The problem is then reduced to show the positivity of a spherical average of the form (\ref{4:Vdc0}). Motivated by the work of Lyall, Magyar and Parshall \cite{LMP19}, we show in Section \ref{4:s:5rr} that (\ref{4:Vdc0}) can be bounded by the spherical Gowers norms of the functions $f_{i}$. Then in Section \ref{4:s:2rr}, based on the work in the previous parts of the series \cite{SunA,SunC}, we provide a structure theorem for spherical Gowers norms. It turns out that the factorization result in \cite{SunA} is not good for our purposes, and so we provide a better version of it based on the idea of Green and Tao \cite{GT10b}.
 
 The structure theorem essentially allows us to reduce the expression (\ref{4:Vdc0}) to the case where $f_{i}$ are sufficiently irrational nilsequences. In Section \ref{4:slast}, by developing the approach of Green and Tao  \cite{GT10b}, we prove a joint equidistribution result (Theorem \ref{4:orbitdesc2}) which can be used to deal with this case. Finally, in Section \ref{4:s:7rr}, we combine all the results in previous sections to complete the proof of Theorem \ref{4:mainmain2} and thus  Theorem \ref{4:mainmain}.

 In Appendix \ref{4:s:AppA}, we recall some notations on polynomials and nilmanifolds defined in \cite{SunA}. In Appendix \ref{4:s:AppB}, we collect some results from previous parts of the series \cite{SunA,SunC} which are used in this paper.

\subsection{Definitions and notations}\label{4:s:defn}

  \begin{conv}
  Throughout this paper, we use
   $\tau\colon\F_{p}\to \{0,\dots,p-1\}$ to denote the natural bijective embedding, and use $\iota\colon \Z\to\F_{p}$ to denote the map given by $\iota(n):=\tau^{-1}(n \mod p\Z)$.
	We also use 
	$\tau$ to denote the map from $\F_{p}^{k}$ to $\Z^{k}$ given by $\tau(x_{1},\dots,x_{k}):=(\tau(x_{1}),\dots,\tau(x_{k}))$,
	and
	$\iota$ to denote the map from $\Z^{k}$ to $\F_{p}^{k}$ given by $\iota(x_{1},\dots,x_{k}):=(\iota(x_{1}),\dots,$ $\iota(x_{k}))$. 
	
	We may also extend the domain of $\iota$ to all the rational numbers of the form $x/y$ with $(x,y)=1, x\in\Z, y\in\Z\backslash p\Z$ by setting $\iota(x/y):=\iota(xy^{\ast})$, where $y^{\ast}$ is any integer with $yy^{\ast}\equiv 1 \mod p\Z$.
  \end{conv}

Below are the notations we use in this paper:

\begin{itemize}
	\item Let $\N,\N_{+},\Z,\Q,\R,\R+,\C$ denote the set of non-negative integers, positive integers, integers, rational numbers, real numbers, positive real numbers, and complex numbers, respectively. Denote $\T:=\R/\Z$. Let $\F_{p}$ denote the finite field with $p$ elements. 
		\item Throughout this paper, $d$ is a fixed positive integer and $p$ is a prime number.
		\item Throughout this paper, unless otherwise stated, all vectors are assumed to be horizontal vectors.
		\item Let $\mathcal{C}$ be a collection of parameters and $A,B,c\in\R$. We write $A\gg_{\mathcal{C}} B$ if $\vert  A\vert\geq K\vert B\vert$ and $A=O_{\mathcal{C}}(B)$ if $\vert A\vert\leq K\vert B\vert$ for some $K>0$ depending only on the parameters in $\mathcal{C}$.
In the above definitions, we allow the set $\mathcal{C}$ to be empty. In this case $K$ will be a universal constant.
\item Denote $[p]:=\{0,\dots,p-1\}$,
	\item
	For $i=(i_{1},\dots,i_{k})\in\Z^{k}$, denote   $\vert i\vert:=\vert i_{1}\vert+\dots+\vert i_{k}\vert$.
	 For $n=(n_{1},\dots,n_{k})\in\Z^{k}$ and $i=(i_{1},\dots,i_{k})\in\N^{k}$, denote $n^{i}:=n_{1}^{i_{1}}\dots n_{k}^{i_{k}}$ 
 and $i!:=i_{1}!\dots i_{k}!$. 
For $n=(n_{1},\dots,n_{k})\in\N^{k}$ and $i=(i_{1},\dots,i_{k})\in\N^{k}$, denote $\binom{n}{i}:=\binom{n_{1}}{i_{1}}\dots \binom{n_{k}}{i_{k}}$.	
	\item We say that $\mathcal{F}\colon \R_{+}\to\R_{+}$ is a \emph{growth function} if $\mathcal{F}$ is strictly increasing and $\mathcal{F}(n)\geq n$ for all $n\in \R_{+}$.
    \item For $x\in\R$, let $\lfloor x\rfloor$ denote the largest integer which is not larger than $x$, and $\lceil x\rceil$ denote the smallest integer which is not smaller than $x$. Let $\{x\}:=x-\lfloor x\rfloor$.
	\item Let  $X$ be a finite set and $f\colon X\to\C$ be a function. Denote $\E_{x\in X}f(x):=\frac{1}{\vert X\vert}\sum_{x\in X}f(x)$, the average of $f$ on $X$.
	\item We say that a set $\Omega\subseteq \Z^{k}$ is \emph{$p$-periodic} if $\Omega=\Omega+p\Z^{k}$.
	\item For $F=\Z^{k}$ or $\F_{p}^{k}$, and $x=(x_{1},\dots,x_{k}), y=(y_{1},\dots,y_{k})\in F$, let $x\cdot y\in \Z$ or $\F_{p}$ denote the dot product given by
	$x\cdot y:=x_{1}y_{1}+\dots+x_{k}y_{k}.$
	\item If $G$ is a connected, simply connected Lie group, then we use  $\log G$ to denote its Lie algebra. Let $\exp\colon \log G\to G$ be the exponential map, and $\log\colon G\to \log G$ be the logarithm map. For $t\in\R$ and $g\in G$, denote $g^{t}:=\exp(t\log g)$.
	\item If $f\colon H\to G$ is a function from an abelian group $H=(H,+)$ to some group $(G,\cdot)$, denote $\Delta_{h} f(n):=f(n+h)\cdot f(n)^{-1}$ for all $n,h\in H$.
	\item If $f\colon H\to \C$ is a function on an abelian group $H=(H,+)$, denote $\Delta_{h} f(n):=f(n+h)\overline{f}(n)$ for all $n,h\in H$.
	\item We write affine subspaces of $\V$ as $V+c$, where $V$ is a subspace of $\V$ passing through $\bold{0}$, and $c\in\V$.
	\item There is a natural correspondence between polynomials taking values in $\F_{p}$ and polynomials  taking values in $\Z/p$.   Let $F\in\poly(\V\to\F_{p}^{d'})$ and $f\in \poly(\Z^{d}\to (\Z/p)^{d'})$ be polynomials of degree at most $s$ for some $s<p$. If  $F=\iota\circ pf\circ\tau$, then we say that $F$ is \emph{induced} by $f$ and $f$ is a \emph{lifting} of $F$.\footnote{As is explained in \cite{SunA},   $\iota\circ pf\circ\tau$ is well defined.}
     We say that $f$ is a  \emph{regular lifting} of $F$ if in addition $f$ has the same degree as $F$ and $f$ has $\{0,\frac{1}{p},\dots,\frac{p-1}{p}\}$-coefficients. 
         \item Let $(X,d_{X})$ be a metric space.
The \emph{Lipschitz norm} of a function $F\colon X\to\C$ is defined as 
$$\Vert F\Vert_{\Lip(X)}:=\sup_{x\in X}\vert F(x)\vert+\sup_{x,y\in X, x\neq y}\frac{\vert F(x)-F(y)\vert}{d_{X}(x,y)}.$$
When there is no confusion, we write $\Vert F\Vert_{\Lip}=\Vert F\Vert_{\Lip(X)}$ for short. 
 \end{itemize}	

 Let $D,D'\in\N_{+}$ and $C>0$. 
Here are some basic notions of complexities:

	\begin{itemize}
		\item \textbf{Real and complex numbers:} a  number $r\in\R$ is of \emph{complexity} at most $C$ if $r=a/b$ for some $a,b\in\Z$ with $-C\leq a,b\leq C$. If $r\notin \Q$, then we say that the complexity of $r$ is infinity. A complex number is of \emph{complexity} at most $C$ if both its real and imaginary parts are of complexity at most $C$.
	   \item \textbf{Vectors and matrices:} a vector or matrix is of \emph{complexity} at most $C$ if all of its entries are of  complexity at most $C$.
	   \item \textbf{Subspaces:} a subspace of $\R^{D}$ is of \emph{complexity} at most $C$ if it is the null space of a matrix of complexity at most $C$.
	  \item \textbf{Linear transformations:} let $L\colon\C^{D}\to\C^{D'}$ be a linear transformation. Then $L$ is associated with an $D\times D'$ matrix $A$ in $\C$. We say that $L$ is of \emph{complexity} at most $C$ if  $A$ is of complexity at most $C$.
	   \item \textbf{Lipschitz function:} the \emph{complexity} of a Lipschitz function is defined to be its Lipschitz norm.
	\end{itemize}

We also need to recall the notations regarding quadratic forms defined in \cite{SunA}.	
	
\begin{defn}	
 We say that a function $M\colon\V\to\F_{p}$ is a \emph{quadratic form} if 
	$$M(n)=(nA)\cdot n+n\cdot u+v$$
	for some $d\times d$ symmetric matrix $A$ in $\F_{p}$, some $u\in \F_{p}^{d}$ and $v\in \F_{p}$.
We say that $A$ is the matrix \emph{associated to} $M$.
We say that $M$ is \emph{pure} if $u=\bold{0}$.
We say that $M$ is \emph{homogeneous} if $u=\bold{0}$ and $v=0$. We say that $M$ is \emph{non-degenerate} if $M$ is of rank $d$, or equivalently, $\det(A)\neq 0$.
\end{defn}

We use $\rank(M):=\rank(A)$ to denote the \emph{rank} of $M$.
Let $V+c$ be an affine subspace of $\V$ of dimension $r$. There exists a (not necessarily unique) bijective linear transformation $\phi\colon \F_{p}^{r}\to V$.
 We define the \emph{rank} $\rank(M\vert_{V+c})$ of $M$  restricted to $V+c$ as the rank of the quadratic form $M(\phi(\cdot)+c)$. It was proved in \cite{SunA} that   $\rank(M\vert_{V+c})$ is independent of the choice of $\phi$.

We also need to use the following notions.

\begin{itemize}
\item For a polynomial $P\in\poly(\F_{p}^{k}\to\F_{p})$, let $V(P)$ denote the set of $n\in\F_{p}^{k}$ such that $P(n)=0$.
\item Let $r\in\N_{+}$, $h_{1},\dots,h_{r}\in \V$ and $M\colon\V\to\F_{p}$ be a quadratic form. Denote $$V(M)^{h_{1},\dots,h_{r}}:=\cap_{i=1}^{r}(V(M(\cdot+h_{i}))\cap V(M).$$
\item Let $\Omega$ be a subset of $\V$ and $s\in\N$.  Let 
$\Gow_{s}(\Omega)$ denote the set of $(n,h_{1},\dots,h_{s})\in(\V)^{s+1}$ such that $n+\e_{1}h_{1}+\dots+\e_{s}h_{s}\in\Omega$ for all $(\e_{1},\dots,\e_{s})\in\{0,1\}^{s}$.
Here we allow $s$ to be 0, in which case $\Gow_{0}(\Omega)=\Omega$.
We say that $\Gow_{s}(\Omega)$ is the \emph{$s$-th Gowers set} of $\Omega$.
\end{itemize}

Quadratic forms can also be defined in the $\Z/p$-setting.

\begin{defn}	
   We say that a function $M\colon\Z^{d}\to\Z/p$ is a \emph{quadratic form} if 
	$$M(n)=\frac{1}{p}((nA)\cdot n+n\cdot u+v)$$
	for some $d\times d$ symmetric matrix $A$ in $\Z$, some $u\in \Z^{d}$ and $v\in \Z$.
We say that $A$ is the matrix \emph{associated to} $M$. 
 \end{defn}

By Lemma 
A.1 of \cite{SunA},
any quadratic form $\tilde{M}\colon\Z^{d}\to\Z/p$ associated with the matrix $\tilde{A}$ induces a quadratic form $M:=\iota\circ p\tilde{M}\circ\tau\colon\F_{p}^{d}\to\F_{p}$ associated with the matrix $\iota(\tilde{A})$. Conversely, any quadratic form $M\colon\F_{p}^{d}\to\F_{p}$ associated with the matrix $A$ admits a regular lifting $\tilde{M}\colon\Z^{d}\to\Z/p$, which is a quadratic form associated with the matrix $\tau(A)$.

For a quadratic form $\tilde{M}\colon\Z^{d}\to\Z/p$, we say that $\tilde{M}$ is \emph{pure/homogeneous/$p$-non-degenerate} if the  quadratic form $M:=\iota\circ p\tilde{M}\circ\tau$ induced by $\tilde{M}$ is pure/homogeneous/non-degenerate. 
The \emph{$p$-rank} of $\tilde{M}$, denoted by $\rank_{p}(\tilde{M})$, is defined to be the rank of $M$.

We say that $h_{1},\dots,h_{k}\in\Z^{d}$ are \emph{$p$-linearly independent} if for all $c_{1},\dots,c_{k}\in\Z/p$, $c_{1}h_{1}+\dots+c_{k}h_{k}\in\Z$
 implies that $c_{1},\dots,c_{k}\in\Z$, or equivalently, if $\iota(h_{1}),\dots,\iota(h_{k})$ are linearly independent.

 We also need to use the following notions.

\begin{itemize}
\item For a polynomial $P\in\poly(\Z^{k}\to\R)$, let $V_{p}(P)$ denote the set of $n\in \Z^{k}$ such that $P(n+pm)\in \Z$ for all $m\in\Z^{k}$.
\item For $\Omega\subseteq\Z^{d}$ and $s\in\N$, let $\Gow_{p,s}(\Omega)$ denote the set of $(n,h_{1},\dots,h_{s})\in(\Z^{d})^{s+1}$ such that $n+\e_{1}h_{1}+\dots+\e_{s}h_{s}\in\Omega+p\Z^{d}$ for all $\e_{1},\dots,\e_{s}\in\{0,1\}$. 
We say that $\Gow_{p,s}(\Omega)$ is the \emph{$s$-th $p$-Gowers set} of $\Omega$.
\end{itemize}

For all other definitions and notations which are used but not mentioned in this paper, we refer the readers to Appendices \ref{4:s:AppA} and \ref{4:s:AppB} for details.

\section{Algebraic structures for configurations}\label{4:s:3rr}

In this section, we provide some basic properties for various algebraic structures pertaining to the Geometric Ramsey Conjecture.  

\subsection{Generating sets}

 Let
 $X=\{x_{j}\in\V\colon j\in J\}\subseteq\F_{p}^{d}$  be a configuration in $\V$ with $J$ being a finite ordered index set. 
 Let $I$ be a  subset of $J$. We say that $I$  is a \emph{generating set} of $X$ if the vectors $x_{i}-x_{i_{0}}, i\in I\backslash\{i_{0}\}$ is a basis of $\sp_{\F_{p}}(X-X)$ for some  $i_{0}\in I$. 
 The following lemma is straightforward and we omit its proof.
 \begin{lem}\label{4:lemgg}
 	$X=\{x_{j}\in\V\colon j\in J\}\subseteq\F_{p}^{d}$  be a configuration in $\V$ with $J$ being a finite ordered index set. 
 	\begin{enumerate}[(i)]
 		\item A subset $I\subseteq J$ is a generating set of $X$  if and only if the vectors $x_{i}-x_{i_{0}}, i\in I\backslash\{i_{0}\}$ is a basis of $\sp_{\F_{p}}(X-X)$ for every $i_{0}\in I$. In particular, the cardinalities of generating sets of $X$ are the same.
 		\item For every $j\in J$, there exists a generating set $I\subseteq J$ of $X$ containing $j$.
 	\end{enumerate}	
 \end{lem}

 If $I$ is a generating set of $X$, then it is clear that  there exist unique $b_{I,i,j}\in\F_{p}, i\in I, j\in J$ such that $\sum_{i\in I}b_{I,i,j}=1$ and $x_{j}=\sum_{i\in I}b_{I,i,j}x_{i}$ for all $j\in J$. 
 We remark that $b_{I,i,j}, i\in I, j\in J$ satisfy the following \emph{normalization property:}
 $$\text{$b_{I,i,j}=\d_{i,j}$ for all $i,j\in I$ and $\sum_{i\in I}b_{I,i,j}=1$ for all $j\in J$.}$$
 Let $L_{I,j}\colon(\V)^{\vert I\vert}\to\V$ denote the linear transformation given by
 $$L_{I,j}(\x):=\sum_{i\in I} b_{I,i,j}x_{i},$$
 and let $\L_{I}\colon(\V)^{\vert I\vert}\to(\V)^{\vert J\vert}$ denote the linear transformation given by
 $$\L_{I}(\x):=(L_{I,j}(\x))_{j\in J}.$$
 We say that $L_{I,j}, j\in J$ (or simply $\L_{I}$) are the  \emph{generating maps} for $X$ with respect to $I$, and $b_{I,i,j}, i\in I, j\in J$ (or $b_{I,i}, i\in I$) are the  \emph{generating constants} for $X$ with respect to $I$.
  We remark that the normalization property of $b_{I,i,j}$ implies the following \emph{normalization property} for $L_{I,j}$:
 $$\text{$L_{I,j}((x_{i})_{i\in I})=x_{j}$ for all $j\in I$ and $L_{I,j}(x,\dots,x)=x$ for all $j\in I$.}$$

 We are now ready to define the complexity of a configuration.

\begin{defn}[Complexity of configurations]\label{4:ccgf}
    Let $C>0$. We say that a configuration $X\subseteq\V$ is of \emph{complexity} at most $C$ if there exists  a generating set $I$ of $X$ such that all the generating constants for $X$ with respect to $I$ belong to $\{-C,\dots,C\}$.
\end{defn}

In particular, if  $X\subseteq \{-C,\dots,C\}^{d}$, then $X$ is of complexity at most $O_{C,k}(1)$, where $k=\dim(\sp_{\F_{p}}(X-X))$.

Next we provide a change of variable property for generating maps. 

\begin{prop}[Change of variable]\label{4:kk4k}
Let $d\in\N_{+},$ $p$ be a prime, $X$ be a configuration in $\V$ with $J$ being a finite ordered index set, and $I,I'\subseteq J$ be two generating sets of $X$ with generating maps $\L_{I}\colon (\V)^{\vert I\vert}\to (\V)^{\vert J\vert}$ and $\L_{I'}\colon (\V)^{\vert I'\vert}\to (\V)^{\vert J\vert}$ respectively. Let $\pi_{I}\colon (\V)^{\vert J\vert}\to (\V)^{\vert I\vert}$ be the projection onto the coordinates with $I$-indices. 
  Then we have that 
$\L_{I}\circ\pi_{I}\circ\L_{I'}=\L_{I'}$. In particular, we have $\L_{I}((\V)^{k+1})=\L_{I'}((\V)^{k+1})$.
\end{prop}
\begin{proof}
Assume that $X=\{v_{j}\colon j\in J\}$ and 
	let $b_{I,i,j}, i\in I, j\in J$ and $b_{I',i,j}, i\in I', j\in J$ be the generating constants for $X$ with respect to $I$ and $I'$ respectively. 
	We need to show that for all $j\in J$ and $\x=(x_{i})_{i\in I'}\in(\V)^{\vert I'\vert}$, we have that 
	\begin{equation}\nonumber
	\begin{split}
	L_{I,j}\circ\pi_{I}\circ\L_{I'}(\x)=L_{I',j}(\x).
	\end{split}
	\end{equation}
	
	For any $j\in J$, we have that
	\begin{equation}\nonumber
	\begin{split}
	   \sum_{i'\in I'}b_{I',i',j}v_{i'}
	   =v_{j}=\sum_{i\in I}b_{I,i,j}v_{i}
	   =\sum_{i\in I}b_{I,i,j}\sum_{i'\in I'}b_{I',i',i}v_{i'}
	   =\sum_{i'\in I'}\Bigl(\sum_{i\in I}b_{I,i,j}b_{I',i',i}\Bigr)v_{i'},
	 \end{split}
	\end{equation}
	or equivalently, 
	\begin{equation}\label{4:cov001}
	\begin{split}
	   \sum_{i'\in I'}\Bigl(b_{I',i',j}-\sum_{i\in I}b_{I,i,j}b_{I',i',i}\Bigr)v_{i'}=\bold{0}.
	 \end{split}
	\end{equation}
	On the other hand, by the normalization property of $b_{I,i,j}$,
	\begin{equation}\label{4:cov002}
	\begin{split}
	   \sum_{i'\in I'}\Bigl(b_{I',i',j}-\sum_{i\in I}b_{I,i,j}b_{I',i',i}\Bigr)=\sum_{i'\in I'}b_{I',i',j}-\sum_{i\in I}b_{I,i,j}\Bigl(\sum_{i'\in I'}b_{I',i',i}\Bigr)=1-\sum_{i\in I}b_{I,i,j}=0.
	 \end{split}
	\end{equation}
	Since $I'$ is a generating set of $X$, the vectors $v_{i'}-v_{i_{0}}, i'\in I'\backslash\{i_{0}\}$ are linearly independent for some $i_{0}\in I'$. It then follows from (\ref{4:cov001}) and (\ref{4:cov002}) that 
\begin{equation}\label{4:biterate}
	\begin{split}
	   b_{I',i',j}=\sum_{i\in I}b_{I,i,j}b_{I',i',i}
	 \end{split}
	\end{equation}
	for all $j\in J$. So
	\begin{equation}\nonumber
	\begin{split}
	&\quad L_{I,j}\circ\pi_{I}\circ\L_{I'}(\x)
	=\sum_{i\in I}b_{I,i,j}(L_{I',i}(\x))
	=\sum_{i\in I}b_{I,i,j}\sum_{i'\in I'}b_{I',i',i}x_{i'}
	\\&=\sum_{i'\in I'}\Bigl(\sum_{i\in I}b_{I,i,j}b_{I',i',i}\Bigr)x_{i'}
	=\sum_{i'\in I'}b_{I',i',j}x_{i'}
	=L_{I',j}(\x).
	\end{split}
	\end{equation}

We now prove that $\L_{I}((\V)^{k+1})=\L_{I'}((\V)^{k+1})$. Note that
$$\L_{I}((\V)^{k+1})=\L_{I'}\circ \pi_{I'}\circ \L_{I}((\V)^{k+1})\subseteq \L_{I'}((\V)^{k+1}).$$
Similarly, $\L_{I'}((\V)^{k+1})\subseteq \L_{I}((\V)^{k+1}).$ This implies that $\L_{I'}((\V)^{k+1})=\L_{I}((\V)^{k+1}).$
	\end{proof}

\subsection{$M$-spherical sets}\label{4:s:msms}

 In this paper,
 instead of working with the inner product $\vert\cdot\vert^{2}$, it is convenient to consider more general quadratic forms. 
Let $M\colon\V\to\F_{p}$ be a homogeneous quadratic form.  We say that a set $X\subseteq\F_{p}^{d}$ is \emph{$M$-spherical} if there exist $z\in \V$ and $r\in\F_{p}$ such that $M(x-z)=r$  for all $x\in X$. It is not hard to see that every simplex is $M$-spherical:

\begin{lem}[Simplices are $M$-spherical]\label{4:sis}
	Let $M\colon\V\to\F_{p}$ be a non-degenerate homogeneous quadratic form and $X=\{x_{i}\colon i\in J\}$ be a configuration in $\V$. Suppose that for some $i_{0}\in J$, the vectors $x_{i}-x_{i_{0}}, i\in J\backslash\{i_{0}\}$ are linearly independent. Then $X$ is $M$-spherical. 
\end{lem}
\begin{proof}
It suffices to show that there exists $z\in\V$ such that $M(x_{i}-z)=M(x_{i_{0}}-z)$ for all $i\in J\backslash\{i_{0}\}$. This is equivalent of solving the linear equation system  $((x_{j}-x_{j_{0}})A)\cdot z=b_{j}, j\in J\backslash\{j_{0}\}$, where $A$ is the matrix associated to $M$ and $b_{j}$ is some constant depending  on $M$ and $X$. Since $x_{j}-x_{j_{0}},j\in J\backslash\{j_{0}\}$ are linearly independent,  so are $(x_{j}-x_{j_{0}})A,j\in J\backslash\{j_{0}\}$ since $A$ is invertible. So such $z\in\V$ exists by the knowledge of linear algebra, which implies that $X$ is $M$-spherical.
\end{proof}

The following is a convenient criterion for $M$-spherical  sets.

\begin{lem}\label{4:lmp}
Let	$M\colon\V\to\F_{p}$ be a non-degenerate homogeneous  quadratic form. 
A configuration	$X=\{x_{j}\colon j\in J\}$ in $\V$ is $M$-spherical if and only if for any quadratic form $M'\colon\V\to\F_{p}$ whose associated matrix is the same as $M$, and any $c_{j}\in\F_{p}, j\in J$, if $\sum_{j\in J}c_{j}=0,$ $\sum_{j\in J}c_{j}x_{j}=\bold{0}$, then $\sum_{j\in J}c_{j}M'(x_{j})=0$.	
\end{lem}
\begin{proof}
	The proof is almost identical to Lemma 16 of \cite{LMP19} and we write it down for completeness. Assume that $M(n)=(nA)\cdot n$ and
	 that $M'(n)=(nA)\cdot n+u\cdot n+v$ for some $u\in\F_{p}^{d}$ and $v\in\F_{p}$, where $A$ is the matrix associated to $M$ and $M'$.
	If $X$ is $M$-spherical, then $M(x_{j}-z)=r$ for some $z\in\F_{p}^{d}$ and $r\in\F_{p}$ for all $j\in J$. So 
	$$M(x_{j})=r-M(z)+x_{j}\cdot (2zA) \text{ and } M'(x_{j})=r-M(z)+v+x_{j}\cdot (2zA+2u).$$
	Therefore,
	$$\sum_{j\in J}c_{j}M'(x_{j})=\sum_{j\in J}c_{j}(r-M(z)+v)+\sum_{j\in J}(c_{j}x_{j})\cdot(2zA+2u)=0.$$
	
Conversely, suppose that $X$ is not $M$-spherical. We wish to construct $c_{j}\in\F_{p},j\in J$ such that 	$\sum_{j\in J}c_{j}=0$, $\sum_{i=0}^{k}c_{i}x_{i}=\bold{0}$, but $\sum_{i=0}^{k}c_{i}M'(x_{i})\neq 0$. We may assume without loss of generality that $X$ is minimal in the sense that any proper subset of $X$ is $M$-spherical. Fix any $j_{0}\in J$.
Since $X$ is not $M$-spherical, by
 Lemma \ref{4:sis}, $x_{j}-x_{j_{0}},j\in J\backslash\{j_{0}\}$ are linearly dependent.
So we may assume that $\sum_{j\in J\backslash\{j_{0}\}}a_{j}(x_{j}-x_{j_{0}})=\bold{0}$ for some  $a_{j}\in\F_{p}$ not all equal to 0. Assume that $a_{j_{1}}\neq 0$ for some $j_{1}\in J\backslash\{j_{0}\}$.

Let	$X'$ denote the configuration $\{x_{j}\colon j\in J\backslash\{j_{1}\}\}$. By the minimality of $X$, $X'$ is $M$-spherical.
So $M(x_{j}-z)=r$ for some $z\in\F_{p}^{d}$ and $r\in\F_{p}$ for all $j\in J\backslash\{j_{1}\}$. Since $X$ is not  $M$-spherical, $M(x_{j_{1}}-z)-r=t\neq 0$. It is not hard to see that
	$$M(x_{j})-M(x_{j_{0}})=(x_{j}-x_{j_{0}})\cdot (2zA)$$
for all $j\in J\backslash\{j_{0},j_{1}\}$, and
    $$M(x_{j_{1}})-M(x_{j_{0}})=(x_{j_{1}}-x_{j_{0}})\cdot (2zA)+t.$$
    Since $\sum_{j\in J\backslash\{j_{0}\}}a_{j}(x_{j}-x_{j_{0}})=0$, we have that
    \begin{equation}\nonumber
    \begin{split}
  &\quad  \sum_{j\in J\backslash\{j_{0}\}}a_{j}(M'(x_{j})-M'(x_{j_{0}}))=\sum_{j\in J\backslash\{j_{0}\}}a_{j}(M(x_{j})-M(x_{j_{0}}))
  \\&=\sum_{j\in J\backslash\{j_{0}\}}a_{j}(x_{j}-x_{j_{0}})\cdot (2zA)+a_{j_{1}}t=a_{j_{1}}t\neq 0.
    \end{split}
    \end{equation}
We are done by taking $c_{j_{0}}=-\sum_{j\in J\backslash\{j_{0}\}}a_{j}$ and $c_{j}=a_{j}$ for all $j\in J\backslash\{j_{0}\}$.
 \end{proof}		

As a  corollary of Lemma \ref{4:lmp}, we have the following generalization of  Lemma 5 of \cite{LMP19}:
\begin{coro}\label{4:lmp2}
	Let	$M\colon\V\to\F_{p}$ be a non-degenerate homogeneous quadratic form and $X=\{x_{j}\colon j\in J\}$ be an $M$-spherical configuration in $\V$. Let $I$ be a generating set of $X$ with $L_{I,j},j\in J$ be its generating maps. If $M(y_{i}-z)=r$ for some $z,y_{i}\in\V$ and $r\in \F_{p}$ for all $i\in I$, then $M(L_{I,j}((y_{i})_{i\in I})-z)=r$ for all $j\in J$.
	
	In particular, if  $M(x_{i}-z)=r$ for some $z\in\V$ and $r\in \F_{p}$ for all $i\in I$, then $M(x_{j}-z)=r$ for all $j\in J$.
\end{coro}
\begin{proof}
	For convenience denote $M'(n):=M(n-z)-r$ and $y_{j}:=L_{I,j}((y_{i})_{i\in I})$ for all $j\in J\backslash I$. Recall that we also have $L_{I,j}((y_{i})_{i\in I})=y_{j}$ for all $j\in I$. Then
	   $M'(y_{i})=0$ for all $i\in I$.
	Let $b_{I,i,j}, i\in I, j\in J$ be the generating constants for $X$ with respect to $I$. Then for all $j\in J\backslash I$, $\sum_{i\in I}b_{I,i,j}-1=0$ and $\sum_{i\in I}b_{I,i,j}y_{i}-y_{j}=0$. 
	Since $X$ is $M$-spherical, 
	By Lemma \ref{4:lmp},  $\sum_{i\in I}b_{I,i,j}M'(y_{i})-M'(y_{j})=0$. So $M'(y_{j})=0$ for all $j\in J\backslash I$.	
	
	The ``in particular" part follows from the fact that $x_{j}:=L_{I,j}((x_{i})_{i\in I})$ for all $j\in J$.
\end{proof}	

\subsection{Different notions of congruency}

In order to prove Theorem \ref{4:mainmain}, instead of working directly on congruent copies of a configuration, we study relavent  patterns with less restrictions. 

\begin{defn}[$M$-isometry and $M$-congruency]
	Let $M\colon\V\to\F_{p}$ be a homogeneous quadratic form 
	and
	$X=\{x_{j}\colon j\in J\}$, $Y=\{y_{j}\colon j\in J\}$ be two configurations in $\V$. We say that $X$ and $Y$ are \emph{$M$-isometric} if $M(x_{i}-x_{j})=M(y_{i}-y_{j})$ for all $i,j\in J$.  
	We say that $Y$ is \emph{$M$-congruent} to $X$ if there exist an \emph{$M$-isometry} $U$ on $\V$ (i.e. $U\colon \V\to\V$ is a linear transformation such that $M(U(x))=M(x)$ for all $x\in\V$) and $z\in\V$ such that $y_{j}=z+U(x_{j})$ for all $j\in J$.

When $M=\vert\cdot\vert^{2}$,  we simply call $M$-isometry/congruency  as  \emph{isometry/congruency} for short.
\end{defn}

	It is not hard to see  that  congruency implies isometry. However, the next example shows that the converse is false.
	
	\begin{ex}
		Let $d=5$, $X=\{x_{1},\dots,x_{4}\}$ and  $Y=\{y_{1},\dots,y_{4}\}$ with $x_{1}=y_{1}=(-1,-1,0,0,$ $0)$, $x_{2}=y_{2}=(-1,1,0,0,0)$, $x_{3}=y_{3}=(1,-1,0,0,0)$, $x_{4}=(1,1,0,0,0)$, and $y_{4}=(1,1,a,b,c)$, where $a^{2}+b^{2}+c^{2}=0$ and $(a,b,c)\neq (0,0,0)$. Such $a,b,c$ exist because of Lemma \ref{4:counting01}. Then $X$ and $Y$ are isometric. It is clear that $X$ and $Y$ are spherical since $\vert x_{i}\vert^{2}=\vert y_{i}\vert^{2}=2$.
		\footnote{Moreover, $X$ is non-degenerate.} On the other hand, if there  exist an isometry $U$ on $\V$   and $z\in\V$ such that $y_{i}=z+U(x_{i})$ for all $1\leq i\leq 4$, then 
		\begin{equation}\nonumber
		\begin{split}
		&\quad (0,0,a,b,c)=y_{1}+y_{4}-y_{2}-y_{3}=(z+U(x_{1}))+(z+U(x_{4}))-(z+U(x_{2}))-(z+U(x_{3}))
		\\&=U(x_{1}+x_{4}-x_{2}-x_{3})=(0,0,0,0,0),
		\end{split}
		\end{equation}
		a contradiction. So $Y$ is not  congruent to $X$.
	\end{ex}

Next we introduce a variation of the $M$-congruency concept.	
\begin{defn}[almost $M$-congruency]
			Let $M\colon\V\to\F_{p}$ be a homogeneous quadratic form 
			and
			$X=\{x_{j}\colon j\in J\}$ and $Y=\{y_{j}\colon j\in J\}$ be two configurations in $\V$.
	Let $I\subseteq J$ be a generating set of $X$ with generating constants $b_{I,i,j}, i\in I, j\in J$.
	We say that $Y$ is \emph{almost $M$-congruent} to $X$ with respect to $I$ if 
	\begin{itemize}
		\item for any $i,i'\in I$, $M(x_{i}-x_{i'})=M(y_{i}-y_{i'})$ (i.e. $\{x_{i}\colon i\in I\}$ is $M$-isometric to $\{y_{i}\colon i\in I\}$);
		\item for any $i\in J\backslash I$, $y_{j}=\sum_{i\in I}b_{I,i,j}y_{i}$.\footnote{We also have $y_{j}=\sum_{i\in I}b_{I,i,j}y_{i}$ for $i\in I$ for free, since $b_{I,i,j}=\d_{i,j}$ for all $i,j\in I$.}
	\end{itemize}	
	We say that $Y$ is \emph{almost $M$-congruent} to $X$ if  $Y$ is  almost $M$-congruent  to $X$ with respect to every generating set of $X$. When $M=\vert\cdot\vert^{2}$,  we simply call almost $M$-congrency  as  \emph{almost congrency} for short.
\end{defn}

Roughly speaking, almost congruency is a mixture of an isometry property for a generating set of the configuration, and some linearity property for the remaining points. The concept of almost congruency was used implicitly in the work of Lyall, Magyar and Parshall \cite{LMP19}, and it is very useful in converting the Geometric Ramsey Conjecture to the study of linear patterns in spheres. We provide some connections between different notions of congruency for later uses.

\begin{prop}\label{4:cdd2}
	Let $M\colon\V\to\F_{p}$ be a non-degenerate homogeneous quadratic form,   $X=\{x_{j}\colon j\in J\}$ and $Y=\{y_{j}\colon j\in J\}$ be two configurations, and $I\subseteq J$ be a generating set $I$ of $X$.
	\begin{enumerate}[(i)]
		\item If $Y$ is $M$-congruent to $X$, then $Y$ is almost $M$-congruent and $M$-isometric to $X$.
		\item If $Y$ is almost $M$-congruent to $X$ with respect to $I$, and if  the vectors $y_{i}-y_{i_{0}}, i\in I\backslash\{i_{0}\}$ are linearly independent for some $i_{0}\in I$,\footnote{When $Y$ is almost $M$-congruent to $X$, this is equivalent of saying that $\sp_{\F_{p}}(X-X)$ and $\sp_{\F_{p}}(Y-Y)$ have the same dimension.} then
		  $Y$ is  $M$-congruent to $X$ (and thus is almost $M$-congruent to $X$ with respect to every generating set of $X$ by Part (i)). If in addition $X$ is $M$-spherical, then so is $Y$.
	\end{enumerate}	
\end{prop}	
\begin{proof}
	Part (i) is straightforward. 
	We now prove Part (ii). 
	Let $X'=\{x_{j}\colon j\in I\}$ and $Y'=\{y_{j}\colon j\in I\}$.
	By linearity, it suffices to show that $X'$ is $M$-congruent to $Y'$. For convenience assume without loss of generality that $I=\{0,\dots,k\}$ for some $k\in\N_{+},$ that $i_{0}=0$ and that $x_{0}=y_{0}=\bold{0}$. Then $(x_{1},\dots,x_{k})$ and $(y_{1},\dots,y_{k})$ are linearly independent tuples. Let $U=\sp_{\F_{p}}\{x_{1},\dots,x_{k}\}$ and  $V=\sp_{\F_{p}}\{y_{1},\dots,y_{k}\}$. Then $x_{i}\mapsto y_{i}, 1\leq i\leq k$ induces a bijective linear transformation $\phi\colon U\to V$. Let $A$ be the matrix associated to $M$.
	Since $Y$ is almost $M$-congruent to $X$, we have  that
	$M(y_{i}-y_{j})=M(x_{i}-x_{j})$ for all $1\leq i, j\leq k$ and that $M(y_{i})=M(x_{i})$ for all $1\leq i\leq k$ (since $x_{0}=y_{0}=\bold{0}$). From this it is not hard to see that $(x_{i}A)\cdot x_{j}=(y_{i}A)\cdot y_{j}$ for all $1\leq i, j\leq k$, which implies that 
	$$M(a_{1}x_{1}+\dots+a_{k}x_{k})=M(a_{1}y_{1}+\dots+a_{k}y_{k})$$
	for all $a_{1},\dots,a_{k}\in \F_{p}$. In other words, $\phi$ is a bijective $M$-isometry. Since $M$ is non-degenerate, by Witt's Extension Theorem (see for example \cite{Cla13} Corollary 24), $\phi$ extension to an $M$-isometry from $\V\to \V$. So $X'$ is  $M$-congruent to $Y'$, and thus $X$ is  $M$-congruent to $Y$.

We now assume further that $X$ is $M$-spherical and show that $Y$ is also $M$-spherical. Since $Y$ is almost $M$-congruent to $X$, by Corollary \ref{4:lmp2}, it suffices to show that $Y'$ is spherical, 
which follows from Lemma \ref{4:sis}.
\end{proof}	

\begin{rem}
A result similar to Proposition \ref{4:cdd2} was proven in Lemma 14 \cite{LMP19} under the additional assumption that $X$ is non-degenerate. In fact, the assumptions in both results  share the same purpose in ensuring
$\dim(\sp_{\F_{p}}(X-X))=\dim(\sp_{\F_{p}}(Y-Y))$, from which one can apply Witt's Extension Theorem to show the existence of the $M$-isometry $\phi$.
\end{rem}

\section{Some preliminary reductions of the density Ramsey theorem}\label{4:s:4rr}

\subsection{A variation of Theroem \ref{4:mainmain} over spheres}

In this section, we reduce Theorem \ref{4:mainmain} to the following result:

\begin{thm}[Density Ramsey Theorem on spheres]\label{4:mainmain2}
	Let $d\in\N_{+}$, $C,\e>0$, $p$ be a prime, $M\colon\V\to\F_{p}$ be a non-degenerate homogeneous quadratic form, and $X\subseteq\F_{p}^{d}$ be an $M$-spherical configuration of complexity at most $C$.  Let $k$ be the dimension of $\sp_{\F_{p}}(X-X)$ and suppose that $d\geq d_{0}(X)$. There exist $p_{0}=p_{0}(C,d,\e)\in\N$ and $\d=\d(C,d,\e)>0$ such that if $p>p_{0}$, then for every $x\in\V$, every $r\in\F_{p}$ and every set $E\subseteq V(M(\cdot+x)-r)$ (recall Section \ref{4:s:defn} for the definition) with $\vert E\vert>\e p^{d-1}$, $E$ contains at least $\d p^{(k+1)d-(k+1)(k+2)/2}$ almost $M$-congruent copies of $X$.    
\end{thm}

\begin{rem}
	Note that Theorem \ref{4:mainmain2} differs from Theorem \ref{4:mainmain} in the following aspects. Firstly, the set $E$ in Theorem \ref{4:mainmain2} is restricted to the ``sphere" $V(M(\cdot+x)-r)$. Secondly, the quantity $\d$ in  Theorem \ref{4:mainmain2} is allowed to be dependent on $d$ (and thus depends on $\vert X\vert$ implicitly since $d\geq d_{0}(X)$). Finally, Theorem \ref{4:mainmain2} deals with 
	$M$-spherical sets and almost $M$-congruent copies for a generate non-degenerate homogeneous quadratic form $M$.	
\end{rem}

In order to see the connection between Theorems \ref{4:mainmain} and \ref{4:mainmain2}, it is convenient to consider the following intermediate result:

\begin{prop}[Density Ramsey Theorem, a weaker form]\label{4:mainmain3}
	Let $d\in\N_{+}$, $C,\e>0$, $p$ be a prime, $M\colon\V\to\F_{p}$ be a non-degenerate homogeneous quadratic form, and $X\subseteq\F_{p}^{d}$ be an $M$-spherical configuration of complexity at most $C$. Let $k$ be the dimension of $\sp_{\F_{p}}(X-X)$ and suppose that $d\geq d_{0}(X)$.   There exist $p_{0}=p_{0}(C,d,\e)\in\N$ and $\d=\d(C,d,\e)>0$ such that if $p>p_{0}$,  then   every set $E\subseteq \V$ with $\vert E\vert>\e p^{d}$ contains at least $\d p^{(k+1)d-(k+1)k/2}$ $M$-congruent copies of $X$.   
\end{prop}

\begin{proof}[Proof of Proposition \ref{4:mainmain3} assuming Theorem \ref{4:mainmain2}]
	For convenience denote 	$$M_{x,r}(n):=M(n+x)-r.$$
	Let $s:=\vert X\vert-k\geq 1$ and $E\subseteq\V$ with $\vert E\vert>\e p^{d}$.
	Assume without loss of generality that $X=\{x_{0},\dots,x_{k+s-1}\}$ and $\{0,\dots,k\}$ is a generating set of $X$.
	Note that for any $n\in\V$, the number of tuples $(x,r)\in\F_{p}^{d+1}$ such that $M_{x,r}(n)=0$ is $p^{d}$ (since $r$ is uniquely determined by $x$ and $n$). So $\sum_{(x,r)\in\F_{p}^{d+1}}\vert E\cap V(M_{x,r})\vert=p^{d}\vert E\vert>\e p^{2d}.$ Since $\vert E\cap V(M_{x,r})\vert\leq \vert V(M_{x,r})\vert\leq 2p^{d-1}$, by the Pigeonhole Principle, there exists a subset $W$ of $\F_{p}^{d+1}$ of cardinality $\gg \e p^{d+1}$ such that $\vert E\cap V(M_{x,r})\vert\gg \e p^{d-1}$ for all $(x,r)\in W$. By Theorem \ref{4:mainmain2}, for all $(x,r)\in W$, the set $E\cap V(M_{x,r})$ contains $\gg_{C,d,\e}p^{(k+1)d-(k+1)(k+2)/2}$ almost $M$-congruent copies $Y=\{v_{0},\dots,v_{k+s-1}\}$ of $X$. 
	
	By Lemma \ref{4:iiddpp}, the number of choices of $v_{0},\dots,v_{k}\in\V$ such that $v_{1}-v_{0},\dots,v_{k}-v_{0}$ are linearly dependent is at most $kp^{d+(d+1)(k-1)}=kp^{dk+(k-1)}$. Since $d\geq \frac{(k+1)(k+2)}{2}+k$, the set $E\cap V(M_{x,r})$ contains $\gg_{C,d,\e}p^{(k+1)d-(k+1)(k+2)/2}$ almost $M$-congruent copies $Y=\{v_{0},\dots,v_{k+s-1}\}$ of $X$ such that $v_{1}-v_{0},\dots,v_{k}-v_{0}$ are linearly independent (because $v_{k+1},\dots,v_{k+s-1}$ are uniquely determined by $v_{0},\dots,v_{k}$).

	Fix any such $Y$. By Proposition \ref{4:cdd2} (ii), $Y$ is $M$-congruent to $X$. Since $\{0,\dots,k\}$ is a generating set of $X$, by Corollary \ref{4:lmp2}, 
	 $Y\subseteq V(M_{x,r})$ if and only if 
	  $v_{0},\dots,v_{k}$ belong to  $V(M_{x,r})$, or equivalently, $x\in V(M_{v_{0},r})^{v_{1}-v_{0},\dots,v_{k}-v_{0}}$ (recall Section \ref{4:s:defn} for the definition). Since $v_{1}-v_{0},\dots,v_{k}-v_{0}$ are linearly independent, 
  and $d\geq 2k+3$, by Lemma \ref{4:counting02}, $\vert V(M_{v_{0},r})^{v_{0}-v_{1},\dots,v_{0}-v_{k}}\vert=p^{d-k-1}(1+O(p^{-1/2}))$.
	Since there are $p$ choices of $r$, the number of $(x,r)\in\F_{p}^{d+1}$ such that $Y\subseteq V(M_{x,r})$ is at most $2p^{(d+1)-(k+1)}$ if $p\gg_{d} 1$.
	
	So the total number of $M$-congruent copies of $X$ is no less than the cardinality of the set $E=\cup_{(x,r)\in\F_{p}^{d+1}}(E\cap V(M_{x,r}))$, which is at least $$\gg_{C,d,\e}p^{(k+1)d-(k+1)(k+2)/2}\cdot\vert W\vert/2p^{(d+1)-(k+1)}\gg_{C,d,\e}p^{(k+1)d-(k+1)k/2}.$$ 
\end{proof}

\begin{proof}[Proof of Theorem \ref{4:mainmain} assuming Proposition \ref{4:mainmain3}]
	Let $X=\{x_{0},\dots,x_{k+s-1}\}\subseteq\F_{p}^{d}$ be a spherical configuration of complexity at most $C$, where   $\dim(\sp_{\F_{p}}(X-X))=k$ and $s=\vert X\vert -k\geq 1$. 
	Assume without loss of generality that $I:=\{0,\dots,k\}$ is a generating set of $X$ and let $b_{I,i,j}, i\in I, j\in\{0,\dots,k+s-1\}$ be its generating constants.
	Let $r\in\F_{p}$, $x\in\V$, and $E\subseteq\V$ with $\vert E\vert>\e p^{d}$.
	Suppose that $d\geq d_{0}:=d_{0}(X)$. Throughout the proof we assume that  $p\gg_{C,d_{0},\e} 1$.

    Our goal is to show that the dependences of $p_{0}$ and $\d$ on $d$ in Proposition \ref{4:mainmain3} can be replaced by the dependences on $k$ and $s$. Let $M\colon \V\to\F_{p}$ be the quadratic form given by $M(x):=x\cdot x$. We say that a basis $v_{1},\dots,v_{d}$   of $\V$ is \emph{good} if $\sp_{\F_{p}}\{v_{1},\dots,v_{d_{0}}\}$ is not $M$-isotropic (see Appendix \ref{4:s:AppB1} for the definition).
	Fix any good basis $v_{1},\dots,v_{d}$   of $\V$. Denote $V:=\sp_{\F_{p}}\{v_{1},\dots,v_{d_{0}}\}$ and $U:=\sp_{\F_{p}}\{v_{d_{0}+1},\dots,v_{d}\}$. Then $M\vert_{V+c}$ is of full rank for all $c\in U$. By the Pigeonhole Principle, there exists a subset $U'(v_{1},\dots,v_{d})$ of $U$ of cardinality at least $\e p^{d-d_{0}}/2$ such that for all $c\in U'(v_{1},\dots,v_{d})$, we have that $\vert E\cap (V+c)\vert>\e p^{d_{0}}/2$.

	Fix such an element $c$ and
	let $\phi\colon\F_{p}^{d_{0}}\to V$ be any bijective linear transformation. Denote $M'=M\circ\phi$.  Then $M'\colon \F_{p}^{d_{0}}\to\F_{p}$ is a non-degenerate homogeneous quadratic form. Let $E':=\{m\in \F_{p}^{d_{0}}\colon \phi(m)+c\in E\}$. Then $\vert E'\vert>\e p^{d_{0}}/2$.

Let $W$ be the set of $(y_{0},\dots,y_{k})\in(\F_{p}^{d_{0}})^{k+1}$ such that $M'(y_{i}-y_{i'})=M(x_{i}-x_{i'})$ for all $0\leq i,i'\leq k$. Using Proposition \ref{4:yy33}, it is not hard to check that $W$ is a consistent $M'$-set 
  of total co-dimension $k(k+1)/2$ (see Appendix \ref{4:s:AppB3} for definitions). 
 By Theorem \ref{4:ct}, since $\rank(M')=d_{0}\geq k^{2}+k+1$, we have that $\vert W\vert=p^{d_{0}(k+1)-k(k+1)/2}(1+O_{k}(p^{-1/2}))$. On the other hand, by Lemma \ref{4:iiddpp}, the number of $(y_{0},\dots,y_{k})\in(\F_{p}^{d_{0}})^{k+1}$ such that $y_{1}-y_{0},\dots,y_{k}-y_{0}$ are linearly dependent is at most $kp^{d_{0}+(d_{0}+1)(k-1)}=kp^{d_{0}k+(k-1)}$. Since $d_{0}\geq \frac{k(k+1)}{2}+k$ and $p\gg_{d_{0},k} 1$, there exists some configuration $Y=\{y_{0},\dots,y_{k+s-1}\}$ in $\F_{p}^{d'}$ such that  $M'(y_{i}-y_{i'})=M(x_{i}-x_{i'})$ for all $0\leq i,i'\leq k$, that 
$y_{j}=\sum_{i=0}^{k}b_{I,i,j}y_{i}$ for all $j\in\{0,\dots,k+s-1\}$,
and  that $y_{1}-y_{0},\dots,y_{k}-y_{0}$ are linearly independent. Therefore, $I$ is a generating set of $Y$. Since $\phi$ is injective,  $I$ is also a generating set of $\phi(Y)+c$.

Note that $$M((\phi(y_{i})+c)-(\phi(y_{i'})+c))=M'(y_{i}-y_{i'})=M(x_{i}-x_{i'})$$
for all $i,i'\in I$ and that 
$$\phi(y_{j})+c=\sum_{i=0}^{k}b_{I,i,j}(\phi(y_{i})+c)$$
 	for all $j\in \{0,\dots,k+s-1\}$. We have that $\phi(Y)+c$ is an almost $M$-congruent copy of $X$ with respect to $I$.   By Proposition \ref{4:cdd2} (ii), $\phi(Y)+c$ is $M$-spherical.
By Lemma \ref{4:sis}, $\{y_{i}\colon i\in I\}$ is $M'$-spherical. Assume that $M'(y_{i}-z)=r$ for some $z\in\F_{p}^{d_{0}}$ and $r\in\F_{p}$ for all $i\in I$. 
Then $M((\phi(y_{i})+c)-(\phi(z)+c))=r$ for all $i\in I$. Since $I$ is a  generating set of the $M$-spherical configuration $\phi(Y)+c$, by Corollary \ref{4:lmp2}, $M((\phi(y_{i})+c)-(\phi(z)+c))=r$  and thus $M'(y_{j}-z)=r$ for all $j\in \{0,\dots,k+s-1\}$. So 
 $Y$ is $M'$-spherical.

	Since $p\gg_{C,d_{0},\e} 1$,by Theorem \ref{4:mainmain3},  $E'$ contains at least $\d p^{(k+1)d_{0}-(k+1)k/2}$    $M'$-congruent copies of $Y$, where $\d:=\d(C,d_{0},\e)>0$. Composing with $\phi$, it is not hard to see that $E\cap (V+c)$  contains at least $\d p^{(k+1)d_{0}-(k+1)k/2}$  $M$-congruent copies of $\phi(Y)+c$.  
 	Ranging over all $c\in U'(v_{1},\dots,v_{d})$, we have that $E$  contains at least $\d\e p^{(d-d_{0})+(k+1)d_{0}-(k+1)k/2}/4$ $M$-congruent copies of $X$ with respect to any good basis $v_{1},\dots,v_{d}$ of $\V$. 
	
	Let $B$ be the $d\times d$ matrix with $v_{1},\dots,v_{d}$ being its row vectors, and $B'$ be the $d_{0}\times d$ matrix with $v_{1},\dots,v_{d_{0}}$ being its row vectors. Then $v_{1},\dots,v_{d}$ is a good basis of $\V$ if any only if $\det(B),\det(B'{B'}^{T})\neq 0$. By Lemma \ref{4:ns}, the number of good basis of $\V$ is $p^{d^{2}}(1+O_{d}(p^{-1}))$.
	Therefore, letting $R$ denote the set of tuples $(X',v_{1},\dots,v_{d},c)$ with  $X'$ being  $M$-congruent to $X$, $v_{1},\dots,v_{d}$ being a good basis of $\V$, and $c\in U'(v_{1},\dots,v_{d})$, we have that 
	\begin{equation}\label{4:cdowef}
	    \vert R\vert\geq \d\e p^{d^{2}+(d-d_{0})+(k+1)d_{0}-(k+1)k/2}/8.
	\end{equation}

	\textbf{Claim.}
	For any $1\leq t\leq d_{0}$ and $u_{0},\dots,u_{t}\in\V$ with $u_{1}-u_{0},\dots,u_{t}-u_{0}$ being linearly independent, the number of tuples $(v_{1},\dots,v_{d_{0}})\in(\V)^{d_{0}}$ such that $u_{0},\dots,u_{t}\in \sp_{\F_{p}}\{v_{1},\dots,v_{d_{0}}\}$ is at most $p^{d_{0}t+d(d_{0}-t)}$.
	
	Denote 	$u'_{i}:=u_{i}-u_{0}$ for $1\leq i\leq t$.
	Note that $u_{0},\dots,u_{t}\in \sp_{\F_{p}}\{v_{1},\dots,v_{d_{0}}\}$ only if $u'_{1},\dots,u'_{t}\in \sp_{\F_{p}}\{v_{1},\dots,v_{d_{0}}\}$. Assume that
	$u'_{i}=\sum_{j=1}^{d_{0}}c_{i,j}v_{j}$ for all $1\leq i\leq t$. There are at most $p^{td_{0}}$ choices of $c_{i,j}$. For any choice of $c_{i,j}$, to achieve $u'_{i}=\sum_{j=1}^{d_{0}}c_{i,j}v_{j}, 1\leq i\leq t$, we must have that $(c_{i,1},\dots,c_{i,d_{0}}), 1\leq i\leq t$ are linearly independent. If this is the case, then there exists a subset $I$ of $\{1,\dots,d_{0}\}$ with $\vert I\vert=t$ such that $v_{j}, j\in I$ are uniquely determined by $u_{i}, v_{j'}, c_{i,j}, 1\leq i\leq t, 1\leq j\leq d_{0}, j'\notin I$. Since there are $p^{d(d_{0}-t)}$ choices of $v_{j'}, j'\notin I$, we conclude that the   number of tuples $(v_{1},\dots,v_{d_{0}})$ such that $u'_{1},\dots,u'_{t}\in \sp_{\F_{p}}\{v_{1},\dots,v_{d_{0}}\}$ is at most $p^{d_{0}t+d(d_{0}-t)}$. This proves the claim.
	
	\

	Fix any  $M$-congruent copy $X'$ of $X$ and any $c\in\V$. 
	Since $\vert X'\vert=k+1$ and $$\dim(\sp_{\F_{p}}((X'-c)-(X'-c)))=\dim(\sp_{\F_{p}}(X'-X'))=\dim(\sp_{\F_{p}}(X-X))=k,$$ by the claim, the number of tuples $(v_{1},\dots,v_{d_{0}})$ such that $X'-c\subseteq\sp_{\F_{p}}\{v_{1},\dots,v_{d_{0}}\}$ is at most $p^{d_{0}k+d(d_{0}-k)}$.
	So each $X'$ appears at most $p^{d(d-d_{0})+d_{0}k+d(d_{0}-k)}$ times in the set $R$.
So it follows from (\ref{4:cdowef}) 
	  that $E$ contains in total at least $$\d\e p^{d^{2}+(d-d_{0})+(k+1)d_{0}-(k+1)k/2-(d(d-d_{0})+d_{0}k+d(d_{0}-k))}/8=\d\e p^{(k+1)d-(k+1)k/2}/8$$  $M$-congruent (and thus congruent) copies of $X$. 	  This completes the proof.
\end{proof}	

In conclusion, in order to prove Theorem \ref{4:mainmain}, it suffices to prove Theorem \ref{4:mainmain2}.

\subsection{Definitions and conventions} 
The rest of the paper is devoted to proving Theorem \ref{4:mainmain2}.  
We start with some simplifications and definitions. 

In the rest of the paper,
  we fix  the non-degenerate homogeneous quadratic form $M$, the matrix $A$ associated to it, and the radius $r\in\F_{p}$.
We fix an $M$-spherical configuration $X\subseteq\F_{p}^{d}$
 and denote $k:=\dim(\sp_{\F_{p}}(X-X)), s:=\vert X\vert-k\geq 1$. 
 
 In the rest of the paper, we label the coordinates of vectors in the sets $(\F_{p}^{t})^{k+1}$ and $(\Z^{t})^{k+1}$ as $0\sim k$ instead of  $1\sim (k+1)$, and  label the coordinates of vectors in the sets $(\F_{p}^{t})^{k+s}$ and $(\Z^{t})^{k+s}$ as $0\sim (k+s-1)$ instead of  $1\sim (k+s)$.
 Under this convention, 
from now on we assume that $J=\{0,\dots,k+s-1\}$ is the index set for $X$ with respect to the canonical ordering.
For each generating set $I$ of $X$,  let $\L_{I}=(L_{I,0},\dots,L_{I,k+s-1})$ be the generating map of $X$ with respect to $I$ with generating constants $b_{I,i,j}, i\in I, j\in J$, and let $\pi_{I}\colon (\V)^{\vert J\vert}\to (\V)^{\vert I\vert}$ be the projection onto the coordinates with $I$-indices.

Replacing $X$ with $X-u_{0}$ if necessary, we may further assume that 
$$X=\{v_{0},v_{1},\dots,v_{k+s-1}\}$$
with $M(v_{0})=\dots=M(v_{k+s-1})=\l$ for some $\l\in\F_{p}$. 
Note that if $E\subseteq V(M(\cdot+x)-r)$, then $E+x\subseteq V(M-r)$. Since the number of almost $M$-congruent copies of $X$ in $E$ is the same as that of $E+r$, we may assume without loss of generality that $x=\bold{0}$ and $E\subseteq V(M-r)$.

 In the rest of the paper, for any $I\subseteq L$, we use $\Omega_{I}$ to denote the set of $(x_{i})_{i\in I}\in V(M-r)^{\vert I\vert}$ such that $M(x_{i}-x_{i'})=M(v_{i}-v_{i'})$ for all $i,i'\in I$. Also denote
 $\Omega:=\L_{I}(\Omega_{I})$ for some generating $I$ of $X$.
We summarize some properties for the sets $\Omega_{I}$ and $\Omega$ (see Appendix \ref{4:s:AppB} for relevant definitions).

\begin{lem}\label{4:kk2d}
Let $I,I'\subseteq \{0,\dots,k+s-1\}$ be generating sets of $X$ and $\mathcal{E}$ be the set of $M$-congruent copies of $X$ lying in $V(M-r)$. 
\begin{enumerate}[(i)]
	\item We have that $\L_{I}(\Omega_{I})=\L_{I'}(\Omega_{I'})$. In particular, the set $\Omega$ is independent of the choice of the generating set, and is equal to the set of all almost $M$-congruent copies of $X$ (with respect to every generating set of $X$). 
	\item If $d\geq (k+1)(k+2)+1,$ then $\Omega_{I}$  is a nice and consistent $M$-set of total co-dimension $(k+1)(k+2)/2$, and we have that $\vert \Omega\vert=\vert \Omega_{I}\vert =p^{d(k+1)-(k+1)(k+2)/2}(1+O_{k}(p^{-1/2}))$.
	\item If $d\geq \frac{(k+1)(k+2)}{2}+k,$ then have that $\mathcal{E}\subseteq \Omega$ with $\vert\Omega\backslash \mathcal{E}\vert=p^{d(k+1)-(k+1)(k+2)/2}O_{k}(p^{-1/2})$.
	\item If $d\geq \max\{2k^{2}+6k+5, 2k+s+13\}$, then for all $C'>0$, there exists $K=O_{C',d}(1)$ such that  $\Omega_{I}$  admits a partially periodic $(t,Kt^{-K})$-Leibman dichotomy up to step $s$ and complexity at most $C'$ for all $0<t<1/2$.
\end{enumerate}	
\end{lem}
\begin{proof}
We first prove Part (i).
Pick any $\x=(x_{i})_{i\in I'}\in \Omega_{I'}$. For all $i,i'\in I'$,
we have $M(x_{i}-x_{i'})=M(v_{i}-v_{i'})$ for all $i,i'\in I'$. 
Using the identity 
	$$2((x-y)A)\cdot (z-w)=M(x-w)+M(y-z)-M(x-z)-M(y-w),$$
	we have that 
$$((x_{i_{1}}-x_{i_{2}})A)\cdot (x_{i_{3}}-x_{i_{4}})=((v_{i_{1}}-v_{i_{2}})A)\cdot (v_{i_{3}}-v_{i_{4}})$$
for all $i_{1},i_{2},i_{3},i_{4}\in I'$. So for any $a_{i}\in\F_{p}, i\in I'$ with $\sum_{i\in I'}a_{i}=0$, we have that
$$M\Bigl(\sum_{i\in I'}a_{i}x_{i}\Bigr)=M\Bigl(\sum_{i\in I'\backslash\{i_{0}\}}a_{i}(x_{i}-x_{i_{0}})\Bigr)=M\Bigl(\sum_{i\in I'\backslash\{i_{0}\}}a_{i}(v_{i}-v_{i_{0}})\Bigr)=M\Bigl(\sum_{i\in I'}a_{i}v_{i}\Bigr),$$
where $i_{0}$ is any element in $I'$. 
Therefore, for all $i,i'\in I$, we have that
$$M(\L_{I',i}(\x)-\L_{I',i'}(\x))=M\Bigl(\sum_{i''\in I'}(b_{I',i'',i}-b_{I',i'',i'})x_{i''}\Bigr)=M\Bigl(\sum_{i''\in I'}(b_{I',i'',i}-b_{I',i'',i'})v_{i''}\Bigr)=M(v_{i}-v_{i'}).$$

On the other hand, since $M(x_{i})=r$ for all $i\in I'$ and since $X$ is spherical, it follows from Corollary \ref{4:lmp2} that $M(L_{I',i}(\x))=r$ for all $i\in I$. In conclusion, we have that $\pi_{I}(\L_{I'}(\x))$ belongs to $\Omega_{I}$. Therefore, by Proposition \ref{4:kk4k}, $\L_{I'}(\x)=\L_{I}(\pi_{I}(\L_{I'}(\x)))\subseteq \L_{I}(\Omega_{I})$ and so $\L_{I'}(\Omega_{I'})\subseteq \L_{I}(\Omega_{I})$. Similarly, we have that $\L_{I}(\Omega_{I})\subseteq \L_{I'}(\Omega_{I'})$ and so $\L_{I}(\Omega_{I})=\L_{I'}(\Omega_{I'})$.

	 We now prove Part (ii). 
	 Assume without loss of generality $I=\{0,\dots,k\}$.
	 By part (i), $\Omega=\L_{I}(\Omega_{I})$. Since the projection of $\L_{I}$ onto the coordinates with $I$-indices is the identity map on $(\V)^{k+1}$, we have that   $\vert \Omega\vert=\vert \Omega_{I}\vert$.  	
	 
	 Note that $\Omega_{I}$ is the set $\Omega_{3}$ in Example 
	 B.4 of \cite{SunA} (up to a relabelling of the variables).    
	 So $\Omega_{I}$ 	 is a nice and consistent $M$-set of total co-dimension $(k+1)(k+2)/2$.
	  Since $d\geq (k+1)(k+2)+1,$ by Theorem \ref{4:ct},  $\vert \Omega_{I}\vert=p^{d(k+1)-\frac{(k+1)(k+2)}{2}}(1+O_{k}(p^{-1/2}))$.

	  We then prove Part (iii). Since we have shown that $\Omega$ is the set of all almost $M$-congruent copies of $X$ in Part (i), it follows from Proposition \ref{4:cdd2} (i) that $\mathcal{E}\subseteq \Omega$. Moreover, by 
Proposition \ref{4:cdd2} (ii), the set $\Omega\backslash \mathcal{E}$ is a subset of $(x_{0},\dots,x_{k+s-1})\in(\V)^{k+s}$ with $x_{i}-x_{i_{0}}, i\in I\backslash\{i_{0}\}$ being linearly independent for any prefixed generating set $I$ and $i_{0}\in I$. So one can use Lemma \ref{4:iiddpp} to see that 	  $\vert\Omega\backslash \mathcal{E}\vert\leq kp^{dk+k-1}$. This completes the proof since $d\geq \frac{(k+1)(k+2)}{2}+k$.

Part (iv) follows from  Theorem \ref{4:veryr} and Part (ii).
\end{proof}	

We conclude this section with a corollary of Proposition \ref{4:kk2d}.

\begin{coro}\label{4:swy}
    For any generating sets $I$ and $I'$ of $X$, we have that $(\pi_{I}\circ L_{I'})^{-1}(\Omega_{I})=\Omega_{I'}$.
\end{coro}
\begin{proof}
Since
	$b_{I',i,j}=\delta_{i,j}$ for all $i,j\in I'$, we have that $\pi_{I'}\circ L_{I'}$ is the identity map on $(\V)^{k+1}$. Similarly $\pi_{I}\circ L_{I}$ is also the identity map.
	So 
	for any $\x'\in(\V)^{k+1}$ with $\pi_{I}\circ L_{I'}(\x')\in \Omega_{I}$,  by Proposition \ref{4:kk4k} and Proposition \ref{4:kk2d} (i), we have that 
	$$\x'=\pi_{I'}\circ L_{I'}(\x')=\pi_{I'}\circ L_{I}\circ\pi_{I}\circ L_{I'}(\x')\in \pi_{I'}\circ L_{I}(\Omega_{I})=\pi_{I'}\circ L_{I'}(\Omega_{I'})=\Omega_{I'}.$$
	Conversely, for any $\x'\in\Omega_{I'}$, by  Proposition \ref{4:kk2d} (i), we have that
	$$\pi_{I}\circ L_{I'}(\x)\in \pi_{I}\circ L_{I'}(\Omega_{I'})=\pi_{I}\circ L_{I}(\Omega_{I})=\Omega_{I}.$$
	In conclusion, we have that $(\pi_{I}\circ L_{I'})^{-1}(\Omega_{I})=\Omega_{I'}$.
\end{proof}

\section{Bounding spherical averages by spherical Gowers norms}\label{4:s:5rr}	

Theorem \ref{4:mainmain2} is essentially equivalent to the positivity of the following spherical multiple ergodic average:
\begin{equation}\label{4:Vdc0}
	\begin{split}
	\E_{(x_{0},\dots,x_{k+s-1})\in \Omega}f_{0}(x_{0})f_{1}(x_{1})\dots f_{k+s-1}(x_{k+s-1}),
	\end{split}
	\end{equation}	
with $f_{0}=\dots=f_{k+s-1}$ being the indicator function of $E$. This can be viewed as a variation of the by now conventional multiple ergodic average, where the set $\Omega$ is the set of solutions to a  linear equation system instead of a quadratic one. It turns out that (\ref{4:Vdc0}) is connected to the spherical Gowers norms of  $f_{0},\dots,f_{k+s-1}$, which we recall below:

\begin{defn}[Shperical Gowers norms]
	Let  $s\in\N_{+}$,
	and $f\colon\V\to \C$ be a function. The \emph{$s$-th $V(M-r)$-Gowers norm} of $f$ is defined by the quantity
	$$\Vert f\Vert_{U^{s}(V(M-r))}:=\Bigl\vert\E_{(n,h_{1},\dots,h_{s})\in \Gow_{s}(V(M-r))}\prod_{\e=(\e_{1},\dots,\e_{s})\in\{0,1\}^{s}} \mathcal{C}^{\vert\e\vert}f(n+\e_{1}h_{1}+\dots+\e_{s}h_{s})\Bigr\vert^{\frac{1}{2^{s}}},$$
	where 
	$\Gow_{s}(V(M-r))$ is the $s$-th Gowers set of $V(M-r)$ defined in Section \ref{4:s:defn}.\footnote{One can show that $\Vert\cdot\Vert_{U^{s}(V(M-r))}$ is indeed a norm when $s\geq 2$.}
\end{defn}

 To prove Theorem \ref{4:mainmain2},
 our first step is to show that (\ref{4:Vdc0}) can be bounded by the spherical Gowers norms of $f_{i}$. 

\begin{thm}[Spherical generalized von-Neumann inequality]\label{4:vdcc}
	If $d\geq (k+s+2)(k+s+1)+1$, then for any functions $f_{0},\dots,f_{k+s-1}\colon V(M-r)\to\mathbb{C}$ bounded by 1, we have 
	\begin{equation}\label{4:Vdc1}
	\begin{split}
	\vert\E_{(x_{0},\dots,x_{k+s-1})\in \Omega}f_{0}(x_{0})f_{1}(x_{1})\dots f_{k+s-1}(x_{k+s-1})\vert
	\leq\Vert f_{0}\Vert_{U^{s}(V(M-r))}+O_{k,s}(p^{-O_{s}(1)^{-1}}).
	\end{split}
	\end{equation}	
\end{thm}

The rest of this section is devoted to the proof of Theorem \ref{4:vdcc}. A special case of Theorem \ref{4:vdcc} for $s=2$ was proved in Proposition 6 of \cite{LMP19}, which inspires the proof of Theorem \ref{4:vdcc} in this paper.

Let $I_{0}\subseteq \{0,\dots,k+s-1\}$ be any generating set of $X$ which contains $0$. By symmetry, we may assume without loss of generality that $I_{0}=\{0,1,\dots,k\}$.
 By Lemma \ref{4:kk2d} (i), the left hand side of (\ref{4:Vdc1}) equals to 
\begin{equation}\label{4:Vdc01}
\begin{split}
\Bigl\vert\E_{\x=(x_{0},\dots,x_{k})\in \Omega_{I_{0}}}f_{0}(x_{0})\dots f_{k}(x_{k}) \prod_{i=1}^{s-1}f_{k+i}(L_{I_{0},k+i}(\x))\Bigr\vert. 
\end{split}
\end{equation}
By Lemma \ref{4:kk2d} (ii) and Theorem \ref{4:ct}, since $d\geq  (k+2)(k+1)+1$, expression (\ref{4:Vdc01}) can be bounded by
\begin{equation}\label{4:vc1}
\begin{split}
&\quad \Bigl\vert\E_{\bold{y}=(x_{1},\dots,x_{k})\in \Omega_{I_{0}\backslash\{0\}}}\prod_{i=1}^{k}f_{i}(x_{i})
\E_{x_{0}\in V(M-r)\cap U_{\bold{y}}}f_{0}(x_{0}) \prod_{i=1}^{s-1}f_{k+i}(L_{I_{0},k+i}(x_{0},\bold{y}))\Bigr\vert+O_{k}(p^{-1/2})
\\&\leq \E_{\bold{y}\in \Omega_{I_{0}\backslash\{0\}}}\Bigl\vert 
\E_{x_{0}\in V(M-r)\cap U_{\bold{y}}}f_{0}(x_{0}) \prod_{i=1}^{s-1}f_{k+i}(L_{I_{0},k+i}(x_{0},\bold{y}))\Bigr\vert+O_{k}(p^{-1/2}),
\end{split}
\end{equation}
where $U_{\bold{y}}$ is the set of $x\in\V$ such that $M(x-x_{i})=M(v_{0}-v_{i})$ for all $1\leq i\leq k$.

We need to introduce some notations. For any subset $I$  of $\{0,\dots,k+s-1\}$ ,  $i_{0}\in I$ and $t\in\N_{+}$,
 let $\Omega^{i_{0}}_{I,t}$ be the set of all $(\bold{x}=(x_{i})_{i\in I},\bold{h}=(h_{1},\dots,h_{t}))\in (\V)^{\vert I\vert+t}$ such that for all $\e_{1},\dots,\e_{t}\in\{0,1\}$ we have that $(\tilde{x}_{i})_{i\in I}\in\Omega_{I}$, where $\tilde{x}_{i}:=x_{i}$ for $i\in I\backslash\{i_{0}\}$ and $\tilde{x}_{i_{0}}:=x_{i_{0}}+\e_{1}h_{1}+\dots+\e_{t}h_{t}$.
It is not hard to see that $\Omega_{I}$ is a pure and consistent $M$-set. So
by  Propositions \ref{4:yy3} and \ref{4:yy33} (vii), 
\begin{equation}\label{4:ostdim}
\text{$\Omega^{i_{0}}_{I,t}$ is a pure and consistent $M$-set of  total co-dimension at most $\binom{\vert I\vert+t+1}{2}$.}
\end{equation}

Denote $$\Delta_{h}f(L(x_{0},\bold{y})):=f(L(x_{0}+h,\bold{y}))\overline{f}(L(x_{0},\bold{y})).$$
Since $d\geq (k+3)(k+2)+1$,
by the Cauchy-Schwartz inequality, (\ref{4:ostdim}) and Theorem \ref{4:ct}, the square of the right hand side of (\ref{4:vc1}) is bounded by
\begin{equation}\label{4:vc2}
\begin{split}
&\quad
\E_{\bold{y}\in \Omega_{I_{0}\backslash\{0\}}}\Bigl\vert 
\E_{x_{0}\in V(M-r)\cap U_{\bold{y}}}f_{0}(x_{0}) \prod_{i=1}^{s-1}f_{k+i}(L_{I_{0},k+i}(x_{0},\bold{y}))\Bigr\vert^{2}+O_{k}(p^{-1})
\\&=\E_{\bold{y}\in \Omega_{I_{0}\backslash\{0\}}} 
\E_{(x_{0},h)\in \Gow_{1}(V(M-r)\cap U_{\bold{y}})}\Delta_{h}f_{0}(x_{0}) \prod_{i=1}^{s-1}\Delta_{h}f_{k+i}(L_{I_{0},k+i}(x_{0},\bold{y}))+O_{k}(p^{-1})
\\&=\E_{(\bold{x}=(x_{i})_{i\in I_{0}},h)\in \Omega^{0}_{I_{0},1}} 
\Delta_{h}f_{0}(x_{0}) \prod_{i=1}^{s-1}\Delta_{h}f_{k+i}(L_{I_{0},k+i}(\bold{x}))+O_{k}(p^{-1}).
\end{split}
\end{equation}

For any generating set $I$ of $X$ with $0\in I$, any $j\in\N$, any $t\in\N_{+}$ and any $\bold{h}=(h_{1},\dots,h_{t})\in(\V)^{t}$, denote
$$\Delta_{\bold{h}}f(L_{I,j}(x_{0},\bold{y}))=\Delta_{h_{t}}\dots\Delta_{h_{1}}f(L_{I,j}(x_{0},\bold{y}))=\prod_{\e_{1},\dots,\e_{t}\in\{0,1\}}\mathcal{C}^{\e_{1}+\dots+\e_{t}+t}f(L_{I,j}(x_{0}+\e_{1}h_{1}+\dots+\e_{t}h_{t},\bold{y})),$$
i.e. the derivative $\Delta_{\bold{h}}$ is taken with respect to the $0$-th variable of the function $f\circ L_{I,j}$. We need to prove the following intermediate estimate:

\begin{prop}\label{4:indvdc}
	Let $1\leq t\leq s-1$ and $I\subseteq \{0,\dots,k\}\cup\{k+t,\dots,k+s-1\}$ be a generating set of $X$ with $0\in I$. If $d\geq (k+s+2)(k+s+1)+1$, then there exists a 	generating set $I'\subseteq \{0,\dots,k\}\cup\{k+t-1,\dots,k+s-1\}$ of $X$ with $0\in I$ such that
	\begin{equation}\label{4:Vdc33}
	\begin{split}
	&\quad\Bigl\vert\E_{(\bold{x}=(x_{i})_{i\in I},\bold{h})\in \Omega^{0}_{I,s-t}} 
	\Delta_{\bold{h}}f_{0}(x_{0}) \prod_{i=1}^{t}\Delta_{\bold{h}}f_{k+i}(L_{I,k+i}(\bold{x}))\Bigr\vert^{2}
	\\&\leq\Bigl\vert\E_{(\bold{x}=(x_{i})_{i\in I'},\bold{h})\in \Omega^{0}_{I',s-t+1}} 
	\Delta_{\bold{h}}f_{0}(x_{0}) \prod_{i=1}^{t-1}\Delta_{\bold{h}}f_{k+i}(L_{I',k+i}(\bold{x}))\Bigr\vert+O_{k,s}(p^{-1/2}).
	\end{split}
	\end{equation}
\end{prop}	

We first explain how Proposition \ref{4:indvdc} implies Theorem \ref{4:vdcc}. 
Since (\ref{4:vc2}) holds, by repeatedly using Proposition \ref{4:indvdc} and the Cauchy-Schwartz inequality, we have that
there exists a generating set $I$ of $X$ with $0\in I$ such that the left hand side of (\ref{4:Vdc1}) is bounded by
\begin{equation}\nonumber
\begin{split}
 \vert\E_{((x_{i})_{i\in I},\bold{h})\in\Omega^{0}_{I,s}}\Delta_{\bold{h}}f_{0}(x_{0})\vert^{\frac{1}{2^{s}}}+O_{k,s}(p^{-O_{s}(1)^{-1}})
\end{split}
\end{equation}
Since $d\geq (k+s+2)(k+s+1)+1$,
by Theorem \ref{4:ct}  and (\ref{4:ostdim}), this implies that 
\begin{equation}\nonumber
\begin{split}
&\quad\vert\E_{((x_{i})_{i\in I},\bold{h})\in\Omega^{0}_{I,s}}\Delta_{\bold{h}}f_{0}(x_{0})\vert^{\frac{1}{2^{s}}}=\vert\E_{(x_{0},\bold{h})\in \Gow_{s}(M-r)}\Delta_{\bold{h}}f_{0}(x_{0})\vert^{\frac{1}{2^{s}}}+O_{k,s}(p^{-O_{s}(1)^{-1}})
\\&=\Vert f_{0}\Vert_{U^{s}(M-r)}+O_{k,s}(p^{-O_{s}(1)^{-1}}).
\end{split}
\end{equation}
This completes the proof of Theorem \ref{4:vdcc}.
So it remains to prove Proposition \ref{4:indvdc}.

\begin{proof}[Proof of Proposition \ref{4:indvdc}]
	By changing the order of the variables, we may assume without loss of generality that $I=\{0,\dots,k\}$.
	Then $$x_{k+t}=\sum_{j=0}^{k}b_{I,k+t,j}x_{j}=x_{0}+\sum_{j=1}^{k}b_{I,k+t,j}(x_{j}-x_{0}).$$ So changing again the order of the variables if necessary, we may further assume that $b_{I,k+t,k}\neq 0$. Let $I'=\{x_{0},\dots,x_{k-1},x_{k+t}\}$. Since $b_{I,k+t,k}\neq 0$, it is not hard to see that $I'$ is also a generating set of $X$.
	
	The key to the proof of Proposition \ref{4:indvdc} is a change of variable that allows one to switch from the ``$I$-basis" to the  ``$I'$-basis". Indeed,
	unpacking the conclusion of Proposition \ref{4:kk4k} and using the fact that $b_{I,i,j}=\delta_{i,j}$ for all $i,j\in I$, we deduce that 
	\begin{equation}\label{4:rec}
	\begin{split}
	L_{I',i}(x_{0},\dots,x_{k-1},x_{k+t})
	=L_{I,i}(x_{0},\dots,x_{k-1},L_{I',k}(x_{0},\dots,x_{k-1},x_{k+t}))
	\end{split}
	\end{equation}
	for all 
	$i\in \{k+1,\dots,k+s-1\}\backslash\{k+t\}$.
	By the change of variable $x_{k}\mapsto L_{I',k}(x_{0},\dots,x_{k-1},x_{k+t})$ and (\ref{4:rec}),  the left hand side of (\ref{4:Vdc33}) is equal to
	\begin{equation}\label{4:Vdc4}
	\begin{split}
\Bigl\vert\E_{(\bold{x}=(x_{0},\dots,x_{k-1},x_{k+t}),\bold{h})\in W} 
	\Delta_{\bold{h}}f_{0}(x_{0}) \Delta_{\bold{h}}f_{k+t}(x_{k+t})\prod_{i=0}^{t-1}\Delta_{\bold{h}}f_{k+i}(L_{I',k+i}(\bold{x}))\Bigr\vert^{2},
	\end{split}
	\end{equation}
	where $W$ is the set of $(\bold{x}=(x_{0},\dots,x_{k-1},x_{k+t}),\bold{h}=(h_{1},\dots,h_{s-t}))\in(\V)^{k+s}$ such that 
	\begin{equation}\nonumber
	\begin{split}
	(x_{0}+\e_{1}h_{1}+\dots+\e_{s-t}h_{s-t},x_{1},\dots,x_{k-1},L_{I',k}(x_{0}+\e_{1}h_{1}+\dots+\e_{s-t}h_{s-t},\dots,x_{k-1},x_{k+t}))\in \Omega_{I}
	\end{split}
	\end{equation}
	for all $\e_{1},\dots,\e_{s-t}\in\{0,1\}$, or equivalently,
	\begin{equation}\nonumber
	\begin{split}
	(x_{0}+\e_{1}h_{1}+\dots+\e_{s-t}h_{s-t},x_{1},\dots,x_{k-1},x_{k+t})\in \pi_{I}\circ L_{I'}^{-1}(\Omega_{I})=\Omega_{I'},
	\end{split}
	\end{equation}
	where we used (\ref{4:rec}) and Corollary \ref{4:swy}.	
	In other words, we have that $W=\Omega^{0}_{I',s-t}$. So (\ref{4:Vdc4}) is equal to   
	\begin{equation}\label{4:Vdc45}
	\begin{split}
	\Bigl\vert\E_{(\bold{x}=(x_{0},\dots,x_{k-1},x_{k+t}),\bold{h})\in \Omega^{0}_{I',s-t}} 
	\Delta_{\bold{h}}f_{0}(x_{0}) \Delta_{\bold{h}}f_{k+t}(x_{k+t})\prod_{i=1}^{t-1}\Delta_{\bold{h}}f_{k+i}(L_{I',k+i}(\bold{x}))\Bigr\vert^{2}.
	\end{split}
	\end{equation}

We refer the readers to Appendices \ref{4:s:AppB3} and \ref{4:s:AppB4} for the terminologies to be used in the discussion below.
 By  (\ref{4:ostdim}), $\Omega^{0}_{I',s-t}$  is a pure and consistent $M$-set of total co-dimension at most $\binom{k+s-t+2}{2}\leq \binom{k+s+1}{2}$. 
So we may write $\Omega^{0}_{I',s-t}=V(\mathcal{J})$ for some consistent $(M,k+s-t+1)$-family $\mathcal{J}\subseteq \F_{p}[x_{0},\dots,x_{k-1},x_{k+t},h_{1},\dots,h_{s-t}]$ of total dimension at most $\binom{k+s-t+2}{2}$. Let $(\mathcal{J}',\mathcal{J}'')$ be an  $\{x_{1},\dots,x_{k-1},x_{k+t},h_{1},\dots,h_{s-t}\}$-decomposition of $\mathcal{J}$. 
	For convenience denote $\y=(x_{1},\dots,x_{k-1},x_{k+t})$. 
	Let $W'$ be the set of $(\y,\h)\in(\V)^{k+s-t}$  such that $(x_{0},\y,\h)\in V(\mathcal{J}')$ for all $x_{0}\in \V$.  
	For $(\y,\h)\in(\V)^{k+s-t}$, let  $\Omega^{0}_{I',s-t}(\y,\h)$ denote the set of $x_{0}\in \V$ such that $(x_{0},\y,\h)\in V(\mathcal{J}'')$.
Since $d\geq (k+s+1)(k+s)+1$, 
By Theorem \ref{4:ct} and the Cauchy-Schwartz inequality,   (\ref{4:Vdc45}) is bounded by  
	\begin{equation}\label{4:Vdc5}
	\begin{split}
	&\quad \Bigl\vert\E_{(\bold{y},\bold{h})\in W'}\Delta_{\bold{h}}f_{k+t}(x_{k+t}) \E_{x_{0}\in \Omega_{I',s-t}^{0}(\bold{y},\bold{h})}
	\Delta_{\bold{h}}f_{0}(x_{0}) \prod_{i=1}^{t-1}\Delta_{\bold{h}}f_{k+i}(L_{I',k+i}(x_{0},\bold{y}))\Bigr\vert^{2}+O_{k,s}(p^{-1})	
	\\&\leq \E_{(\bold{y},\bold{h})\in W'}\Bigl\vert\E_{x_{0}\in \Omega_{I',s-t}^{0}(\bold{y},\bold{h})}
	\Delta_{\bold{h}}f_{0}(x_{0}) \prod_{i=1}^{t-1}\Delta_{\bold{h}}f_{k+i}(L_{I',k+i}(x_{0},\bold{y}))\Bigr\vert^{2}+O_{k,s}(p^{-1})
	\\&=\E_{(\bold{y},\bold{h})\in W'}\E_{x_{0},x_{0}+h'\in\Omega_{I',s-t}^{0}(\bold{y},\bold{h})}
	\Delta_{h'}\Delta_{\bold{h}}f_{0}(x_{0}) \prod_{i=1}^{t-1}\Delta_{h'}\Delta_{\bold{h}}f_{k+i}(L_{I',k+i}(x_{0},\bold{y}))+O_{k,s}(p^{-1}).
	\end{split}
	\end{equation}

	By  (\ref{4:ostdim}), $\Omega^{0}_{I',s-t+1}$  is a consistent $M$-set of total co-dimension at most $\binom{k+s-t+3}{2}\leq \binom{k+s+2}{2}$. So we may write $\Omega^{0}_{I',s-t+1}=V(\mathcal{J})$ for some consistent $(M,k+s-t+2)$-family $\mathcal{J}\subseteq \F_{p}[x_{0},\dots,x_{k-1},x_{k+t},h_{1},\dots,h_{s-t},h']$ of total dimension at most $\binom{k+s-t+3}{2}$. Let $(\mathcal{J}',\mathcal{J}'')$ be an  $\{x_{1},\dots,x_{k-1},x_{k+t},h_{1},\dots,h_{s-t}\}$-decomposition of $\mathcal{J}$. 
	It is not hard to compute that the set of $(\y,\h)\in(\V)^{k+s-t}$  such that $(x_{0},\y,\h,h')\in V(\mathcal{J}')$ for all $x_{0},h'\in \V$ is equal to $W'$.
	For $(\y,\h)\in W'$  and $(x_{0},h')\in(\V)^{2}$, note that  $(x_{0},\y,\h,h')\in V(\mathcal{J}'')$ if and only if $x_{0},x_{0}+h\in \Omega_{I',s-t}^{0}(\y,\h)$.
	Write $\h':=(\h,h')$.
	By Theorem \ref{4:ct},  since $d\geq (k+s+2)(k+s+1)+1$, the right hand side of (\ref{4:Vdc5}) is bounded by
	\begin{equation}\nonumber
	\begin{split}
	\E_{(\bold{x}=(x_{i})_{i\in I'},\bold{h}')\in \Omega^{0}_{I',s-t+1}} 
	\Delta_{\bold{h}'}f_{0}(x_{0}) \prod_{i=1}^{t-1}\Delta_{\bold{h}'}f_{k+i}(L_{I',k+i}(\bold{x}))+O_{k,s}(p^{-1/2}).
	\end{split}
	\end{equation}
	where $\h':=(\h,h')$. This completes the proof of Proposition \ref{4:indvdc}.
\end{proof}

\section{The rationality of polynomial sequences}\label{4:s:2rr}

\subsection{Type-III horizontal  torus and character}
We refer the readers to Appendix \ref{4:s:AppA} for the definitions and properties for nilmanifolds defined in \cite{SunA}. In this paper, in addition to the type-I and type-II horizontal  toruses and characters defined in \cite{SunA} and \cite{SunC} respectively, we need to introduce a third type of horizontal torus and character.

\begin{defn}[$i$-th type-III horizontal  torus and character]
	Let $s\in\N_{+}$ and $G/\Gamma$ be an $\N$-filtered nilmanifold of degree at most $s$ with a filtration $G_{\N}$ and a Mal'cev basis $\mathcal{X}$. 
	Let $G_{i}^{\nabla}$ denote the group generated by $G_{i+1}$ and $[G_{j},G_{i-j}]$ for all $0\leq j\leq i$ for $i\in\N$. 	Then $G_{i}^{\nabla}$ is a subgroup of $G_{i}$. It is not hard to check that $(G_{i}^{\nabla})_{i\in\N}$ is an $\N$-filtration. We say that $G_{i}/G_{i}^{\nabla}(\Gamma\cap G_{i})$
	is the \emph{$i$-th type-III horizontal torus} of $G/\Gamma$.

	An \emph{$i$-th type-III horizontal character} of $G/\Gamma$ is a continuous homomorphism $\eta_{i}\colon G_{i}\to \R$ which vanishes on $G_{i}^{\nabla}$, and maps $\Gamma\cap G_{i}$ to $\Z$. We say that $\eta_{i}$ is \emph{nontrivial} if it is not constant zero. 
	The Mal'cev basis $\mathcal{X}$ induces a natural isomorphism $\psi\colon G_{i}/G_{i+1}\to\R^{k}$, and thus we may write $\eta_{i}(g_{i}):=(m_{1},\dots,m_{k})\cdot \psi(g_{i})$ for some $m_{1},\dots,m_{k}\in\Z$, where $k=\dim(G_{i})-\dim(G_{i+1})$. 
	We call the quantity $\Vert\eta_{i}\Vert:=\vert m_{1}\vert+\dots+\vert m_{k}\vert$ the \emph{complexity} of $\eta_{i}$ (with respect to $\mathcal{X}$). 
\end{defn}

Type-III horizontal  torus and character are used to characterize the rationality property of a nilsequence (see for example \cite{CS14,GT10b}), which is a property more delicate than the equidistribution property. 

We provide a proposition for type-III horizontal characters for later uses.

\begin{prop}\label{4:goodpartial}
	Let $d,s\in\N_{+}$ and $G/\Gamma$ be an $\N$-filtered nilmanifold of degree at most $s$ with a filtration $G_{\N}$. Let $g,g'\in \poly(\Z^{d}\to G_{\N})$, $1\leq i\leq s$, and $\eta_{i}\colon G_{i}\to\R$ be an  $i$-th type-III horizontal character of $G/\Gamma$. 
	
	(i) Let 
	$$g(n)=\prod_{j\in\N^{d},0\leq \vert j\vert\leq s}g_{j}^{\binom{n}{j}}$$
	 be the type-I Taylor expansion of $g$, where $g_{j}\in G_{\vert j\vert}$. Then for all $h_{1},\dots,h_{i}\in\Z^{d}$, $\eta_{i}(\Delta_{h_{i}}\dots\Delta_{h_{1}}g(n))$ is well defined and is equal to
	$\Delta_{h_{i}}\dots\Delta_{h_{1}}\Bigl(\sum_{j\in\N^{d}, \vert j\vert=i}\eta_{i}(g_{j}){\binom{n}{j}}\Bigr)$.
	
	(ii) For all $h_{1},\dots,h_{i}\in\Z^{d}$, we have that 
	$$\eta_{i}(\Delta_{h_{i}}\dots\Delta_{h_{1}}g(n))+\eta_{i}(\Delta_{h_{i}}\dots\Delta_{h_{1}}g'(n))
  =\eta_{i}(\Delta_{h_{i}}\dots\Delta_{h_{1}}gg'(n)).$$
\end{prop}	
\begin{proof}
	We first prove Part (i). $\eta_{i}(\Delta_{h_{i}}\dots\Delta_{h_{1}}g(n))$ is well defined since $\Delta_{h_{i}}\dots\Delta_{h_{1}}g(n)\in G_{i}$.	
	 Since $\eta_{i}$ vanishes on $G_{i}^{\nabla}$, 	 by repeatedly using Corollary B.7 of \cite{GTZ12}, it is not hard to see that
	 $$\Delta_{h_{i}}\dots\Delta_{h_{1}}g(n)\equiv \prod_{j\in\N^{d},\vert j\vert=i}g_{j}^{\Delta_{h_{i}}\dots \Delta_{h_{1}}\binom{n}{j}} \mod G^{\nabla}_{\vert i\vert}.$$
	 So
	$$\eta_{i}(\Delta_{h_{i}}\dots\Delta_{h_{1}}g(n))=\eta_{i}\Bigl(\prod_{j\in\N^{d},\vert j\vert=i}g_{j}^{\Delta_{h_{i}}\dots \Delta_{h_{1}}\binom{n}{j}}\Bigr)=\Delta_{h_{i}}\dots\Delta_{h_{1}}\Bigl(\sum_{j\in\N^{d}, \vert j\vert=i}\eta_{i}(g_{j}){\binom{n}{j}}\Bigr).$$
	
	We now prove Part (ii). By  the Baker-Campbell-Hausdorff formula (see Appendix \ref{4:s:AppA3}) and Corollary B.7 of \cite{GTZ12}, it is not hard to see that the $i$-th  type-I Taylor coefficient of $gg'$ equals to that   of $g$ multiplied by that  of $g'$ modulo $G^{\nabla}_{\vert i\vert}$. The conclusion then follows from Part (i) and the linearity of $\eta_{i}\mod G^{\nabla}_{\vert i\vert}$.
\end{proof}

\subsection{The rationality version of the factorization theorem}

As was explain the paper of Green and Tao \cite{GT10b}, the type-I horizontal characters and the factorization theorem pertaining to them  are insufficient for the study of the equidistribution property for Leibman groups (see Section \ref{4:slast} for the definition).
For this reason, we need to upgrade the equidistribution result obtained in the previous part of the series
(i.e. Theorems 
11.2 of \cite{SunA})
 for type-III horizontal characters. We start with the following concept.

\begin{defn}[Irrationality]
	Let $d,s\in\N_{+}, N>0$, $p$ be a prime, $G/\Gamma$ be an $\N$-filtered nilmanifold of degree at most $s$ with a filtration $G_{\N}$ and a Mal'cev basis $\mathcal{X}$, $\Omega$ be a subset of $\Z^{d}$, and $g\in\poly_{p}(\Omega\to G_{\N}\vert\Gamma)$. We say that $g$ is \emph{$(N,\Omega,p)$-rational} there exist $1\leq i\leq s$ and a nontrivial $i$-th type-III horizontal character $\eta_{i}\colon G_{i}\to\R$ of complexity at most $N$ 	such that the map
	$$(n,h_{1},\dots,h_{i})\mapsto\eta_{i}(\Delta_{h_{i}}\dots\Delta_{h_{1}}g(n)) \mod \Z$$
	(which is well defined by Proposition \ref{4:goodpartial}) is a constant on $\Gow_{p,i}(\Omega)$ (recall Section \ref{4:s:defn} for the definition).  
	We say that  $g$ is \emph{$(N,\Omega,p)$-irrational} if it is not $(N,\Omega,p)$-rational.

Let  $\Omega$ be a subset of $\F_{p}^{d}$ and $g\in\poly_{p}(\Omega\to G_{\N}\vert\Gamma)$ (recall the definitions of different notions of polynomial sequences in Appendices \ref{4:s:AppA1} and \ref{4:s:AppA2}).	 We say that $g$ is \emph{$(N,\Omega)$-rational/irrational} if $g=g'\circ\tau$ for some $g'\in\poly_{p}(\iota^{-1}(\Omega)\to G_{\N}\vert\Gamma)$ which is  $(N,\iota^{-1}(\Omega),p)$-rational/irrational.  
\end{defn}

For convenience, we define every polynomial of degree $s=0$ (i.e. every constant polynomial) to be $(N,\iota^{-1}(\Omega),p)$-irrational or $(N,\Omega)$-irrational for all $N>0$.

The next lemma says that the rationality of a polynomial in $\poly_{p}(\Omega\to G_{\N}\vert\Gamma)$ is independent of choice of the representatives in $\poly_{p}(\iota^{-1}(\Omega)\to G_{\N}\vert\Gamma)$.

\begin{lem}\label{4:wdirr}
Let  $\Omega$ be a subset of $\F_{p}^{d}$ and $g\in\poly_{p}(\Omega\to G_{\N}\vert\Gamma)$.
If for some $g'\in\poly_{p}(\iota^{-1}(\Omega)\to G_{\N}\vert\Gamma)$  with $g=g'\circ\tau$, $g'$ is $(N,\iota^{-1}(\Omega),p)$-rational/irrational, then for every $g'\in\poly_{p}(\iota^{-1}(\Omega)\to G_{\N}\vert\Gamma)$  with $g=g'\circ\tau$, we have that $g'$ is $(N,\iota^{-1}(\Omega),p)$-rational/irrational. In other words, the  $(N,\Omega)$-rational/irrational property of $g$ is independent of the choices of $g'$.
\end{lem}
\begin{proof}
Assume without loss of generality that $G/\Gamma$ is of degree at least 1. 
Let $g',g''\in\poly_{p}(\iota^{-1}(\Omega)\to G_{\N}\vert\Gamma)$ 
   with   $g=g'\circ\tau=g''\circ\tau$. It suffices to show that  for any $i$-th type-III horizontal character $\eta_{i}\colon G_{i}\to\R$ and any $(n,h_{1},\dots,h_{i})\in\Gow_{p,i}(\iota^{-1}(\Omega))$, we have that
   \begin{equation}\label{4:gpp}
   \begin{split}
    \eta_{i}(\Delta_{h_{i}}\dots\Delta_{h_{1}}g''(n))\equiv\eta_{i}(\Delta_{h_{i}}\dots\Delta_{h_{1}}g'(n)) \mod \Z.  
   \end{split}
   \end{equation}
   
   Write $\tilde{g}:={g'}^{-1}g''$. Since $g'\circ\tau=g''\circ\tau$, we have that $g'(n)^{-1}g''(n)\in\Gamma$ for all $n\in\tau(\Omega)$. Since $g',g''\in\poly_{p}(\iota^{-1}(\Omega)\to G_{I}\vert\Gamma)$, for all $n\in\tau(\Omega)$ and $m\in\Z^{k}$, we have that
     $$g'(n+pm)^{-1}g''(n+pm)=(g'(n+pm)^{-1}g'(n))(g'(n)^{-1}g''(n))(g''(n)^{-1}g''(n+pm))\in\Gamma.$$
     So  
 $\tilde{g}\in\poly(\iota^{-1}(\Omega)\to G_{\N}\vert\Gamma)$. 
 
   By  Proposition  \ref{4:goodpartial},     
   \begin{equation}\label{4:gpp2}
      \begin{split}
    \eta_{i}(\Delta_{h_{i}}\dots\Delta_{h_{1}}g''(n))-\eta_{i}(\Delta_{h_{i}}\dots\Delta_{h_{1}}g'(n))
  =\eta_{i}(\Delta_{h_{i}}\dots\Delta_{h_{1}}\tilde{g}(n)).  
   \end{split}
   \end{equation}
 For  any $(n,h_{1},\dots,h_{i})\in\Gow_{p,i}(\iota^{-1}(\Omega))$,
  since $\tilde{g}\in\poly(\iota^{-1}(\Omega)\to G_{\N}\vert\Gamma)$, $\Delta_{h_{i}}\dots\Delta_{h_{1}}\tilde{g}(n)$ belongs to $\Gamma$. So we have that 
$\eta_{i}(\Delta_{h_{i}}\dots\Delta_{h_{1}}\tilde{g}(n))\in\Z$ and thus (\ref{4:gpp}) follows from  (\ref{4:gpp2}).
\end{proof}

  Let $\Omega$ be a non-empty finite set and $G/\Gamma$ be an $\N$-filtered nilmanifold. We say that a sequence $\O\colon\Omega\to G/\Gamma$ is \emph{$\d$-equidistributed} on $G/\Gamma$ if for all Lipschitz function $F\colon G/\Gamma\to\C$, we have that
 	$$\Bigl\vert\frac{1}{\vert \Omega\vert}\sum_{n\in\Omega}F(\O(n))-\int_{G/\Gamma}F\,dm_{G/\Gamma}\Bigr\vert\leq \d\Vert F\Vert_{\Lip},$$
	where  $m_{G/\Gamma}$ is the Haar measure of $G/\Gamma$.\footnote{This definition is slight different from the one used in \cite{SunA,SunB,SunC}, where $\Omega$ is a (possibly infinite) subset of $\Z^{d}$.}
The next lemma explains the connection between  irrationality and equidistribution of polynomial sequences. We do not use it in this paper, but it is interesting on its own.

\begin{lem}[Irrationality implies equidistribution]\label{4:Leib2}
	Let $d\in\N_{+},s\in\N,C,C',K,N>0$, $p$ be a prime, 
	and $G/\Gamma$ be an  $\N$-filtered nilmanifold of degree at most $s$ and complexity at most $C$. There exists $C''=C''(C,C',K,s)>0$ such that  for any
  $\Omega\subseteq\V$ which admits a partially periodic  $(\d,C'\d^{-K})$-Leibman dichotomy\footnote{See Appendix \ref{4:s:vr} for the definition.} up to step $s$ and complexity $C$ and any $(N,\Omega)$-irrational $g\in\poly_{p}(\Omega\to G_{\N}\vert\Gamma)$, the sequence  $(g(n)\Gamma)_{n\in\Omega}$ is 
	$C''N^{-1/K}$-equidistributed on $G/\Gamma$.
\end{lem}	
\begin{proof}
There is nothing to prove when $s=0$, and so we assume that $s\geq 1$.
	Let $\d>0$ to be chosen later. Assume that $g=g'\circ\tau$ for some $g'\in\poly_{p}(\iota^{-1}(\Omega)\to G_{\N}\vert\Gamma)$. Suppose that $(g(n)\Gamma)_{n\in\Omega}$ is not $\d$-equidistributed on $G/\Gamma$, then by assumption, there exists a nontrivial type-I horizontal character $\eta\colon G\to\R$ of complexity at most $C'\d^{-K}$ such that $\eta\circ g \mod \Z$ is a constant on $\Omega$ and thus $\eta\circ g' \mod \Z$ is a constant on $\iota^{-1}(\Omega)$.

	Let $i$ be the largest integer $1\leq i\leq s$ such that $\eta\vert_{G_{i}}$ is nontrivial. 
	Since $\eta\circ g' \mod \Z$ is a constant on $\iota^{-1}(\Omega)$, we have that $$\eta\vert_{G_{i}}(\Delta_{h_{i}}\dots\Delta_{h_{1}}g'(n))=\eta(\Delta_{h_{i}}\dots\Delta_{h_{1}}g'(n))=\Delta_{h_{i}}\dots\Delta_{h_{1}}\eta\circ g'(n)$$ is a constant mod $\Z$ for all $(n,h_{1},\dots,h_{i})\in\Gow_{p,i}(\iota^{-1}(\Omega))$. 
	Similar to the proof of Lemma 3.7 of \cite{GT10b}, $\eta\vert_{G_{i}}$ is an $i$-th type-III  horizontal character of complexity $O_{C,s}(C'\d^{-K})$. So $g'$ is not $(O_{C,s}(C'\d^{-K}),\Omega,p)$-irrational and thus $g$ is not $(O_{C,s}(C'\d^{-K}),\Omega)$-irrational. We arrive at a contradiction if  $O_{C,s}(C'\d^{-K})\leq N$, i.e. if $\d\geq C''N^{-1/K}$ for some $C''>0$ depending only on $C,C',K$ and $s$.
\end{proof}	

We now state and prove the factorization property regarding the irrationality polynomials sequences.

\begin{prop}[Weak factorization property, the rationality version]\label{4:newfacw}
	Let $d\in\N_{+}, s\in\N$ with $d\geq s^{2}+s+3$, $C,N>0$, $p$ be a prime, 
	 $M\colon\V\to\F_{p}$ be a non-degenerate quadratic form,
	  $G/\Gamma$ be an $s$-step $\N$-filtered nilmanifold of complexity at most $C$, equipped with a $C$-rational Mal'cev basis $\mathcal{X}$, and $g\in \poly_{p}(V(M)\to G_{\N}\vert\Gamma)$. If $p\gg_{C,d,N} 1$,  
	then either $g$ is $(N,V(M))$-irrational or
	there exists a factorization 
	$$g(n)=\e g'(n)\gamma(n)      \text{ for all }  n\in \V,$$  where $\e\in G$ is of complexity at most 1, 
	$g'\in \poly_{p}(V(M)\to G'_{\N}\vert\Gamma')$
	for some sub-nilmanifold $G'/\Gamma'$ of $G/\Gamma$ whose Mal'cev basis is $O_{C,d,N}(1)$-rational relative to $\mathcal{X}$ and whose total dimension is strictly smaller than that of $G/\Gamma$, and   $\gamma\in\poly(V(M)\to G_{\N}\vert\Gamma)$. 
	\end{prop}
\begin{proof}
The proof uses ideas from Lemma 2.9 of \cite{GT12} and the work \cite{CS14}. However, in our case, we need some extra effort to obtain the desired periodicity requirements for the sequences $g'$ and $\gamma$.

There is nothing to prove when $s=0$, and so we assume that $s\geq 1$.
Let $\tilde{M}\colon\Z^{d}\to\Z/p$ be a regular lifting of $M$.
	We may write $g=\tilde{g}\circ\tau$ for some $\tilde{g}\in\poly_{p}(\iota^{-1}(V(M))\to G_{\N}\vert\Gamma)=\poly_{p}(V_{p}(\tilde{M})\to G_{\N}\vert\Gamma)$ (recall the definition of $V_{p}(\tilde{M})$ in Section \ref{4:s:defn}). Suppose that $g$ is not $(N,V(M))$-irrational. 
	Then  there exist $1\leq i\leq s$ and a nontrivial $i$-th type-II horizontal character $\eta_{i}\colon G_{i}\to\R$ of complexity at most $N$  such that 
	$$\eta_{i}(\Delta_{h_{i}}\dots\Delta_{h_{1}}\tilde{g}(n)) \mod \Z$$
	is a constant for all $(n,h_{1},\dots,h_{i})\in\Gow_{p,i}(\iota^{-1}(V(M)))=\Gow_{p,i}(V_{p}(\tilde{M}))$.
	We may use the type-I Taylor expansion to write
	$$\tilde{g}(n)=\prod_{j\in\N^{d}, \vert j\vert\leq s}g_{j}^{\binom{n}{j}}$$
	for some $g_{j}\in G_{\vert j\vert}$.
	Denote
	$g_{-}(n):=\prod_{j\in\N^{d}, \vert j\vert\leq i-1}g_{j}^{\binom{n}{j}}$, $g_{+}(n):=\prod_{j\in\N^{d}, i+1\leq \vert j\vert\leq s}g_{j}^{\binom{n}{j}}$ and 
	  $f(n):=\prod_{j\in\N^{d}, \vert j\vert=i}g_{j}^{\binom{n}{j}}$.
	  Then
	  \begin{equation}\label{4:gfg}
	      \tilde{g}=g_{-}fg_{+}.
	  \end{equation}
	By Proposition  \ref{4:goodpartial} (i), we have that
	$$\eta_{i}(\Delta_{h_{i}}\dots\Delta_{h_{1}}\tilde{g}(n))\equiv\Delta_{h_{i}}\dots\Delta_{h_{1}}(\eta_{i}\circ f(n)) \mod \Z$$
	is a constant for all $(n,h_{1},\dots,h_{i})\in\Gow_{p,i}(V_{p}(\tilde{M}))$. Note that $\eta_{i}\circ f$ is a polynomial from $\Z^{d}$ to $\R$ of degree $i$.

	For any $j=(j_{1},\dots,j_{d})\in\N^{d}$ with $\vert j\vert=i$,  pick $1\leq d_{1},\dots,d_{i}\leq d$ so that $j_{\ell}$ of  $d_{1},\dots,d_{i}$ equals to $\ell$ for all $1\leq \ell\leq d$.
	 Then for any $n\in V_{p}(\tilde{M})$, we have  that $(n,pe_{d_{1}},\dots,pe_{d_{i}})\in \Gow_{p,i}(V_{p}(\tilde{M}))$.  Since $\tilde{g}\in \poly_{p}(V_{p}(\tilde{M})\to G_{\N}\vert\Gamma)$, we have	$\Delta_{pe_{d_{i}}}\dots\Delta_{pe_{d_{1}}}\tilde{g}(n)\in\Gamma$ and so 
	\begin{equation}\nonumber
	\begin{split}
	  0\equiv\eta_{i}(\Delta_{pe_{d_{i}}}\dots\Delta_{pe_{d_{1}}}\tilde{g}(n))
	 \equiv\Delta_{pe_{d_{i}}}\dots\Delta_{pe_{d_{1}}}(\eta_{i}\circ f(n)) 
	=p^{i}\eta_{i}(g_{j}) \mod\Z.
	\end{split}
	\end{equation} 
	Therefore, $\eta_{i}\circ f$ takes values in $\Z/p^{s}$.
	Since $d\geq s^{2}+s+3$, by Proposition \ref{4:att3}, we may write 
	$$\eta_{i}\circ f=\frac{1}{q}P_{1}+P_{2}$$ for some $q\in\N, p\nmid q$, some $P_{1}\in\poly(V_{p}(\tilde{M})\to \R\vert\Z)$ and some polynomial $P_{2}$  of degree at most $i-1$. 
		
	Let $\psi\colon G_{i}/G_{i+1}\to\R^{k}$ be the natural isomorphism induced by the Mal'cev basis of $G/\Gamma$ for some $k\in\N$, and write $\eta_{i}(g):=(m_{1},\dots,m_{k})\cdot \psi(g)$ for some $m_{1},\dots,m_{k}\in\Z$ with $\vert m_{1}\vert+\dots+\vert m_{k}\vert\leq N$ for all $g\in G_{i}$. Let $c$ be the greatest common divisor of $m_{1},\dots,m_{k}$. Since $\eta_{i}$ is nontrivial and is of complexity at most $N$, we have that $0<c\leq O_{C,N,s}(1)$. In particular, $c<p$.
	 We may then pick $u=(u_{1},\dots,u_{k})\in \Z^{k}$ such that $(m_{1},\dots,m_{k})\cdot(u_{1},\dots,u_{k})=c$.  Let $h_{1},\dots,h_{k}\in G_{i}\cap\Gamma$ be such that $\psi(h_{t}\mod G_{i+1})=e_{t}$ for all $1\leq t\leq k$.
	Denote 
	$$\gamma_{0}(n):=\prod_{t=1}^{k}h_{t}^{c^{\ast}u_{t}P_{1}(n)},$$
	where $c^{\ast}$  is an integer  such that $c^{\ast}cq\equiv 1 \mod p^{s}\Z$ (which exists since $0<c<p$). Since $P_{1}(n)\in\Z$ if $n\in V_{p}(\tilde{M})$, we have that $\gamma_{0}\in \poly(V_{p}(\tilde{M})\to G_{\N}\vert\Gamma)$. Moreover, for all $n\in\Z^{d}$, we have 
	\begin{equation}\label{4:fxi2}
	\eta_{i}\circ f-\eta_{i}\circ\gamma_{0}=\frac{1-c^{\ast}cq}{q}P_{1}+P_{2}.
	\end{equation}

	\textbf{Claim.}
	For any  $i$-th type-III horizontal character $\xi_{i}$, and $j\in\N^{d}$ with $\vert j\vert=i$, we have that $\xi_{i}(g_{j})\in \Z/p^{s}$.
	
	Indeed, denote $F_{\xi_{i}}(h_{1},\dots,h_{i}):=\xi_{i}(\Delta_{h_{i}}\dots\Delta_{h_{1}}g'(n))$. Note that $F_{\xi_{i}}(h_{1},\dots,h_{i})$ is independent of the choice of $n$. It is not hard to compute that
	\begin{equation}\label{4:fxi}
	F_{\xi_{i}}(h_{1},\dots,h_{i}):=\sum_{j\in\N^{d},\vert j\vert=i}\xi_{i}(g_{j})\sum_{(t_{1},\dots,t_{i})\in J(j)}h_{1,t_{1}}\dots h_{i,t_{i}},
	\end{equation}
	where we write $h_{i'}=(h_{i',1},\dots,h_{i',d})$ for $1\leq i'\leq i$ and $J(j)$ is the set of tuples $(t_{1},\dots,t_{i})\in\{1,\dots,d\}^{i}$ such that the number of $t_{k}$ which equals to $r$ is $j_{r}$ for $1\leq r\leq d$.
	Since $g'\in\poly_{p}(V_{p}(\tilde{M})\to G_{\N}\vert\Gamma)$,
	for all $n\in V_{p}(\tilde{M}),$ and $h_{1},\dots,h_{i}\in\Z^{d}$,  we have that  $\Delta_{ph_{i}}\dots\Delta_{ph_{1}}g'(n)\in\Gamma$, which implies that $F_{\xi_{i}}(h_{1},\dots,h_{i})\in\Z/p^{s}$. 
	By (\ref{4:fxi}), this implies that 
	$\xi_{i}(g_{j})\in \Z/p^{s}$. This completes the proof of the claim.

\
	
	By the claim, $\psi(g_{j}^{p^{s}} \mod G_{i+1})\in \Z^{k}$ and so we have that 
	$p^{s}\eta_{i}(g_{j})=\eta_{i}(g_{j}^{p^{s}})\in \eta_{i}(G_{i}\cap\Gamma)=c\Z$. Since $1-c^{\ast}cq\in p\Z$, we have $(1-c^{\ast}cq)\eta_{i}(g_{j})\in c\Z$.
	Then for all $j\in\N^{d}$ with $\vert j\vert=i$, there exists $\gamma_{j}\in G_{i}\cap\Gamma$ such that $\eta_{i}(\gamma_{j})=(1-c^{\ast}cq)\eta_{i}(g_{j})$.
	
	Note that the leading terms of $\frac{1-c^{\ast}cq}{q}P_{1}$ is  $\sum_{j\in\N^{d}, \vert j\vert=i}(1-c^{\ast}cq)\eta_{i}(g_{j})\frac{n^{j}}{j!}$.
	Writing
	$$\gamma'(n):=\prod_{j\in\N^{d},\vert j\vert=i}\gamma_{j}^{\binom{n}{j}},$$
	we have that $\gamma'$ takes values in $\Gamma$ and that
	$\eta_{i}\circ\gamma'(n)$ has the same leading terms as $\frac{1-c^{\ast}cq}{q}P_{1}(n)$.
	So it follows from (\ref{4:fxi2}) that 
	\begin{equation}\nonumber
	\eta_{i}\circ f-\eta_{i}\circ\gamma_{0}-\eta_{i}\circ\gamma'
	\end{equation}
	is a polynomial   of degree at most $i-1$.

	Since $\eta_{i}$ is nontrivial, 
	 there exists  $f'\in\poly(\Z^{d}\to G_{\N})$ 	of the form
	$$f'(n)=\prod_{j\in\N^{d},\vert j\vert\leq i-1}u_{j}^{\binom{n}{j}}$$
	for some $u_{j}\in G_{i}$ such that 
	$$\eta_{i}\circ f-\eta_{i}\circ\gamma_{0}-\eta_{i}\circ\gamma'=\eta_{i}\circ f'.$$
	
	 	Let $\gamma:=\gamma_{0}\gamma'$, and $g'':={f'}^{-1}f\gamma^{-1}$. We have that $\gamma\in \poly(V_{p}(\tilde{M})\to G_{\N}\vert\Gamma)$,  that $g''$ takes values in $G_{i}$, and that
	$\eta_{i}\circ g''=0$. 
	By (\ref{4:gfg}), we may write 
	$$\tilde{g}=g_{-}f g_{+}=g_{-}f'g''\gamma g_{+},$$
	where $g_{-}(n):=\prod_{j\in\N^{d}, \vert j\vert\leq i-1}g_{j}^{\binom{n}{j}}$ and $g_{+}(n):=\prod_{j\in\N^{d}, i+1\leq \vert j\vert\leq s}g_{j}^{\binom{n}{j}}$.
	By  the Baker-Campbell-Hausdorff formula, we may write
	$$\tilde{g}=g_{-}f'g''g'_{+}\gamma$$
	for some $g'_{+}$ which can be written in the form
	$g'_{+}(n)=\prod_{j\in\N^{d}, \vert j\vert\leq s}h_{j}^{\binom{n}{j}}$
	for some $h_{j}\in G_{\max\{\vert j\vert,i+1\}}$. 
	
	Let $G'/\Gamma'$ be the sub-nilmanifold of $G/\Gamma$ given by $\Gamma'=\Gamma\cap G'$, $G'_{j}=G_{j}$ for all $j\neq i$ and $G'_{i}=\ker(\eta_{i})$. Similar to the proof of Lemma 2.9 of \cite{GT10b}, $G'_{\N}$ is indeed an $\N$-filtration.\footnote{Unlike in the proof Lemma 2.9 of \cite{GT12}, our $G'_{\N}$ is an $\N$-filtration even if $i=1$, since in this paper we allow $G'_{0}\neq G'_{1}$ in the $\N$-filtration $G'_{\N}$.} Moreover, the  Mal'cev basis of $G'/\Gamma'$  is $O_{C,d,N}(1)$-rational relative to $\mathcal{X}$, and the total dimension of $G'/\Gamma'$ is strictly smaller than that of $G/\Gamma$.		
	
	For any $h\in G$, we may write 
  $h=\{h\}[h]$
such that $\psi(\{h\})\in [0,1)^{m}$ and $[h]\in\Gamma$, where $\psi\colon G\to\R^{m}$ is the Mal'cev coordinate map.
	Let $\e=\{\tilde{g}(\bold{0})\}$, $\tilde{\gamma}=[\tilde{g}(\bold{0})]\gamma$ and 
	$$g'=\{\tilde{g}(\bold{0})\}^{-1}\tilde{g}\tilde{\gamma}^{-1}=\{\tilde{g}(\bold{0})\}^{-1}\tilde{g}\gamma^{-1}[\tilde{g}(\bold{0})]^{-1}=\{\tilde{g}(\bold{0})\}^{-1}g_{-}f'g''g'_{+}[\tilde{g}(\bold{0})]^{-1}.$$ 
	Then  $\tilde{g}=\e g'\tilde{\gamma}$ and $\tilde{\gamma}\in \poly(V_{p}(\tilde{M})\to G_{\N}\vert\Gamma)$.
		By the constructions, it is not hard to see that  
	$g'',g'_{+},g_{-},f'\in\poly(\Z^{d}\to G'_{\N})$. By Corollary B.4 of \cite{GTZ12}, we have that $g'\in\poly(\Z^{d}\to G'_{\N})$. Since
	$\tilde{g}\in \poly_{p}(V_{p}(\tilde{M})\to G_{\N}\vert\Gamma)$ and $\tilde{\gamma}\in \poly(V_{p}(\tilde{M})\to G_{\N}\vert\Gamma)$, by Lemma \ref{4:grg}, it is not hard to see have that $g'$ belongs $\poly_{p}(V_{p}(\tilde{M})\to G_{\N}\vert\Gamma)$ and thus  belongs $\poly_{p}(V_{p}(\tilde{M})\to G'_{\N}\vert\Gamma')$. Since   $\iota^{-1}(V(M))=V_{p}(\tilde{M})$, we are done by composing $\tau$ on both sides.
\end{proof}

By repeatedly using Proposition \ref{4:newfacw}, we get the strong factorization property:

\begin{thm}[Strong factorization property, the rationality version]\label{4:newfacs}
	Let $d\in\N_{+}, s\in\N$ with $d\geq s^{2}+s+3$, $C>0$,  $\mathcal{F}\colon \R_{+}\to\R_{+}$ be a growth function, $p\gg_{C,d,\mathcal{F}} 1$ be a prime, 	  $M\colon\V\to\F_{p}$  be a non-degenerate quadratic form,
	  $G/\Gamma$ be an $s$-step $\N$-filtered nilmanifold of complexity at most $C$, equipped with a $C$-rational Mal'cev basis $\mathcal{X}$, and $g\in \poly_{p}(V(M)\to G_{\N}\vert\Gamma)$.
Then  there exists $C\leq C'\leq O_{C,d,\mathcal{F}}(1)$
	and a factorization 
	$$g(n)=\e g'(n)\gamma(n)  \text{ for all }  n\in\V$$  where $\e\in G$ if of complexity $O_{C'}(1)$,
	$g'\in \poly_{p}(V(M)\to G'_{\N}\vert \Gamma')$	is $(\mathcal{F}(C'),V(M))$-irrational
	for some sub-nilmanifold $G'/\Gamma'$ of $G/\Gamma$ whose Mal'cev basis is $C'$-rational  relative to $\mathcal{X}$, and   $\gamma\in\poly(V(M)\to G_{\N}\vert\Gamma)$. 
\end{thm}

Theorem \ref{4:newfacs} can be derived by repeatedly using Proposition \ref{4:newfacw}. The method is almost identical to the proof of Theorem 
11.2 in \cite{SunA}. So we omit the proof.

We conclude this section with a property stating that every partially periodic nilsequence on $V(M)$ can be represent as a sufficiently irrational one. 

\begin{prop}[Irrational representation property]\label{4:st2}
	Let $d\in\N_{+}, s\in\N$ with $d\geq  s^{2}+s+3$, $C>0$, $\mathcal{F}\colon\R_{+}\to\R_{+}$ be a growth function, $p\gg_{C,d,\mathcal{F}} 1$ be a prime, and  $M\colon\V\to\F_{p}$  be a non-degenerate quadratic form. Then
	there exists $C\leq C'\leq O_{C,d,\mathcal{F}}(1)$ such that for any function $f\colon V(M)\to\C$ of the form 
	$$f(n)=F(g(n)\Gamma)  \text{ for all }  n\in V(M),$$
	where   $G/\Gamma$  is an  $\N$-filtered nilmanifold   of degree at most $s$ and complexity at most $C$, $g\in\poly_{p}(V(M)\to G_{\N}\vert \Gamma)$, and $F\in\Lip(G/\Gamma\to\C)$ with Lipschitz norm at most $C$, we may rewrite $f$ as 
	$$f(n)=F'(g'(n)\Gamma')  \text{ for all }  n\in V(M)$$
	for some   $\N$-filtered nilmanifold $G'/\Gamma'$  of degree at most $s$ and complexity at most $O_{C,C'}(1)$, some $(\mathcal{F}(C'),V(M))$-irrational  $g'\in\poly_{p}(V(M)\to G'_{\N}\vert \Gamma')$, and some   $F'\in\Lip(G'/\Gamma'\to\C)$  with Lipschitz norm at most $O_{C,C'}(1)$. 
\end{prop}
\begin{proof}
	By Theorem \ref{4:newfacs}, 	there exists $C\leq C'\leq O_{C,d,\mathcal{F}}(1)$
	and a factorization 
	$$g(n)=\e g'(n)\gamma(n)  \text{ for all }  n\in\V$$  where $\e\in G$ is of complexity $O_{C'}(1)$,
	$g'\in \poly_{p}(V(M)\to G'_{\N}\vert\Gamma')$  
	 is $(\mathcal{F}(C'),V(M))$-rational
	in some sub-nilmanifold $G'/\Gamma'$ of $G/\Gamma$ whose Mal'cev basis is $C'$-rational relative to $\mathcal{X}$, and $\gamma\in\poly(V(M)\to G_{\N}\vert\Gamma)$. 
	Then $G'/\Gamma'$ is of complexity $O_{C,C'}(1)$.
	Writing $F'=F(\e\cdot)$, we have that $F'\colon G'/\Gamma'\to\C$ is a  Lipschitz function  with Lipschitz norm at most $O_{C,C'}(1)$. Moreover, for all $n\in V(M)$, we have that 
	$f(n)=F(\e g'(n)\gamma(n)\Gamma)=F'(g'(n)\Gamma).$ 
\end{proof}	

\begin{rem}
We caution the readers that were unable to extend Proposition \ref{4:st2} to  $p$-periodic nilsequences, i.e. even if $g$ belongs to $\poly_{p}(\V\to G_{\N}\vert\Gamma)$, we could not ensure that 
$g'$ belongs to $\poly_{p}(\V\to G'_{\N}\vert\Gamma')$ in the construction.
This is because our factorization theorem (Theorem \ref{4:newfacs}) applies only for partially $p$-periodic nilsequences (see also Remark 
11.3 in \cite{SunA}).
\end{rem}

\subsection{Structure theorems for spherical Gowers norms}

For any $s\in\N$ and $C>0$,
it is not hard to see that the sum and product of two $p$-periodic nilsequences of degrees at most  $s$ and complexities at most $C$ is still a $p$-periodic nilsequence of degree at most  $s$ and complexity at most $O(C)$ (see \cite{SunC} for definitions). So 
combining the spherical Gowers inverse theorem (Theorem 
1.2 of \cite{SunC}) with an argument similar to Section 2 of the work of Green and Tao  \cite{GT10b} (see also the discussion before Theorem 3.4 of \cite{CS14}),  we have the following structure theorem (whose proof is omitted):

\begin{thm}[Structure theorem for spherical Gowers norms in $\V$]\label{4:st}
	For every $s\in\N$, 	every $d\geq (2s+14)(15s+438)$, every $\e>0$, and every growth function $\mathcal{F}\colon \R_{+}\to\R_{+}$, there exist $C:=C(d,\e,\mathcal{F})>0$ and $p_{0}:=p_{0}(d,\e,\mathcal{F})\in\N$
	such that for every prime $p\geq p_{0}$, every non-degenerate quadratic form $M\colon\V\to\F_{p}$, and every $f\colon V(M)\to \mathbb{C}$ taking values in $[0,1]$, there exist  a decomposition
	$$f=f_{\str}+f_{\uni}+f_{\err}$$
	such that the followings hold:
	\begin{itemize}
		\item $f_{\str}$ and $f_{\str}+f_{\err}$ take values in $[0,1]$;
		\item we may write $f_{\str}$ as 
	$$f_{\str}(n)=F(g(n)\Gamma)  \text{ for all }  n\in V(M)$$
	for some   $\N$-filtered nilmanifold $G/\Gamma$  of degree at most $s$ and complexity at most $C$, some $g\in\poly_{p}(\V\to G_{\N}\vert \Gamma)$, and some    $F\in\Lip(G/\Gamma\to\C)$  of Lipschitz norm at most $C$;
		\item  $\Vert f_{\uni}\Vert_{U^{s+1}(V(M))}\leq 1/\mathcal{F}(C)$;
		\item $\Vert f_{\err}\Vert_{L^{2}(V(M))}\leq \e$.
	\end{itemize}	
\end{thm}
 
   We may also impose additional irrationality restrictions to the $f_{\str}$ part in Theorem \ref{4:st}, and have the following variation of the structure theorem:

 \begin{thm}[Another structure theorem for spherical Gowers norms in $\V$]\label{4:st3}
 	For every $s\in\N$, 	every $d\geq (2s+14)(15s+438)$, every
 	$\e>0$, and every growth function $\mathcal{F}\colon \R_{+}\to\R_{+}$, there exist $C:=C(d,\e,\mathcal{F})>0$ and $p_{0}:=p_{0}(d,\e,\mathcal{F})\in\N$
 	such that for every prime $p\geq p_{0}$, every non-degenerate quadratic form $M\colon\V\to\F_{p}$, and every $f\colon V(M)\to \mathbb{C}$ taking values in $[0,1]$, there exist  a decomposition
 	$$f=f_{\str}+f_{\uni}+f_{\err}$$
 	such that the followings hold:
 	\begin{itemize}
 		\item $f_{\str}$ and $f_{\str}+f_{\err}$ take values in $[0,1]$;
	\item we may write $f_{\str}$ as 
	$$f_{\str}(n)=F(g(n)\Gamma)  \text{ for all }  n\in V(M)$$
	for some   $\N$-filtered nilmanifold $G/\Gamma$  of degree at most $s$ and complexity at most $C$, some $(\mathcal{F}(C),V(M))$-irrational  $g\in\poly_{p}(V(M)\to G_{\N}\vert \Gamma)$, and some    $F\in\Lip(G/\Gamma\to\C)$  of Lipschitz norm at most $C$;
 		\item  $\Vert f_{\uni}\Vert_{U^{s+1}(V(M))}\leq 1/\mathcal{F}(C)$;
 		\item $\Vert f_{\err}\Vert_{L^{2}(V(M))}\leq \e$.
 	\end{itemize}	
 \end{thm}

 Since $\poly_{p}(\V\to G_{\N}\vert \Gamma)\subseteq \poly_{p}(V(M)\to G_{\N}\vert \Gamma)$,
  Theorem \ref{4:st3} can be proved by first applying Theorem \ref{4:st} for some growth function $\mathcal{F}_{1}$, and then applying Proposition  \ref{4:st2} for some growth function $\mathcal{F}_{2}$, with $\mathcal{F}_{1}$ growing much faster than $\mathcal{F}_{2}$ and with $\mathcal{F}_{2}$ growing much faster than $\mathcal{F}$. We omit the proof.  
 
 \begin{rem}
 Again we caution the readers that unlike in Theorem \ref{4:st} where $f_{\str}$ is $p$-periodic, we can only require $f_{\str}$ to be partially $p$-periodic on $V(M)$ in Theorem \ref{4:st3}. This is because we are only able to obtain the factorization theorem (Theorem \ref{4:newfacs}) for partially $p$-periodic nilsequences, which in turn is insufficient to extend Proposition  \ref{4:st2} to $p$-periodic nilsequences.
 \end{rem}

\section{Equisitribution properties for polynomial sequences on product spaces} \label{4:slast}

With the help of
Theorems  \ref{4:vdcc} and  \ref{4:st3}, one can reduce the study of (\ref{4:Vdc0}) to the special case when all the $f_{i}$ are nilsequences. Therefore, in order to obtain the positivity of (\ref{4:Vdc0}), in this section we provide a joint equisitribution property for polynomial sequences on the product space.

  \subsection{Parametrization for $\Omega$ and its relative self-products}\label{4:s:lbg}
  
  From now on it is convenient to work in the $\Z^{d}$-setting instead of the $\F_{p}^{d}$ one. 
Recall that for any generating set
$I$ of $X$, we have generating constants $b_{I,i,j}\in \F_{p}, i\in I, 0\leq j\leq k+s-1$ and generating maps $L_{I,j}\colon (\V)^{k+1}\to\V, 0\leq j\leq k+s-1$. 
For all $i\in I, 0\leq j\leq k+s-1$, let
 $\tilde{b}_{I,i,j}:=\tau(b_{I,i,j})\in\Z$,  and denote
$\tilde{b}_{I,i}:=(\tilde{b}_{I,i,0},\dots,\tilde{b}_{I,i,k+s-1})\in\Z^{k+s}$.
 Let
  $\tilde{L}_{I,j}\colon (\R^{d})^{k+1}\to\R^{d}$ be the linear map given by  
 $$\tilde{L}_{I,j}(\x):=\sum_{i\in I}\tilde{b}_{I,i,j}x_{i} \text{ for all } \x=(x_{0},\dots,x_{k})\in(\R^{d})^{k+1},$$
and $\tilde{\L}_{I}\colon (\R^{d})^{k+1}\to(\R^{d})^{k+s}$ be the linear map given by
 $$\tilde{\L}_{I}(\x):=(\tilde{L}_{I,0}(\x),\dots,\tilde{L}_{I,k+s-1}(\x)) \text{ for all } \x\in (\R^{d})^{k+1}.$$
 It is not hard to see that $L_{I,j}=\iota\circ \tilde{L}_{I,j}\circ \tau$. In other words, $\frac{1}{p}\tilde{L}_{I,j}$ is the regular lifting of $L_{I,j}$.

Let $C_{0}$ be the complexity of $X$.
If $p\gg_{C_{0},k,s} 1$, then the normalization properties of $b_{I,i,j}$ and $L_{I,j}$ imply the following \emph{normalization properties} of $\tilde{b}_{I,i,j}$ and $\tilde{L}_{I,j}$:
$$\tilde{b}_{I,i,j}=\delta_{i,j} \text{ for all $i,j\in I$ and} \sum_{i\in I}\tilde{b}_{I,i,j}=1 \text{ for all $j\in J$};\footnote{Since $p\gg_{C_{0},k,s} 1$, we have that $\sum_{i\in I}\tilde{b}_{I,i,j}=1$ instead of just $\sum_{i\in I}\tilde{b}_{I,i,j}\equiv 1\mod p\Z$.}$$
$$\tilde{L}_{I,j}((x_{i})_{i\in I})=x_{j} \text{ for all $j\in I$ and } \tilde{L}_{I,j}(x,\dots,x)=x \text{ for all $j\in J$}.$$

For $t\in\N_{+}$, we say that $V\subseteq \R^{t}$ is a \emph{rational subspace} of $\R^{t}$ if there exist $\R$-linearly independent vectors $u_{1},\dots,u_{r}\in \Z^{t}$ for some $r\in\N$ such that $V$ is the $\R$-span of $u_{1},\dots,u_{r}$.
For $i\in\N_{+}$, let $V^{[i]}$ denote the linear subspace of $\R^{t}$ spanned by $u_{1}\ast\ldots\ast u_{i}, u_{1},\dots,u_{i}\in V$, where $\ast$ is the coordinate-wise product.

Let $V_{X}:=\tilde{\L}_{I}((\R^{d})^{k+1})$. We summarize some basic properties for $V_{X}$.

\begin{lem}\label{4:vxprop}
Assume that the complexity of $X$ is at most $C_{0}$ and $p\gg_{C_{0},k,s} 1$.
\begin{enumerate}[(i)]
  \item The set $V_{X}$ is a well defined rational subspace of $\R^{k+s}$,  i.e. it is independent of the choice of the generating set $I$.
  \item We have that $V_{X}=\sp_{\R}\{\tilde{b}_{I,i}, i\in I\}$ for any generating set $I$ of $X$.
  \item The vectors $(1,\dots,1)$ belongs to $V_{X}$. Therefore, we have the \emph{flag property} $V_{X}^{[1]}\leq V_{X}^{[2]}\leq \dots.$
  \item The complexities of all of $V_{X}^{[i]}, 1\leq i\leq s-1$ (recall Section \ref{4:s:defn} for definitions) are at most $O_{C_{0},k,s}(1)$.
\end{enumerate}
\end{lem}
\begin{proof}
  We start with Part (i). It is clear that $V_{X}$ is a  rational subspace of $\R^{k+s}$. Let $I,I'$ be generating sets of $X$ with generating constants $b_{I,i,j}$ and $b_{I',i',j}$. Recall from (\ref{4:biterate})
 that
  \begin{equation}\nonumber
	\begin{split}
	   b_{I',i',j}=\sum_{i\in I}b_{I,i,j}b_{I',i',i}
	 \end{split}
	\end{equation}
	for all $j\in J$ and $i'\in I$. Since $p\gg_{C_{0},k,s} 1$ and all of $b_{I,i,j}$ and $b_{I',i,j}$ are of complexity at most $C_{0}$, we have that 
	  \begin{equation}\nonumber
	\begin{split}
	   \tilde{b}_{I',i',j}=\sum_{i\in I}\tilde{b}_{I,i,j}\tilde{b}_{I',i',i}
	 \end{split}
	\end{equation}
	for all $j\in J$ and $i'\in I$. 
	For all $\x=(x_{i'})_{i'\in I'}\in (\R^{d})^{k+1}$. Let $\y:=(\tilde{L}_{I',i}(\x))_{i\in I}$. Then for all $j\in J$,
	\begin{equation}\nonumber
	\begin{split}
	&\quad \tilde{L}_{I',j}(\x)
	=\sum_{i'\in I'}\tilde{b}_{I',i',j}x_{i'}
	=\sum_{i'\in I'}\Bigl(\sum_{i\in I}\tilde{b}_{I,i,j}\tilde{b}_{I',i',i}\Bigr)x_{i'}
	\\&=\sum_{i\in I}\tilde{b}_{I,i,j}\Bigl(\sum_{i'\in I'}\tilde{b}_{I',i',i}x_{i'}\Bigr)
	=\sum_{i\in I}\tilde{b}_{I,i,j}\tilde{L}_{I',i}(\x)
	=\tilde{L}_{I,j}(\y).
	\end{split}
	\end{equation}
	This means that $\tilde{\L}_{I'}((\R^{d})^{k+1})\subseteq \tilde{\L}_{I}((\R^{d})^{k+1})$. Similarly, $\tilde{\L}_{I}((\R^{d})^{k+1})\subseteq \tilde{\L}_{I'}((\R^{d})^{k+1})$. Thus $\tilde{\L}_{I'}((\R^{d})^{k+1})=\tilde{\L}_{I}((\R^{d})^{k+1})$. This completes the proof of Part (i).

  Part (ii) follows from the knowledge of linear algebra.
  Part (iii) follows from the fact that $(1,\dots,1)=\tilde{b}_{I,0}+\dots+\tilde{b}_{I,k}\in V_{X}$, which follows from the normalization property of $\tilde{b}_{I,i,j}$. Part (iv) is straightforward.
\end{proof}

For technical reasons, in additional to (\ref{4:Vdc0}), in later sections, we also need to study recurrence results of the form
\begin{equation}\label{4:Vdc00}
	\begin{split}
	\E_{(x_{0},\dots,x_{k+s-1})\in \Omega}f_{i}(x_{i})\Bigl\vert \prod_{0\leq j\leq k+s-1, j\neq i}f_{j}(x_{j})\Bigr\vert^{2}.
	\end{split}
	\end{equation}	
This motivates us to study the relative self-products of the set $\Omega$.

 Let $B$ be a set, $J$ be an (ordered) finite index set, $j\in J$, $x_{j}\in B$ and  
 $\x'=(x_{j'})_{j'\in J\backslash\{j\}}\in B^{\vert J\vert-1}$, write
 $(x_{j}\da_{j\ca J} \x'):=(x_{j'})_{j'\in J}\in B^{\vert J\vert}$. In other words, $(x_{j}\da_{j\ca J} \x')$ is the completion of the $J\backslash\{j\}$-indexed vector $\x'$ to a $J$-indexed vector by inserting $x_{j}$ as the $j$-th coordinate.
 When $J=\{0,\dots,k+s-1\}$, we write $(x_{j}\da_{j} \x'):=(x_{j}\da_{j\ca J} \x')$ for short.

For any $0\leq i\leq k+s-1$, denote 
$$\Omega\times_{i}\Omega:=\{(x,\x^{+},\x^{-})\in (\V)\times (\V)^{k+s-1}\times (\V)^{k+s-1}\colon (x\da_{j}\x^{+}), (x\da_{j}\x^{-})\in\Omega\}.$$
It is not hard to see that we may rewrite (\ref{4:Vdc00}) as 
\begin{equation}\nonumber
	\begin{split}
	\E_{(x_{i},\x^{+},\x^{-})\in \Omega\times_{i}\Omega}f_{i}(x_{i})\prod_{0\leq j\leq k+s-1, j\neq i}f_{j}(x^{+}_{j})\overline{f}_{j}(x^{-}_{j}),
	\end{split}
	\end{equation}	
where $\x^{+}=(x^{+}_{j})_{0\leq j\leq k+s-1, j\neq i}$ and $\x^{-}=(x^{-}_{j})_{0\leq j\leq k+s-1, j\neq i}$.

	Let $I$ be a generating set containing $i$. Define
$$\Omega_{I}\times_{i}\Omega_{I}:=\{(x,\x^{+},\x^{-})\in (\V)\times (\V)^{k}\times (\V)^{k}\colon (x\da_{i\ca I}\x^{+}), (x\da_{i\ca I}\x^{-})\in\Omega_{I}\}.$$

\begin{ex}
    Assume that $I=\{0,\dots,k\}$ is a generating set of $X$. Then $\Omega\times_{0}\Omega$ is the set of $(x_{0},\dots,x_{2k+2s-2})\in (\V)^{2k+2s-1}$ such that $(x_{0},x_{1},\dots,x_{k+s-1}), (x_{0},x_{k+s},\dots,x_{2k+2s-2})\in \Omega$, and $\Omega_{I}\times_{0}\Omega_{I}$ is the set of $(x_{0},\dots,x_{2k})\in (\V)^{2k+1}$ such that $(x_{0},x_{1},\dots,x_{k})$ and $(x_{0},x_{k+1},\dots,$ $x_{2k})$ belong to $\Omega_{I}$. For more general sets $\Omega\times_{i}\Omega$ and $\Omega_{I}\times_{i}\Omega_{I}$, we have a similar structure except that the notations are more complicated.
\end{ex}

Recall from Lemma \ref{4:kk2d} that $\Omega=\L_{I}(\Omega_{I})$  for any  generating set $I$ of $X$. We wish to write the set $\Omega\times_{i}\Omega$ in a similar way. 
For all $j\in \{0,\dots,k+s-1\}$, let $L^{+}_{i\ca I,j}\colon (\V)^{2k+1}\to \V$ be the linear map given by	
$$L^{+}_{i\ca I,j}(x,\x^{+},\x^{-}):=L_{I,j}(x\da_{i\ca I}\x^{+}) \text{ for all } (x,\x^{+},\x^{-})\in(\V)\times(\V)^{k}\times (\V)^{k},$$
	and let $L^{-}_{i\ca I,j}\colon (\V)^{2k+1}\to \V$ be the linear map given by	
$$L^{-}_{i\ca I,j}(x,\x^{+},\x^{-}):=L_{I,j}(x\da_{i\ca I}\x^{-}) \text{ for all } (x,\x^{+},\x^{-})\in(\V)\times(\V)^{k}\times (\V)^{k}.$$
Let
$\L_{i\ca I}^{\pm}\colon (\V)^{2k+1}\to (\V)^{2k+2s-1}$ be the linear map given by	
$$\L^{\pm}_{i\ca I}(\x):=(L^{+}_{i\ca I,i}(\x),(L^{+}_{i\ca I,j}(\x))_{j\in \{0,\dots,k+s-1\}\backslash\{i\}},(L^{-}_{i\ca I,j}(\x))_{j\in \{0,\dots,k+s-1\}\backslash\{i\}})$$
for all $\x\in (\V)^{2k+1}$ (note that $L^{+}_{i\ca I,i}(x,\x^{+},\x^{-})=L^{-}_{i\ca I,i}(x,\x^{+},\x^{-})=x$ for all $i\in I$).
 
For all $i\in I$ and $i'\in I\backslash\{i\}$,
  write 
  $$b^{\pm}_{i\ca I,i}:=(b_{I,i,i},\b'_{I,i},\b'_{I,i}), b^{+}_{i\ca I,i'}:=(b_{I,i',i},\b'_{I,i'},0,\dots,0)
 \text{ and }
  b^{-}_{i\ca I,i'}:=(b_{I,i',i},0,\dots,0,\b'_{I,i'}),$$
  where 
  $$\b'_{I,i''}:=(b_{I,i'',0},\dots,b_{I,i'',i-1},b_{I,i'',i+1},\dots,b_{I,i'',k+s-1})$$
  for all $i''\in I$
  and 
  0 appears $k+s-1$ times in each expression. 
  It is not hard to see that 
 $$L_{i\ca I,j}^{+}(x_{i},(x_{i'}^{+})_{i'\in I\backslash\{i\}},(x_{i'}^{-})_{i'\in I\backslash\{i\}}):=b^{\pm}_{i\ca I,i,j}x_{i}+\sum_{i'\in I\backslash\{i\}}b^{+}_{i\ca I,i',j}x_{i'}^{+}$$
 and 
 $$L_{i\ca I,j}^{-}(x_{i},(x_{i'}^{+})_{i'\in I\backslash\{i\}},(x_{i'}^{-})_{i'\in I\backslash\{i\}}):=b^{\pm}_{i\ca I,i,j}x_{i}+\sum_{i'\in I\backslash\{i\}}b^{-}_{i\ca I,i',j}x_{i'}^{-}$$
 for all
 $i\in I$ and $0\leq j\leq k+s-1$,
 where $b^{\pm}_{i\ca I,i,j}$ is the $j$-th coordinate of $b^{\pm}_{i\ca I,i}$.
The normalization properties of $b_{I,i,j}$ and $L_{I,j}$ imply the following \emph{normalization properties} of $L^{\pm}_{i\ca I,j}$:
$$
\sum_{i'\in I}b^{\pm}_{i\ca I,i',j}=1 \text{ for all $i\in I, j\in J$ and }
L^{\pm}_{i\ca I,j}(x,\dots,x)=x \text{ for all $i\in I, j\in J$}.$$

We summarize some properties of the sets $\Omega\times_{i}\Omega$ and $\Omega_{I}\times_{i}\Omega_{I}$ for later uses.

\begin{lem}\label{4:kk3d}
Let $I\subseteq \{0,\dots,k+s-1\}$ be a generating set of $X$ and $i\in I$.
\begin{enumerate}[(i)]
	\item We have that $\Omega\times_{i}\Omega=\L_{i\ca I}^{\pm}(\Omega_{I}\times_{i}\Omega_{I})$. 
	\item If $d\geq 2k^{2}+6k+3,$ then $\Omega_{I}\times_{i}\Omega_{I}$  is a nice and consistent $M$-set of total co-dimension $k^{2}+3k+1$, and we have that $\vert \Omega_{I}\times_{i}\Omega_{I}\vert=p^{d(2k+1)-(k^{2}+3k+1)}(1+O_{k}(p^{-1/2}))$.
	\item If $d\geq \max\{4k^{2}+12k+7, 4k+s+13\}$, then for all $C'>0$, there exists $K=O_{C',d}(1)$ such that $\Omega_{I}\times_{i}\Omega_{I}$ admits a partially periodic  $(t,Kt^{-K})$-Leibman dichotomy up to step $s$ and complexity at most $C'$ for all $0<t<1/2$.
\end{enumerate}	
\end{lem}
\begin{proof}
   We first prove Part (i). Let $(x,\x^{+},\x^{-})\in \Omega\times_{i}\Omega$.
    Since $\Omega=\L_{I}(\Omega_{I})$ by Lemma \ref{4:kk2d} (i), we have that $(x\da_{i\ca I}\x^{+})=\L_{I}(\y^{+}), (x\da_{i\ca I}\x^{-})=\L_{I}(\y^{-})$ for some $\y^{+}=(y_{j}^{+})_{j\in I},\y^{-}=(y_{j}^{-})_{j\in I}\in \Omega_{I}$. Since $i\in I$, we have that $y_{i}^{+}=L_{I,i}(\y^{+})=x=L_{I,i}(\y^{-})=y_{i}^{-}$. Therefore, we have that $(y_{i}^{+},(y_{j}^{+})_{j\in I\backslash\{i\}},(y_{j}^{-})_{j\in I\backslash\{i\}})\in\Omega_{I}\times_{i}\Omega_{I}$. So
    $$(x,\x^{+},\x^{-})=\L_{i\ca I}^{\pm}(y_{i}^{+},(y_{j}^{+})_{j\in I\backslash\{i\}},(y_{j}^{-})_{j\in I\backslash\{i\}})\in \L_{i\ca I}^{\pm}(\Omega_{I}\times_{i}\Omega_{I}).$$
     
    Conversely, let $(x,\x^{+},\x^{-})=\L_{i\ca I}^{\pm}(y_{i},\y^{+},\y^{-})$ for some $$(y_{i},\y^{+},\y^{-}):=(y_{i},(y_{j}^{+})_{j\in I\backslash\{i\}},(y_{j}^{-})_{j\in I\backslash\{i\}})\in\Omega_{I}\times_{i}\Omega_{I}.$$ By definition, we have that $(y_{i}\da_{i\ca I}\y^{+})$ and $(y_{i}\da_{i\ca I}\y^{-})$ belong to $\Omega_{I}$, where we denote $y_{i}^{+}:=y_{i}^{-}:=y_{i}$. Then
    \begin{equation}\nonumber
    \begin{split}
      &\quad (x\da_{i\ca J} \x^{+})
      =(L^{+}_{i\ca I,i}(y_{i},\y^{+},\y^{-})\da_{i\ca J} (L^{+}_{i\ca I,j}(y_{i},\y^{+},\y^{-}))_{j\in J\backslash\{i\}})
      \\&=(L^{+}_{i\ca I,j}(y_{i},\y^{+},\y^{-}))_{j\in J}
      =(L_{I,j}(y_{i}\da_{i\ca I}\y^{+}))_{j\in J}=\L_{I}(y_{i}\da_{i\ca I}\y^{+})\in\L_{I}(\Omega_{I})=\Omega,
    \end{split}
    \end{equation}
  where the last equality follows from  Lemma \ref{4:kk2d} (i).
 Similarly,   $(x\da_{i\ca J} \x^{-})\in \Omega$. By definition, this implies that 
 $(x,\x^{+},\x^{-})\in \Omega\times_{i}\Omega$. This complete the proof of Part (i).
 
 We now prove Part (ii).
  Note that $\Omega_{I}\times_{i}\Omega_{I}$ is the set $\Omega_{4}$ in Example 
  B.4 of \cite{SunA} (up to a relabelling of the variables).    
	 So $\Omega_{I}$ 	 is a nice and consistent $M$-set of total co-dimension $k^{2}+3k+1$.
	  Since $d\geq 2k^{2}+6k+3,$ by Theorem \ref{4:ct},  $\vert \Omega_{I}\times_{i}\Omega_{I}\vert=p^{d(2k+1)-(k^{2}+3k+1)}(1+O_{k}(p^{-1/2}))$.
	  
	  Part (iii) follows from  Theorem \ref{4:veryr} and Part (ii).
\end{proof}

We now lift the above mentioned notations to the $\Z$-setting.
Recall that 
 $\tilde{b}_{I,i,j}:=\tau(b_{I,i,j})\in\Z$. 
 For all $i,i'\in I$ and $j\in J$,
  write $\tilde{b}^{\pm}_{i\ca I,i',j}:=\tau(b^{\pm}_{i\ca I,i',j})\in\Z$ and $\tilde{b}^{\pm}_{i\ca I,i'}:=\tau(b^{\pm}_{i\ca I,i'})\in\Z^{2k+2s-1}$.
  For $i\in I$ and $0\leq j\leq k+s-1$,
let $\tilde{L}^{\pm}_{i\ca I,j}\colon (\R^{d})^{2k+1}\to\R^{d}$ be the 
linear linear transformation given by
 $$\tilde{L}_{i\ca I,j}^{+}(x_{i},(x_{i'}^{+})_{i'\in I\backslash\{i\}},(x_{i'}^{-})_{i'\in I\backslash\{i\}}):=\tilde{b}^{\pm}_{i\ca I,i,j}x_{i}+\sum_{i'\in I\backslash\{i\}}\tilde{b}^{+}_{i\ca I,i',j}x_{i'}^{+}$$
 and 
 $$\tilde{L}_{i\ca I,j}^{-}(x_{i},(x_{i'}^{+})_{i'\in I\backslash\{i\}},(x_{i'}^{-})_{i'\in I\backslash\{i\}}):=\tilde{b}^{\pm}_{i\ca I,i,j}x_{i}+\sum_{i'\in I\backslash\{i\}}\tilde{b}^{-}_{i\ca I,i',j}x_{i'}^{-}.$$
Let $\tilde{\L}^{\pm}_{i\ca I}\colon (\R^{d})^{2k+1}\to(\R^{d})^{2k+2s-1}$ be the linear map given by
 $$\tilde{\L}^{\pm}_{i\ca I}(\x):=(\tilde{L}^{+}_{i\ca I,i}(\x),(\tilde{L}^{+}_{i\ca I,j}(\x))_{j\in \{0,\dots,k+s-1\}\backslash\{i\}},(\tilde{L}^{-}_{i\ca I,j}(\x))_{j\in \{0,\dots,k+s-1\}\backslash\{i\}})$$
for all $\x\in (\R^{d})^{2k+1}$.
 It is not hard to see that $L^{\pm}_{i\ca I,j}=\iota\circ \tilde{L}^{\pm}_{i\ca I,j}\circ \tau$. In other words, $\frac{1}{p}\tilde{L}^{\pm}_{i\ca I,j}$ is the regular lifting of $L^{\pm}_{i\ca I,j}$.
 Let $C_{0}$ be the complexity of $X$.
 Note that if $p\gg_{C_{0},k,s} 1$, then the normalization property of $L^{\pm}_{i\ca I,j}$ imply the following \emph{normalization property} of   $\tilde{L}^{\pm}_{i\ca I,j}$:
  $$\sum_{i'\in I}\tilde{b}^{\pm}_{i\ca I,i',j}=1 \text{ for all $i\in I, j\in J$ and }
\tilde{L}^{\pm}_{i\ca I,j}(x,\dots,x)=x \text{ for all $i\in I, j\in J$}.$$

Recall that $V_{X}$ is a rational subspace of $\R^{k+s}$ defined in Section \ref{4:s:lbg}.
 For any $0\leq i\leq k+s-1$, denote 
$$V_{X}\times_{i}V_{X}:=\{(x,\x^{+},\x^{-})\in \R\times \R^{k+s-1}\times \R^{k+s-1}\colon (x\da_{i}\x^{+}), (x\da_{i}\x^{-})\in V_{X}\}.$$

\begin{lem}\label{4:vxxprop}
Assume that the complexity of $X$ is at most $C_{0}$ and $p\gg_{C_{0},k,s} 1$. Let $0\leq i\leq k+s-1$.
\begin{enumerate}[(i)]
  \item The set $V_{X}\times_{i}V_{X}$ is a rational subspace of $\R^{2k+2s-1}$.
  \item We have that $V_{X}\times_{i}V_{X}=\tilde{\L}^{\pm}_{i\ca I}((\R^{d})^{2k+1})=\sp_{\R}\{\tilde{b}^{\pm}_{i\ca I,i'}, i'\in I\}$ for any generating set $I$ of $X$ containing $i$.
  \item For all $j\in\N_{+}$, we have that $(V_{X}\times_{i}V_{X})^{[j]}=V_{X}^{[j]}\times_{i}V_{X}^{[j]}$.
  \item The vectors $(1,\dots,1)$ belongs to $V_{X}\times_{i}V_{X}$. Therefore, we have the \emph{flag property} $(V_{X}\times_{i}V_{X})^{[1]}\leq (V_{X}\times_{i}V_{X})^{[2]}\leq \dots.$
  \item The complexities of all of $(V_{X}\times_{i}V_{X})^{[j]}, 1\leq j\leq s-1$ are at most $O_{C_{0},k,s}(1)$.
\end{enumerate}
\end{lem}
\begin{proof}
Parts (i), (iii) and (v) are straightforward. Part (iv) follows from the normalization property $\sum_{i'\in I}\tilde{b}^{\pm}_{i\ca I,i',j}=1$ for all $i\in I, j\in J$.
So it remains to  prove Part (ii). The fact that $\tilde{\L}^{\pm}_{i\ca I}((\R^{d})^{2k+1})=\sp_{\R}\{\tilde{b}^{\pm}_{i\ca I,i'}, i'\in I\}$ follows from the the knowledge of linear algebra and the definition of $\tilde{L}_{i\ca I,j}^{\pm}$. So it remains to show that $V_{X}\times_{i}V_{X}=\sp_{\R}\{\tilde{b}^{\pm}_{i\ca I,i'}, i'\in I\}$.

By definition, for all $i'\in I\backslash\{i\}$,
  we have that
  $$\tilde{b}^{\pm}_{i\ca I,i}:=(\tilde{b}_{I,i,i},\tilde{\b}'_{I,i},\tilde{\b}'_{I,i}), \tilde{b}^{+}_{i\ca I,i'}:=(\tilde{b}_{I,i',i},\tilde{\b}'_{I,i'},0,\dots,0)
 \text{ and }
  \tilde{b}^{-}_{i\ca I,i'}:=(\tilde{b}_{I,i',i},0,\dots,0,\tilde{\b}'_{I,i'}),$$
  where 
  $$\tilde{\b}'_{I,i''}:=(\tilde{b}_{I,i'',0},\dots,\tilde{b}_{I,i'',i-1},\tilde{b}_{I,i'',i+1},\dots,\tilde{b}_{I,i'',k+s-1})$$
  for all $i''\in I$
  and 
  0 appears $k+s-1$ times in each expression. By Lemma \ref{4:vxprop} (ii), we have that $(\tilde{b}_{I,i'',i}\da_{i}\tilde{\b}'_{I,i''})=\tilde{b}_{I,i''}\in V_{X}$ for all $i''\in I$. For all $i'\in I\backslash\{i\}$, since $\tilde{b}_{I,i',i}=0$, we have that   
$(\tilde{b}_{I,i',i}\da_{i} 0,\dots,0)=\bold{0}\in V_{X}$. From this we can conclude that $\tilde{b}^{\pm}_{i\ca I,i'}\in V_{X}\times_{i}V_{X}$ for all $i'\in I$. So $\sp_{\R}\{\tilde{b}^{\pm}_{i\ca I,i'},i'\in I\}\subseteq V_{X}\times_{i}V_{X}$.

Conversely for any $(x,\x^{+},\x^{-})\in V_{X}\times_{i} V_{X}$, since $(x\da_{j}\x^{+}), (x\da_{j}\x^{-})\in V_{x}$, by Lemma \ref{4:vxprop} (ii), we may write
$$(x\da_{j}\x^{+})=\sum_{i'\in I}c^{+}_{i'}\tilde{b}_{I,i'} \text{ and } (x\da_{j}\x^{-})=\sum_{i'\in I}c^{-}_{i'}\tilde{b}_{I,i'}$$
for some $c^{\pm}_{i'}\in\R$. Since $\tilde{b}_{I,i',i}=\delta_{i,i'}$ for all $i'\in I$, projecting to the $i$-th coordinate, we have that $x=c_{i}^{+}=c_{i}^{-}$. Therefore, it is not hard to see that
$$(x,\x^{+},\x^{-})=x\tilde{b}^{\pm}_{i\ca I,i}+\sum_{i'\in I\backslash\{i\}}\Bigl(c_{i'}^{+}\tilde{b}^{+}_{i\ca I,i'}+c_{i'}^{-}\tilde{b}^{-}_{i\ca I,i'}\Bigr)\in \sp_{\R}\{\tilde{b}^{\pm}_{i\ca I,i'},i'\in I\}.$$
So $V_{X}\times_{i}V_{X}\subseteq \sp_{\R}\{\tilde{b}^{\pm}_{i\ca I,i'},i'\in I\}$ and thus $V_{X}\times_{i}V_{X}=\sp_{\R}\{\tilde{b}^{\pm}_{i\ca I,i'},i'\in I\}$.
\end{proof}

\subsection{The joint equidistribution theorem}

We now provide an equidistribution property pertaining to both (\ref{4:Vdc0}) and (\ref{4:Vdc00}).
We begin with recalling the concept of Leibman group (first studied in \cite{Lei10}) which plays an important role in \cite{GT10b}.

Let $t\in\N_{+}$ and $G$ be a nilpotent group with filtration $(G_{i})_{i\in\N}$ and $V$ be a rational subspace of $\R^{t}$ containing $(1,\dots,1)$. 
For $g\in G$ and $w=(w_{0},\dots,w_{t-1})\in\R^{t}$, denote $$g^{w}:=(\exp(w_{0}\log(g)),\dots,\exp(w_{t-1}\log(g)))\in G^{t}.$$
Let $G^{V}$ be the subgroup of $G^{t}$ generated by the elements of the form $g^{w}$, where $g$ is an iterated commutator of $g_{1},\dots,g_{t}$ for some $t\in\N_{+}$, $g_{j}\in G_{i_{j}}$ for some $i_{j}\in \N_{+}$ for all $1\leq j\leq t$, and $w=w_{1}\ast\ldots\ast w_{t}$ for some $w_{j}\in V^{[i_{j}]}$ for all $1\leq j\leq t$, where $\ast$ is the coordinate-wise product. 
  Similar to the argument in \cite{GT10b}, one can show that $G^{V}$ is a normal and rational subgroup of $G^{t}$, and thus $\Gamma^{V}:=\Gamma^{t}\cap G^{V}$ is a discrete cocompact subgroup of $G^{V}$.

  For $i\in\N_{+}$, let $G^{V}_{(i)}$ be the subgroup of $G^{V}$ generated by the elements of the form $g^{w}$, where $g$ is an iterated commutator of $g_{1},\dots,g_{t}$ for some $t\in\N_{+}$, $i_{1}+\dots+i_{t}\geq i$, $g_{j}\in G_{i_{j}}$ for some $i_{j}\geq i$ for all $1\leq j\leq t$, and $w=w_{1}\ast\ldots\ast w_{t}$ for some $w_{j}\in V^{[i_{j}]}$ for all $1\leq j\leq t$.
     Set $G_{(0)}^{V}:=G^{V}$.
     Since $(1,\dots,1)\in V$,
   similar to the proof of Lemma 3.5 of \cite{GT10b}, one may use the formula (\ref{4:C2})   to show that $(G_{(i)}^{V})_{i\in\N}$ is a filtration of $G^{V}$. 

We are now ready to state our main equidistribution result.

\begin{thm}[Joint equidistribution theorem]\label{4:orbitdesc2}
	Let $C_{0},C,\d>0$. Assume that $X$ has complexity at most $C_{0}$ and  
	$d\geq \max\{4k^{2}+12k+7,4k+s+13,3k+2s,s^{2}-s+3\}.$\footnote{Note that this lower bound of $d$ is controlled by $d_{0}(X)$ defined in (\ref{4:finald}).} 	 Let   $(G/\Gamma)_{\N}$ be an $\N$-filtered nilmanifold of step $s-1$ and complexity at most $C$. 
	There exists $N_{0}:=N_{0}(C_{0},C,\d,d)>0$ such that if $p\gg_{C_{0},C,\d,d} 1$, then for all $N\geq N_{0}$ and   
	$(N,V(M-r))$-irrational $g\in\poly_{p}(V(M-r)\to G_{\N}\vert\Gamma)$,
	\begin{enumerate}[(i)]
		\item the sequence
		\begin{equation}\nonumber
			(g(x_{0})\Gamma,\dots,g(x_{k+s-1})\Gamma)_{(x_{0},\dots,x_{k+s-1})\in\Omega}
		\end{equation}
		is $2\d$-equidistributed  on $G^{V_{X}}/\Gamma^{V_{X}}$; 
		\item the sequence
		\begin{equation}\nonumber
			(g(x_{0})\Gamma,\dots,g(x_{2k+2s-2})\Gamma)_{(x_{0},\dots,x_{2k+2s-2})\in\Omega\times_{j}\Omega}
		\end{equation}
		is $\d$-equidistributed  on $G^{V_{X}\times_{j}V_{X}}/\Gamma^{V_{X}\times_{j}V_{X}}$ for all $0\leq j\leq k+s-1$.  
	\end{enumerate}	
\end{thm}	

The rest of this section is devoted to the proof of Theorem \ref{4:orbitdesc2}. 
For convenience, when proving Part (ii) of Theorem \ref{4:orbitdesc2}, we only consider the case $j=0$ since the other cases are similar.
By Lemma \ref{4:lemgg} (ii), there exists a generating set $I$ of $X$ containing 0. We may assume without loss of generality that $I=\{0,\dots,k\}$.

		We may write $g=\tilde{g}\circ \tau$ for some $\tilde{g}\in\poly_{p}(\iota^{-1}(V(M-r))\to G_{\N}\vert\Gamma)$, and expand $\tilde{g}$ by its type-I Taylor expansion
	\begin{equation}\label{4:taylor4g}
	\tilde{g}(n)=\prod_{j\in\N^{d},0\leq \vert j\vert\leq s-1}\tilde{g}_{j}^{\binom{n}{j}}
	\end{equation}
		for some $\tilde{g}_{j}\in G_{j}$. 	It is convenience to rewrite $\tilde{\L}^{\pm}_{0\ca I}$ as 
	$$\L^{\pm}_{0\ca I}(\x):=(L'_{0}(\x),\dots,L'_{2k+2s-2}(\x))$$
	for all $\x\in(\F_{p}^{d})^{2k+1}$, and to 
	  rewrite $\tilde{\L}^{\pm}_{0\ca I}$ as 
	$$\tilde{\L}^{\pm}_{0\ca I}(\x):=(\tilde{L}'_{0}(\x),\dots,\tilde{L}'_{2k+2s-2}(\x))$$
	for all $\x\in(\R^{d})^{2k+1}$.
	 Then by (\ref{4:taylor4g}),
	\begin{equation}\label{4:taylor4g2}
	\tilde{g}^{\otimes (2k+2s-1)}\circ \tilde{\L}^{\pm}_{0\ca I}(\x)=\prod_{j\in\N^{d},0\leq \vert j\vert\leq s-1}\Bigl(\tilde{g}_{j}^{\binom{\tilde{L}'_{0}(\x)}{j}},\dots,\tilde{g}_{j}^{\binom{\tilde{L}'_{2k+2s-2}(\x)}{j}}\Bigr)
	\end{equation}
	for all $\x\in (\Z^{d})^{2k+1}$. We begin with some properties for the polynomial sequence described in (\ref{4:taylor4g2}).

		\begin{lem}\label{4:piscst2}
	 We have that $$\tilde{g}^{\otimes (2k+2s-1)}\circ\tilde{\L}^{\pm}_{0\ca I}\in\poly_{p}(\iota^{-1}(\Omega_{I}\times_{0}\Omega_{I})\to G^{2k+2s-1}_{\N}\vert\Gamma^{2k+2s-1}).$$ Moreover, we have that $$\tilde{g}^{\otimes (2k+2s-1)}\circ\tilde{\L}^{\pm}_{0\ca I}(\tau(\x))\Gamma^{2k+2s-1}=g^{\otimes (2k+2s-1)}\circ\L^{\pm}_{0\ca I}(\x)\Gamma^{2k+2s-1}$$ for all $\x\in \Omega_{I}\times_{0}\Omega_{I}$.
	\end{lem} 
	\begin{proof}	
		Using the flag property of $V_{X}\times_{0}V_{X}$ and expanding the linear transformations $\tilde{L}'_{j}$, it is not hard to see that the map $\x\mapsto \Bigl(\tilde{g}_{j}^{\binom{\tilde{L}'_{0}(\x)}{j}},\dots,\tilde{g}_{j}^{\binom{\tilde{L}'_{2k+2s-2}(\x)}{j}}\Bigr)$ belongs to $\poly((\Z^{d})^{2k+1}\to G_{\N}^{V_{X}\times_{0}V_{X}})$ for all $j\in\N^{d}$.  So by Corollary B.4 of \cite{GTZ12} and (\ref{4:taylor4g2}), $\tilde{g}^{\otimes (2k+2s-1)}\circ \tilde{\L}^{\pm}_{0\ca I}$ also belongs to $\poly((\Z^{d})^{2k+1}\to G_{\N}^{V_{X}\times_{0}V_{X}})$. 
		
		Now fix $0\leq i\leq 2k+2s-1$, $\x\in \iota^{-1}(\Omega_{I}\times_{0}\Omega_{I})$ and $\y\in (\Z^{d})^{2k+1}$. 
		Since $\Omega\times_{0}\Omega\subseteq V(M-r)^{2k+2s-1}$, it follows from Lemma \ref{4:kk3d} (i) that $L'_{i}(\Omega_{I}\times_{0}\Omega_{I})\subseteq V(M-r)$.
		Since $\tilde{L}'_{i}\circ\tau\equiv\tau\circ L'_{i} \mod p(\Z^{d})^{2k+1}$, we have that  
	$$\tilde{L}'_{i}(\x)\in \tilde{L}'_{i}(\iota^{-1}(\Omega_{I}\times_{0}\Omega_{I}))\subseteq \iota^{-1}\circ L'_{i}(\Omega_{I}\times_{0}\Omega_{I})\subseteq \iota^{-1}(V(M-r)).$$	
		  Since
		  $\tilde{L}'_{i}(\y)\in\Z^{d}$ and $\tilde{g}\in\poly_{p}(\iota^{-1}(V(M-r))\to G_{\N}\vert\Gamma)$, we have that 
		$$\tilde{g}(\tilde{L}'_{i}(\x+p\y))^{-1}\tilde{g}(\tilde{L}'_{i}(\x))=\tilde{g}(\tilde{L}'_{i}(\x)+p\tilde{L}'_{i}(\y))^{-1}\tilde{g}(\tilde{L}'_{i}(\x))\in\Gamma.$$
		Therefore, $$\tilde{g}^{\otimes (2k+2s-1)}(\tilde{\L}^{\pm}_{0\ca I}(\x+p\y))^{-1}\tilde{g}^{\otimes (2k+2s-1)}(\tilde{\L}^{\pm}_{0\ca I}(\x))\in\Gamma^{2k+2s-1}.$$ This implies that   $\tilde{g}^{\otimes (2k+2s-1)}\circ\tilde{\L}^{\pm}_{0\ca I}$ belongs to $\poly_{p}(\iota^{-1}(\Omega_{I}\times_{0}\Omega_{I})\to G^{2k+2s-1}_{\N}\vert\Gamma^{2k+2s-1})$.  Since we have shown that $\tilde{g}^{\otimes (2k+2s-1)}\circ \tilde{\L}^{\pm}_{0\ca I}\in\poly((\Z^{d})^{2k+1}\to G_{\N}^{V_{X}\times_{0}V_{X}})$, this implies that   $\tilde{g}^{\otimes (2k+2s-1)}\circ \tilde{\L}^{\pm}_{0\ca I}$ belongs to $$\poly_{p}(\iota^{-1}(\Omega_{I}\times_{0}\Omega_{I})\to G_{\N}^{V_{X}\times_{0}V_{X}}\vert \Gamma^{V_{X}\times_{0}V_{X}})\subseteq \poly_{p}(\iota^{-1}(\Omega_{I}\times_{0}\Omega_{I})\to G^{2k+2s-1}_{\N}\vert\Gamma^{2k+2s-1}).$$	
		
		\
		
		Now for any $0\leq i\leq 2k+2s-1$ and $\x\in \Omega_{I}\times_{0}\Omega_{I}$, since  
		$\tilde{L}'_{i}\circ \tau (\x)\equiv \tau\circ L'_{i}(\x) \mod p(\Z^{d})^{2k+1}$, $\tilde{g}\in\poly_{p}(\iota^{-1}(V(M-r))\to G_{\N}\vert\Gamma)$ and $\tau\circ L'_{i}(\x)\in \tau(L'_{i}(\Omega_{I}\times_{0}\Omega_{I}))\subseteq \tau(V(M-r))$ (as we have shown earlier in the proof), we have that 
		$$\tilde{g}\circ\tilde{L}'_{i}\circ \tau (\x)\Gamma=\tilde{g}\circ\tau\circ\tilde{L}_{i} (\x)\Gamma=g\circ\tilde{L}_{i} (\x)\Gamma.$$
		This implies that $$\tilde{g}^{\otimes (2k+2s-1)}\circ\tilde{\L}^{\pm}_{0\ca I}(\tau(\x))\Gamma^{2k+2s-1}=g^{\otimes (2k+2s-1)}\circ\L^{\pm}_{0\ca I}(\x)\Gamma^{2k+2s-1}$$ for all $\x\in \Omega_{I}\times_{0}\Omega_{I}$.	
		\end{proof}	

Assume that 		
the sequence
		\begin{equation}\nonumber
			(g(x_{0})\Gamma,\dots,g(x_{2k+2s-2})\Gamma)_{(x_{0},\dots,x_{2k+2s-2})\in\Omega\times_{0}\Omega}
		\end{equation}
		is not $\d$-equidistributed  on $G^{V_{X}\times_{0}V_{X}}/\Gamma^{V_{X}\times_{0}V_{X}}$. Then by 
		Lemma \ref{4:kk3d} (i), the sequence
		\begin{equation}\nonumber
			(g^{\otimes (2k+2s-1)}\circ\L^{\pm}_{0\ca I}(\x)\Gamma^{2k+2s-1})_{\x\in\Omega_{I}\times_{0}\Omega_{I}}
		\end{equation}
		is not $\d$-equidistributed  on $G^{V_{X}\times_{0}V_{X}}/\Gamma^{V_{X}\times_{0}V_{X}}$. By Lemma \ref{4:piscst2}, 
	 the sequence 
	 		\begin{equation}\nonumber
			(\tilde{g}^{\otimes (2k+2s-1)}\circ\tilde{\L}^{\pm}_{0\ca I}(\x)\Gamma^{2k+2s-1})_{\x\in\tau(\Omega_{I}\times_{0}\Omega_{I})}
		\end{equation}
		is not $\d$-equidistributed  on $G^{V_{X}\times_{0}V_{X}}/\Gamma^{V_{X}\times_{0}V_{X}}$, or equivalently, the sequence
		\begin{equation}\nonumber
			(\tilde{g}^{\otimes (2k+2s-1)}\circ\tilde{\L}^{\pm}_{0\ca I}\circ\tau(\x)\Gamma^{2k+2s-1})_{\x\in \Omega_{I}\times_{0}\Omega_{I}}
		\end{equation}
		is not $\d$-equidistributed  on $G^{V_{X}\times_{0}V_{X}}/\Gamma^{V_{X}\times_{0}V_{X}}$. Similar to Lemma 3.7 of \cite{GT10b}, our strategy is to obtain a contradiction by showing that $g$ is not sufficiently irrational via the construction of some suitable type-III horizontal character. However,  our case is more complicated, and so we need to use some  tools developed in the previous part of the series \cite{SunA}.  
		
		By Lemma \ref{4:piscst2}, we have that
$\tilde{g}^{\otimes (2k+2s-1)}\circ\tilde{\L}^{\pm}_{0\ca I}\in\poly_{p}(\iota^{-1}(\Omega_{I}\times_{0}\Omega_{I})\to G^{2k+2s-1}_{\N}\vert\Gamma^{2k+2s-1})$ and thus by definition $\tilde{g}^{\otimes (2k+2s-1)}\circ\tilde{\L}^{\pm}_{0\ca I}\circ\tau\in\poly_{p}(\Omega_{I}\times_{0}\Omega_{I}\to G^{2k+2s-1}_{\N}\vert\Gamma^{2k+2s-1})$.	
By Lemma \ref{4:vxprop} (iv) and Lemma \ref{4:vxxprop} (v),
		the complexities of $G^{V_{X}}/\Gamma^{V_{X}}$ and $G^{V_{X}\times_{0}V_{X}}/\Gamma^{V_{X}\times_{0}V_{X}}$ are bounded by some $C'=O_{C_{0},C,k,s}(1)$.  	
	Finally, since  $d\geq \max\{4k^{2}+12k+7, 4k+s+13\}$, it follows from Lemma \ref{4:kk3d} (iii) that  
	$\Omega_{I}\times_{0}\Omega_{I}$ admits a partially periodic  $(\d,O_{C_{0},C,\d,d}(1))$-Leibman dichotomy up to step $s$ and complexity at most $C'$ for all $0<\d<1/2$. 
	So there exists a nontrivial type-I horizontal character $\eta\colon G^{V_{X}\times_{0}V_{X}}\to\R$ of complexity at most $O_{C_{0},C,\d,d}(1)$ such that 
	$\eta\circ \tilde{g}^{\otimes (2k+2s-1)}\circ\tilde{\L}^{\pm}_{0\ca I}\circ\tau \mod\Z$
	is  a constant  on $\Omega_{I}\times_{0}\Omega_{I}$, or equivalently, 
	$\eta\circ \tilde{g}^{\otimes (2k+2s-1)}\circ\tilde{\L}^{\pm}_{0\ca I}\mod\Z$
	is  a constant  on $\tau(\Omega_{I}\times_{0}\Omega_{I})$. Since   $\tilde{g}^{\otimes (2k+2s-1)}\circ\tilde{\L}^{\pm}_{0\ca I}\in\poly_{p}(\tau(\Omega_{I}\times_{0}\Omega_{I})\to G^{2k+2s-1}_{\N}\vert\Gamma^{2k+2s-1})$, this implies that $\eta\circ \tilde{g}^{\otimes (2k+2s-1)}\circ\tilde{\L}^{\pm}_{0\ca I}\mod\Z$
	is  a constant  on $\iota^{-1}(\Omega_{I}\times_{0}\Omega_{I})$.
	 
	Denote
	$$P:=\eta\circ\tilde{g}^{\otimes (2k+2s-1)}\circ \tilde{\L}^{\pm}_{0\ca I}.$$

	By (\ref{4:taylor4g2}),
	\begin{equation}\nonumber
	P(\x)=\eta\circ\tilde{g}^{\otimes (2k+2s-1)}\circ \tilde{\L}^{\pm}_{0\ca I}(\x)=\prod_{j\in\N^{d},0\leq \vert j\vert\leq s-1}\eta\Bigl(\tilde{g}_{j}^{\binom{\tilde{L}'_{0}(\x)}{j}},\dots,\tilde{g}_{j}^{\binom{\tilde{L}'_{2k+2s-2}(\x)}{j}}\Bigr)
	\end{equation}
	for all $\x\in (\Z^{d})^{2k+1}$.  
Let $1\leq i\leq s-1$ be the largest integer such that $\eta\vert_{G_{(i)}^{V_{X}\times_{0}V_{X}}}$ is nontrivial. Then we may rewrite $P$ as 
	\begin{equation}\label{4:thisisp}
	P(\x)=\prod_{j\in\N^{d},0\leq \vert j\vert\leq i}\eta\Bigl(\tilde{g}_{j}^{\binom{\tilde{L}'_{0}(\x)}{j}},\dots,\tilde{g}_{j}^{\binom{\tilde{L}'_{2k+2s-2}(\x)}{j}}\Bigr).
	\end{equation}
  
 In order make use of the fact that $P \mod\Z$ is a  constant  on $\iota^{-1}(\Omega_{I}\times_{0}\Omega_{I})$, we compute the partial derivatives of $P$. 
 For any
	 $\bold{a}=(a_{0},\dots,a_{2k})\in\N^{2k+1}$ with  $a_{0}+\dots+a_{2k}=i$, let $W(\bold{a})$ denote the set of all 
	$(\x,(h_{t,t'})_{0\leq t\leq 2k, 1\leq t'\leq a_{t}})\in(\Z^{d})^{2k+i+1}, \x=(x_{0},\dots,x_{2k})\in (\Z^{d})^{2k+1}, h_{t,t'}\in \Z^{d}$ such that for all $\e_{t,t'}\in\{0,1\}$,
	$0\leq t\leq 2k$, $1\leq t'\leq a_{t}$, we have
	$$\Bigl(x_{0}+\sum_{t'=1}^{a_{0}}\e_{0,t'}h_{0,t'},\dots,x_{2k}+\sum_{t'=1}^{a_{2k}}\e_{2k,t'}h_{2k,t'}\Bigr)\in\iota^{-1}(\Omega_{I}\times_{0}\Omega_{I}).$$
	We first show that the ``projection" of the set $W(\bold{a})$ to the $h_{t,t'}$ coordinates has a nice algebraic description. With a slight abuse of notation, we use $A$ to denote the matrix associated to both $M-r$ and its regular lifting.

	\begin{lem} \label{4:lemhtt} 
	If $d\geq 3k+2s$, then
	for any $h_{t,t'}\in\Z^{d}, 0\leq t\leq 2k, 1\leq t'\leq a_{t}$ which are $p$-linearly independent (recall Section \ref{4:s:defn} for the definition) and pairwise $(M,p)$-orthogonal (meaning that $p\vert (h_{t,t'}A)\cdot h_{\tilde{t},\tilde{t}'}$ for $(t,t')\neq (\tilde{t},\tilde{t}')$),  there exists $\x\in (\Z^{d})^{2k+1}$ such that $(\x,(h_{t,t'})_{0\leq t\leq 2k, 1\leq t'\leq a_{t}})$ belongs to $W(\bold{a})$.
	\end{lem}
	\begin{proof} 
	It is more convenient to prove Lemma \ref{4:lemhtt} under the $\V$-setting. Indeed,
	it suffices to show that for any $h_{t,t'}\in\V, 0\leq t\leq 2k, 1\leq t'\leq a_{t}$ which are  linearly independent and pairwise  $M$-orthogonal (meaning that $(h_{t,t'}A)\cdot h_{\tilde{t},\tilde{t}'}=0$ for $(t,t')\neq (\tilde{t},\tilde{t}')$),  there exists $\y\in(\V)^{k+1}$ such that $(\y,\h)$ belongs to $\iota(W(\bold{a}))$, where $\bold{h}:=(h_{t,t'})_{0\leq t\leq 2k, 1\leq t'\leq a_{t}}$. Write $\y=(y_{0},\dots,y_{k})$ for some $y_{0},\dots,y_{k}\in\V$.
	Note that $(\y,\h)$ belongs to $\iota(W(\bold{a}))$ if and only if 
	for all $\e_{t,t'}\in\{0,1\}$,
	$0\leq t\leq 2k$, $1\leq t'\leq a_{t}$, writing $$\w:=(w_{0},\dots,w_{2k}):=\Bigl(y_{0}+\sum_{t'=1}^{a_{0}}\e_{0,t'}h_{0,t'},\dots,y_{2k}+\sum_{t'=1}^{a_{2k}}\e_{2k,t'}h_{2k,t'}\Bigr)$$ we have that $\w\in \Omega_{I}\times_{0}\Omega_{I}$. Recall that $\Omega_{I}\times_{0}\Omega_{I}$ is the set $\Omega_{4}$ in Example 
	B.4 of \cite{SunA}, $M$ is homogeneous and $X=\{v_{0},\dots,v_{k+s-1}\}$.  So $(\y,\h)\in\iota(W(\bold{a}))$ if and only if for all the above mentioned $\w$, we have that
   \begin{itemize}
   	\item $M(w_{j})=r$  for all $0\leq j\leq 2k$;
   	\item $M(w_{0}-w_{j})=M(w_{0}-w_{k+j})=M(v_{0}-v_{j})$ for all $1\leq j\leq k$;
   	\item $M(w_{j}-w_{j'})=M(w_{k+j}-w_{k+j'})=M(v_{j}-v_{j'})$ for all $1\leq j,j'\leq k$.
   \end{itemize}
   
   This is equivalent of saying that 
  $$(w_{j}A)\cdot w_{j'}=c_{j,j'}$$
    for all $(j,j')\in \{0,\dots,k\}^{2}\cup \{0,k+1,\dots,2k\}^{2}$,
    where $c_{j,j'}\in\F_{p}$ are some constants depending only on $v_{0},\dots,v_{k}$    with $c_{j,j'}=c_{j',j}$ and $c_{j,j}=r$ (since $M$ is homogeneous).
By Lemma \ref{4:or}, this implies that  
$(h_{t,v}A)\cdot h_{t',v'}=0$ for all $0\leq t,t'\leq k$, $1\leq v\leq a_{t}$, $1\leq v'\leq a_{t'}$ with $(t,v)\neq (t',v')$.
It is then not hard to compute that $(\y,\h)$ belongs to $\iota(W(\bold{a}))$ if and only if 
 \begin{itemize}
 	\item $(y_{t}A)\cdot y_{t'}=c_{t,t'}$
 	 for all $(t,t')\in \{0,\dots,k\}^{2}\cup \{0,k+1,\dots,2k\}^{2}$;
 	 \item $(h_{t,v}A)\cdot h_{t',v'}=0$ for all $(t,t')\in \{0,\dots,k\}^{2}\cup \{0,k+1,\dots,2k\}^{2}$, $1\leq v\leq a_{t}$, $1\leq v'\leq a_{t'}$ with $(t,v)\neq (t',v')$;
 	 \item $(y_{t}A)\cdot h_{t',v'}=0$ for all $(t,t')\in \{0,\dots,k\}^{2}\cup \{0,k+1,\dots,2k\}^{2}$,  $1\leq v'\leq a_{t'}$ with $t\neq t'$;
 	 \item $2(y_{t}A)\cdot h_{t,v}+(h_{t,v}A)\cdot h_{t,v}=0$ for all  $0\leq t\leq 2k$,  $1\leq v\leq a_{t}$. 
 \end{itemize}

 Let $0\leq j\leq 2k$
  and suppose that we have constructed $y_{0},\dots,y_{j-1}$ such that all the above mentioned conditions which involve $y_{0},\dots,y_{j-1}$ hold, and that $y_{j'},h_{t,t'}, 0\leq j'\leq j-1, 0\leq t\leq 2k, 1\leq t'\leq a_{t}$ are linearly independent (in the special case $j=0$, there is nothing to construct). We now construct $y_{j}$.
  Let $W_{j}$ be the set of $y\in\V$ such that
  \begin{itemize}
 	\item $(y_{j}A)\cdot y_{j'}=c_{j,j'}$
 	 for all  $0\leq j'\leq j-1$ with $(j,j')\in \{0,\dots,k\}^{2}\cup \{0,k+1,\dots,2k\}^{2}$;
 	 \item $(y_{j}A)\cdot h_{j',v'}=0$ for all $j'\neq j$ with $(j,j')\in \{0,\dots,k\}^{2}\cup \{0,k+1,\dots,2k\}^{2}$ and $1\leq v'\leq a_{j'}$;
 	 \item $2(y_{j}A)\cdot h_{j,v}+(h_{j,v}A)\cdot h_{j,v}=0$ for all   $1\leq v\leq a_{j}$. 
 \end{itemize} 
 To complete the induction, we need to find some $y_{j}\in V(M-r)\cap W_{j}$ such that $y_{j'},h_{t,t'}, 0\leq j'\leq j, 0\leq t\leq 2k, 1\leq t'\leq a_{t}$ are linearly independent.
 
    Since $y_{j'},h_{t,t'}, 0\leq j'\leq j-1, 0\leq t\leq 2k, 1\leq t'\leq a_{t}$ are linearly independent and $A$ is invertible, $W_{j}$ is an affine subspace of $\V$ of co-dimension at most $k+\sum_{t=0}^{2k}a_{t}\leq k+s-1$. By Proposition \ref{4:iissoo}, $\rank(M\vert_{W_{j}})\geq d-2(k+s-1)$. Since $d\geq 2k+2s+1$, by Lemma \ref{4:counting01}, $\vert V(M-r)\cap W_{j}\vert\geq p^{d-k-s}/2$ if $p\gg_{d} 1$. 
  On the other hand, the number of $y_{j}\in\V$ such that $y_{j'},h_{t,t'}, 0\leq j'\leq j, 0\leq t\leq 2k, 1\leq t'\leq a_{t}$ are  linearly dependent is at most $p^{2k+s-1}$. Since $d-k-s\geq 2k+s$, there exists at least one $y_{j}$ in $V(M-r)\cap W_{j}$ such that $y_{j'},h_{t,t'}, 0\leq j'\leq j, 0\leq t\leq 2k, 1\leq t'\leq a_{t}$ are  linearly independent. This completes the construction of $y_{j}$.
 The claim follows by constructing $y_{0},\dots,y_{2k}$ inductively as described above. 
    \end{proof}

	Fix any $(\y,(h_{t,t'})_{0\leq t\leq 2k, 1\leq t'\leq a_{t}})\in W(\bold{a})$.
	Let $\tilde{h}_{t,t'}\in (\Z^{d})^{2k+1}$ be the vector whose $td+1$ to $(t+1)d$ coordinates equal to $h_{t,t'}$, and whose all other coefficients are zero. 
	Since $P$ is  a constant $\mod\Z$ on $\iota^{-1}(\Omega_{I}\times_{0}\Omega_{I})$, we have that 
	\begin{equation}\label{4:2declear1}
		\prod_{0\leq t\leq 2k, 1\leq t'\leq a_{t}}\Delta_{\tilde{h}_{t,t'}}P(\x)\in\Z \text{ for all } (\x,(h_{t,t'})_{0\leq t\leq k, 1\leq t'\leq a_{t}})\in\iota^{-1}(\Omega_{I}\times_{0}\Omega_{I}).
	\end{equation}
	
	We need to introduce some notations. Write
	$\tilde{L}'_{j}(x_{0},\dots,x_{2k}):=\sum_{i=0}^{2k} u_{i,j}x_{i}$
	for some $u_{i,j}\in\Z$ of complexity at most $O_{C_{0},k,s}(1)$ for all $0\leq j\leq 2k+2s-2$. Let $u_{i}:=(u_{i,0},\dots,u_{i,2k+2s-2})\in \Z^{2k+2s-1}$, and let $\bold{u}:=(u_{0},\dots,u_{2k})$. For any $\bold{a}=(a_{0},\dots,a_{2k})\in\N^{2k+1}$, denote $\bold{u}^{\bold{a}}:=u_{0}^{\ast a_{0}}\ast \ldots\ast u_{2k}^{\ast a_{2k}}$, where $u_{i}^{\ast a_{i}}:=u_{i}\ast\ldots\ast u_{i}$ with $u_{i}$ appearing $a_{i}$ times.
	
	After some computations, it follows from (\ref{4:thisisp}) that
	\begin{equation}\nonumber
		\begin{split}
			\prod_{0\leq t\leq k, 1\leq t'\leq a_{t}}\Delta_{\tilde{h}_{t,t'}}P(\x)
			=\sum_{j\in\N^{d},\vert j\vert=i}\eta(\tilde{g}_{j}^{\bold{u}^{\bold{a}}})\sum_{i_{0},\dots,i_{2k}\in\N^{d}, \vert i_{t}\vert=a_{t}, i_{0}+\dots+i_{2k}=j}\Bigl(\sum_{R\in\mathcal{G}(i_{0},\dots,i_{2k})}\prod_{t=0}^{2k}\prod_{t'=1}^{a_{t}}h_{t,t',R(t,t')}\Bigr),
		\end{split}
	\end{equation}
	where $h_{t,t'}=(h_{t,t',1},\dots,h_{t,t',d})$, $i_{t}=(i_{t,1},\dots,i_{t,d})$, $\mathcal{G}(i_{0},\dots,i_{k})$ is the collection of all functions $R\colon\{(t,t')\colon 0\leq t\leq 2k, 1\leq t'\leq a_{t}\}\to\{1,\dots,d\}$ such that for all $0\leq t\leq 2k$ and $1\leq t''\leq d$, the number of $R(t,t'), 1\leq t'\leq a_{t}$ which equals to $t''$ is $i_{t,t''}$.	
	Let $\phi$ be a bijection between $\{1,\dots,i\}$ and $\{(t,t')\colon 0\leq t\leq 2k, 1\leq t'\leq a_{t}\}$. Then
	\begin{equation}\nonumber
		\begin{split}
			\prod_{0\leq t\leq 2k, 1\leq t'\leq a_{t}}\Delta_{\tilde{h}_{t,t'}}P(\x)
			=\sum_{j\in\N^{d},\vert j\vert=i}\eta(\tilde{g}_{j}^{\bold{u}^{\bold{a}}})\sum_{i_{0},\dots,i_{2k}\in\N^{d}, \vert i_{t}\vert=a_{t}, i_{0}+\dots+i_{2k}=j}\Bigl(\sum_{R\in\mathcal{G}(i_{0},\dots,i_{2k})}\prod_{t=1}^{i}h_{\phi(t),R(\phi(t))}\Bigr),
		\end{split}
	\end{equation}
	
	Let $\mathcal{F}(j)$ be the collection of all maps $r\colon \{1,\dots,i\}\to \{1,\dots,d\}$ such that $\vert r^{-1}(t'')\vert=j_{t''}$ for all $1\leq t''\leq d$.
	Note that for all $r\in\mathcal{F}(j)$, there exist a unique choice of $(i_{0},\dots,i_{2k})\in(\N^{d})^{2k+1}$ 
	with $\vert i_{t}\vert=a_{t},i_{0}+\dots+i_{2k}=j$,
	and $R\in \mathcal{G}(i_{0},\dots,i_{2k})$ such that $r=R\circ\phi$. Moreover, this induces a bijection between $\mathcal{F}(j)$ and $\cup_{i_{0},\dots,i_{2k}\in\N^{d},\vert i_{t}\vert=a_{t},i_{0}+\dots+i_{2k}=j}\mathcal{G}(i_{0},\dots,i_{2k})$. In other words,		 
	\begin{equation}\label{4:2compare2}
		\begin{split}
			\prod_{0\leq t\leq 2k, 1\leq t'\leq a_{t}}\Delta_{\tilde{h}_{t,t'}}P(\x)
			=\sum_{j\in\N^{d},\vert j\vert=i}\eta(\tilde{g}_{j}^{\bold{u}^{\bold{a}}})\cdot\Bigl(\sum_{r\in\mathcal{F}(j)}\prod_{t=1}^{i}h_{\phi(t),r(\phi(t))}\Bigr).
		\end{split}
	\end{equation}		
	
		By (\ref{4:2declear1}), (\ref{4:2compare2}) and Lemma \ref{4:lemhtt},
	$\prod_{0\leq t\leq 2k, 1\leq t'\leq a_{t}}\Delta_{\tilde{h}_{t,t'}}P(\x)$ is independent of $\x$, and belongs to $\Z$
	whenever $h_{t,t'}, 0\leq t\leq 2k, 1\leq t'\leq a_{t}$ are $p$-linearly independent and pairwise $(M,p)$-orthogonal. 
	By a change of indices, 
	for all $h_{1},\dots,h_{i}\in\Z^{d}$ which are $p$-linearly independent and pairwise $(M,p)$-orthogonal, writing $h_{t}=(m_{t,1},\dots,m_{t_{d}})$,
		we have that
		\begin{equation}\label{4:egmmr}
		\sum_{j\in\N^{d},\vert j\vert=i}\eta(\tilde{g}_{j}^{\bold{u}^{\bold{a}}})\cdot\Bigl(\sum_{r\in\mathcal{F}(j)}\prod_{t=1}^{i}m_{t,r(t)}\Bigr)\in\Z.
		\end{equation}

\begin{lem}\label{4:rwpp}	
There exist $\bold{a}=(a_{0},\dots,a_{2k})\in\N^{2k+1}$ with  $a_{0}+\dots+a_{2k}=i$ such that the map $\xi_{\bold{a}}(h):=\eta(h^{\bold{u}^{\bold{a}}}), h\in G$ is a nontrivial $i$-th type-III horizontal character of $G/\Gamma$ of complexity at most $O_{C_{0},C,\d,d}(1)$.
\end{lem}	
\begin{proof}	
	The proof is similar to the argument on pages 29--30 of \cite{GT10b}.
	Since $\eta\vert_{G_{(i)}^{V_{X}\times_{0}V_{X}}}$ is nontrivial, and the vectors $\bold{u}^{\bold{a}},\bold{a}=(a_{0},\dots,a_{2k})\in\N^{2k+1},a_{0}+\dots+a_{2k}=i$ span $(V_{X}\times_{0}V_{X})^{[i]}$,
	there exist $\bold{a}=(a_{0},\dots,a_{2k})\in\N^{2k+1}$ with  $a_{0}+\dots+a_{2k}=i$ such that $\xi_{\bold{a}}$ is nontrivial.  Fix such an $\bold{a}$.
	
	For all $\gamma\in\Gamma_{i}$, since $\gamma^{\bold{u}^{\bold{a}}}\in\Gamma_{(i)}^{V_{X}\times_{0}V_{X}}$, we have that $\xi_{\bold{a}}(\gamma)\in\Z$.
	
	Moreover, by the flag property of $V_{X}\times_{0} V_{X}$,  $\bold{u}^{\bold{a}}\in (V_{X}\times_{0}V_{X})^{[i]}\subseteq (V_{X}\times_{0}V_{X})^{[i+1]}$. So for all $h_{i+1}\in G_{i+1}$, we have $h_{i+1}^{\bold{u}^{\bold{a}}}\in G_{(i+1)}^{V_{X}\times_{0}V_{X}}$. By the maximality of $i$, we have that $\eta(h_{i+1}^{\bold{u}^{\bold{a}}})\in\Z$ and so $\xi_{\bold{a}}$ annihilates $G_{i+1}$.
	
	On the other hand, for any $0\leq j\leq i$, there exist $0\leq a'_{0}\leq a_{0},\dots,0\leq a'_{2k}\leq a_{2k}$ such that $a'_{0}+\dots+a'_{2k}=j$. Write $\bold{a}':=(a'_{0},\dots,a'_{2k})$. Then
	we may write
	$\bold{u}^{\bold{a}}=w\ast w',$
	where 
	$w:=\bold{u}^{\bold{a}'}$ and $w':=\bold{u}^{\bold{a}-\bold{a}'}.$
	It is clear that $w\in (V_{X}\times_{0}V_{X})^{[j]}$ and $w'\in (V_{X}\times_{0}V_{X})^{[i-j]}$.
	If $h\in G_{j}$ and $h'\in G_{i-j}$, then
	$$[h^{w},{h'}^{w'}]\equiv [h,h']^{ww'} \mod G^{V_{X}\times_{0}V_{X}}_{(i+1)}$$
	by the Baker-Campbell-Hausdorff formula. Since $\eta$ is trivial on $G^{V_{X}\times_{0}V_{X}}_{(i+1)}$ and $[G^{V_{X}\times_{0}V_{X}},$ $G^{V_{X}\times_{0}V_{X}}]$, we have that
	$$\xi_{\bold{a}}([h,h'])\equiv \eta([h,h']^{ww'})\equiv \eta([h^{w},h^{w'}])=0 \mod \Z.$$
	So $\xi_{\bold{a}}$ vanishes on $[G_{j},G_{i-j}]$ for all $0\leq j\leq i$.
	In conclusion, we have that $\xi_{\bold{a}}$ is an $i$-th type-III horizontal character. It is clear to see that the complexity of $\xi_{\bold{a}}$ is $O_{C_{0},C,\d,d}(1)$.
\end{proof}

 			Fix such an $\bold{a}$ and $\xi:=\xi_{\bold{a}}$ given by Lemma \ref{4:rwpp}, and denote 
	$$Q(n,h_{1},\dots,h_{i}):=\xi\circ \Delta_{h_{i}}\dots \Delta_{h_{1}}\tilde{g}(n).$$		
	Since $\xi$ is a $i$-th type-III horizontal character, similar to (\ref{4:fxi}),
	one can  compute that 
	\begin{equation}\label{4:2compare1}
		Q(n,h_{1},\dots,h_{i})=\sum_{j\in\N^{d},\vert j\vert=i}\xi(\tilde{g}_{j})\sum_{r\in \mathcal{F}(j)}\prod_{t=1}^{i}m_{t,r(t)},
	\end{equation}
	where $h_{t}=(m_{t,1},\dots,m_{t,d})$, and recall that $\mathcal{F}(j)$ is the collection of all maps $r\colon \{1,\dots,i\}\to \{1,\dots,d\}$ such that $\vert r^{-1}(t'')\vert=j_{t''}$ for all $1\leq t''\leq d$.
	So $Q$
	is a polynomial in $\poly((\Z^{d})^{i+1}\to\R)$ which is independent of $n$.
	Let $\tilde{M}_{r}\colon\Z^{d}\to\Z/p$ be the regular lifting of $M-r$.
	Since $\tilde{g}\in\poly_{p}(\iota^{-1}(V(M-r))\to G_{\N}\vert\Gamma)=\poly_{p}(V_{p}(\tilde{M}_{r})\to G_{\N}\vert\Gamma)$ , by Lemma \ref{4:goodcoordinates}, $p^{s-1}\eta(g_{i})\in\Z$ for all $i\in\N^{d}, 0<\vert i\vert\leq s-1$. 
	So $Q\in \poly((\Z^{d})^{i+1}\to\Z/p^{s-1})$.

	Combining (\ref{4:2compare1}) with (\ref{4:egmmr}), we have that 
	\begin{equation}\label{4:qisconstant}
		Q(n,h_{1},\dots,h_{i})\in\Z \text{ for all $(n,h_{1},\dots,h_{i})\in W_{i}$,}
	\end{equation}	 
	where $W_{i}$ is the set of
	$(n,h_{1},\dots,h_{i})\in \Gow_{p,i}(V_{p}(\tilde{M}_{r}))$ with $h_{1},\dots,h_{i}$ being $p$-linearly independent. Here we used the fact that 
	for all $h_{1},\dots,h_{i}\in\Z^{d}$ which are $p$-linearly independent and pairwise $(M,p)$-orthogonal, there exists $n\in\Z^{d}$ with $(n,h_{1},\dots,h_{i})\in \Gow_{p,i}(V_{p}(\tilde{M}_{r}))$ (which follows from the $\Z^{d}$-version of Lemmas \ref{4:counting02} and  \ref{4:changeh} since $d\geq 2s+1$).
 	In other words,
	$W_{i}\subseteq V_{p}(Q)\cap \Gow_{p,i}(V_{p}(\tilde{M}_{r}))$ (note that $W_{i}+p(\Z^{d})^{i+1}=W_{i}$).

	 Next we use the irreducibility property  of $\Gow_{p,i}(V_{p}(\tilde{M}_{r}))$ to remove the $p$-linearly independence restriction in (\ref{4:qisconstant}).
	By Lemma \ref{4:countingh}, since $d\geq s^{2}-s+3$, $$\vert \Gow_{p,i}(V_{p}(\tilde{M}_{r}))\cap[p]^{d(i+1)}\vert=\vert \Gow_{i}(V(M-r))\vert=p^{d(i+1)-(\frac{i(i+1)}{2}+1)}(1+O_{s}(p^{-1/2})).$$   On the other hand, by Lemma \ref{4:iiddpp}, $$\vert [p]^{d(i+1)}\backslash W_{i}\vert\leq  ip^{d+(d+1)(i-1)}=ip^{di+(i-1)}.$$ 
	Since $d\geq s^{2}-s+3$ and $p\gg_{d,k,s} 1$,   this implies that 
	  \begin{equation}\nonumber
		\vert (\Gow_{p,i}(V_{p}(\tilde{M}_{r}))\cap[p]^{d(i+1)})\backslash W_{i}\vert<\vert \Gow_{p,i}(V_{p}(\tilde{M}_{r}))\cap[p]^{d(i+1)}\vert/2.
	\end{equation}	
	Since $W_{i}\subseteq V_{p}(Q)\cap \Gow_{p,i}(V_{p}(\tilde{M}_{r}))$, we have that
	 \begin{equation}\label{4:shit2}
		\vert V_{p}(Q)\cap \Gow_{p,i}(V_{p}(\tilde{M}_{r}))\cap[p]^{d(i+1)}\vert>\vert \Gow_{p,i}(V_{p}(\tilde{M}_{r}))\cap[p]^{d(i+1)}\vert/2.
	\end{equation}

	By Example 
	B.4 of \cite{SunA}, $\Gow_{i}(V(M-r))$ is a nice and consistent $M$-set of total co-dimension $\binom{i+1}{2}+1$. So by Theorem \ref{4:irrr},
	since $d\geq s^{2}-s+3$, we have that
	$\\iota^{-1}(\Gow_{i}(V(M-r)))=\Gow_{p,i}(V_{p}(\tilde{M}_{p}))$ is strongly  $(1/2,p)$-irreducible up to degree $s-1$ (see Appendix \ref{4:s:ir} for definitions). Since $Q\in \poly((\Z^{d})^{i+1}\to\Z/p^{s-1})$,  (\ref{4:shit2}) implies that 
	$\Gow_{p,i}(V_{p}(\tilde{M}_{r}))\subseteq V_{p}(Q)$. 
		Therefore,
	$$\text{$Q(n,h_{1},\dots,h_{i})\in\Z$	 for all 
	$(n,h_{1},\dots,h_{i})\in \Gow_{p,i}(V_{p}(\tilde{M}_{r}))$.  }$$

	By the definition of $Q$, this means that
	  $\tilde{g}$ is not $(O_{C_{0},C,\d,d}(1),\iota^{-1}(V(M-r)),p)$-irrational and thus $g$ is not $(O_{C_{0},C,\d,d}(1),V(M-r))$-irrational. We arrive at contradiction if we pick some $N_{0}\gg_{C_{0},C,\d,d} 1$. This completes the proof of Part (ii) of Theorem \ref{4:orbitdesc2}.

	\

		So now it suffices to prove Part (i) of Theorem \ref{4:orbitdesc2}. We refer the readers to Appendices \ref{4:s:AppB3} and \ref{4:s:AppB4} for the terminologies to be used in the discussion below.	
	By Lemma \ref{4:kk3d} (ii),   $\Omega_{I}\times_{0}\Omega_{I}$  is a nice and consistent $M$-set of total co-dimension at most $k^{2}+3k+1$. 
So we may write $\Omega_{I}\times_{0}\Omega_{I}=V(\mathcal{J})$ for some consistent $(M,2k+1)$-family $\mathcal{J}\subseteq \F_{p}[x_{0},\dots,x_{2k}]$ of total dimension at most $k^{2}+3k+1$. Let $(\mathcal{J}',\mathcal{J}'')$ be an  $\{x_{0},\dots,x_{k}\}$-decomposition of $\mathcal{J}$. 
	Note that the set of $(x_{0},\dots,x_{k})\in(\V)^{2k+1}$  such that $(x_{0},\dots,x_{2k})\in V(\mathcal{J}')$ for all $x_{k+1},\dots,x_{2k}\in \V$ is equal to $\Omega_{I}$.

	We remark that $G^{V_{X}}$ is the projection of $G^{V_{X}\times_{0}V_{X}}$ onto the first $(k+s)d$ coordinates.
	Let $F\in\Lip(G^{V_{X}}/\Gamma^{V_{X}}\to\C)$ with Lipschitz norm bounded by 1. Define $\tilde{F}\colon G^{V_{X}\times_{0}V_{X}}/\Gamma^{V_{X}\times_{0}V_{X}}\to\C$, $\tilde{F}(x_{0},\dots,x_{2k+2s-2}):=F(x_{0},\dots,x_{k+s-1})$.  
	Since Part (ii) holds, by Theorem \ref{4:ct}, Lemma \ref{4:kk2d}, and the observation that $L'_{j}(\x,\x')=L_{I,j}(\x)$ for all $(\x,\x')\in(\V)^{k+1}\times (\V)^{k}$ and $0\leq j\leq k+s-1$, we have that
\begin{equation}\nonumber
\begin{split}
&\quad \Bigl\vert\E_{(x_{0},\dots,x_{k+s-1})\in\Omega}F((g(x_{0}),\dots,g(x_{k+s-1}))\Gamma^{k+s})-\int_{G^{V_{X}}/\Gamma^{V_{X}}} F\,dm\Bigr\vert
\\&=\Bigl\vert\E_{\x\in\Omega_{I}}F((g(L_{I,0}(\x)),\dots,g(L_{I,k+s-1}(\x)))\Gamma^{k+s})-\int_{G^{V_{X}}/\Gamma^{V_{X}}} F\,dm\Bigr\vert
\\&=\Bigl\vert\E_{\x\in\Omega_{I}}\E_{\x'\in (\V)^{k}\colon (\x,\x')\in V(\mathcal{J}'')}
 F((g(L_{I,0}(\x)),\dots,g(L_{I,k+s-1}(\x)))\Gamma^{k+s})-\int_{G^{V_{X}\times_{0}V_{X}}/\Gamma^{V_{X}\times_{0}V_{X}}} \tilde{F}\,dm\Bigr\vert
 \\&=\Bigl\vert\E_{(\x,\x')\in\Omega_{I}\times_{0}\Omega_{I}}F((g(L_{I,0}(\x)),\dots,g(L_{I,k+s-1}(\x)))\Gamma^{k+s})-\int_{G^{V_{X}\times_{0}V_{X}}/\Gamma^{V_{X}\times_{0}V_{X}}} \tilde{F}\,dm\Bigr\vert+O_{k,s}(p^{-1/2})
\\&=\Bigl\vert\E_{\y\in\Omega_{I}\times_{0}\Omega_{I}}\tilde{F}((g(L'_{0}(\y)),\dots,g(L'_{2k+2s-2}(\y)))\Gamma^{k+s})-\int_{G^{V_{X}\times_{0}V_{X}}/\Gamma^{V_{X}\times_{0}V_{X}}} \tilde{F}\,dm\Bigr\vert+O_{k,s}(p^{-1/2})
\\&\leq \d \Vert\tilde{F}\Vert_{\Lip(G^{V_{X}\times_{0}V_{X}}/\Gamma^{V_{X}\times_{0}V_{X}})}+O_{k,s}(p^{-1/2})
=\d \Vert F\Vert_{\Lip(G^{V_{X}}/\Gamma^{V_{X}})}+O_{k,s}(p^{-1/2})\leq \d +O_{k,s}(p^{-1/2}),
\end{split}
\end{equation}
where $m$ is the Haar measure of $G^{V_{X}\times_{0}V_{X}}/\Gamma^{V_{X}\times_{0}V_{X}}.$
This completes the proof of Theorem \ref{4:orbitdesc2}
  by taking $p\gg_{C_{0},C,\d,d} 1$.

\section{Proof of Theorem \ref{4:mainmain2}}\label{4:s:7rr}

We are now ready to complete the proof of Theorem \ref{4:mainmain2}. The outline of the proof is similar to Section 6 of \cite{GT10b}.
By Lemma \ref{4:kk2d} (ii) and (iii),
it suffices to show that if 
$d\geq d_{0}(X),$
 and $p\gg_{C_{0},d,\e} 1$, then  
\begin{equation}\label{4:mmm1}
\begin{split}
\E_{(x_{0},\dots,x_{k+s-1})\in \Omega}\bold{1}_{E}(x_{0})\bold{1}_{E}(x_{1})\dots \bold{1}_{E}(x_{k+s-1})\gg_{C_{0},d,\e} 1
\end{split}
\end{equation}
for all $E\subseteq V(M-r)$ with $\vert E\vert>\e p^{d-1}$. Let $\mathcal{F}\colon\R_{+}\to\R_{+}$ be a growth function and $\e'>0$ to be chosen later, which depend only on  $C_{0},d,\e$.
By Theorem \ref{4:st3}, if $p\gg_{d,\e',\mathcal{F}} 1$  and $d\geq (2s+12)(15s+423)$, then there exists $C:=C(d,\e',\mathcal{F})>0$ such that we may write
\begin{equation}\label{4:mm2}
\begin{split}
\bold{1}_{E}=f_{\str}+f_{\uni}+f_{\err}
\end{split}
\end{equation}
 such that  
\begin{itemize}
\item $f_{\str}$ and $f_{\str}+f_{\err}$ take values in $[0,1]$;
\item we may write $f_{\str}$ as $$f_{\str}(n)=F(g(n)\Gamma)  \text{ for all }  n\in V(M-r)$$
	for some   $\N$-filtered nilmanifold $G/\Gamma$  of degree at most $s-1$ and complexity at most $C$, some $(\mathcal{F}(C),V(M-r))$-irrational  $g\in\poly_{p}(V(M)\to G_{\N}\vert \Gamma)$, and some    $F\in\Lip(G/\Gamma\to\C)$  of Lipschitz norm at most $C$;
	\item  $\Vert f_{\uni}\Vert_{U^{s}(V(M-r))}\leq 1/\mathcal{F}(C)$;
	\item $\Vert f_{\err}\Vert_{L^{2}(V(M-r))}\leq \e'$.
\end{itemize}
We remark that $C$ will be dependent only on $C_{0},d,\e$ since we will choose $\mathcal{F}$ and $\e'$ to be dependent only on  $C_{0},d,\e$.

  We need to first construct a good weight function for our average:

\begin{prop}\label{4:recset}
	If the growth function $\mathcal{F}$ grows sufficiently fast depending on $C_{0},d,\e'$ and  $p\gg_{C_{0},C,d,\e',\mathcal{F}} 1$, then there exists  a function $\omega\colon \Omega\to \mathbb{R}$ such that the following holds:
	\begin{enumerate}[(i)]
		\item $0\leq\omega(\x)\leq O_{C_{0},C,d,\e'}(1)$ for all $\x\in\Omega$;
		\item $\E_{\x\in\Omega}\omega(\x)=1$; 
		\item $\vert\E_{\x\in \Omega\times_{i}\Omega}(\omega\times_{i}\omega)(\x)-1\vert\leq O_{C_{0},d}(\e')$ for all $0\leq i\leq k+s-1$, where  $\omega\times_{i}\omega\colon \Omega\times_{i}\Omega\to \C$ is the function given by
$$\omega\times_{i}\omega(x,\x^{+},\x^{-}):=\omega(x\da_{i} \x^{+})\overline{\omega}(x\da_{i} \x^{+})$$
for all $(x,\x^{+},\x^{-})\in \Omega\times_{i}\Omega$;\footnote{It is crutial that the implicit constants in Condition (iii) is independent of $C$.}  
		\item 	$\vert f_{\str}(x_{i})-f_{\str}(x_{i'})\vert\leq \e'$ for all $0\leq i,i'\leq k+s-1$ if $\omega(x_{0},\dots,x_{k+s-1})>0$.
	\end{enumerate}	
\end{prop}
\begin{proof}
The idea of the proof comes from Proposition 6.2 of \cite{GT10b}. However, we need to tailor the approach of  \cite{GT10b} in many aspects for our purposes.
 Let $\e'',R,\rho_{0},\rho,\d>0$ 	to be chosen later. We will chose $\e'',R,\rho_{0},\rho,\d$ in a way such that the later parameter in the list below
	$$1, \e'', R^{-1}, \rho_{0}, \rho, \d$$
	is sufficiently smaller than the former one, depending only on $C_{0},C,d,\e'$. 
	 
	 For the moment write the quantity $\d:=\d_{C_{0},d,\e'}(C)$ to be chosen laster as a function of $C$.
 		Since $g$ is $(\mathcal{F}(C),V(M-r))$-irrational, by Theorem \ref{4:orbitdesc2},
			 if $\mathcal{F}$ is a growth function such that 
			$\mathcal{F}(x)>N_{0}(\max\{C_{0},x\},\d_{C_{0},d,\e'}(x),d,k,s)$ (where $N_{0}$ is defined in Theorem \ref{4:orbitdesc2}) for all $x>0$ and $p\gg_{C_{0},C,\d,d,\mathcal{F}} 1$, then 
	 	\begin{itemize}
	 		\item the sequence
	 		\begin{equation}\nonumber
	 			(g(x_{0})\Gamma,\dots,g(x_{k+s-1})\Gamma)_{(x_{0},\dots,x_{k+s-1})\in\Omega}
	 		\end{equation}
	 		is $2\d$-equidistributed  on $G^{V_{X}}/\Gamma^{V_{X}}$; 
	 		\item the sequence
	 		\begin{equation}\nonumber
	 			(g(x_{0})\Gamma,\dots,g(x_{2k+2s-1})\Gamma)_{(x_{0},\dots,x_{2k+2s-1})\in \Omega\times_{i}\Omega}
	 		\end{equation}
	 		is $\d$-equidistributed  on $G^{V_{X}\times_{i}V_{X}}/\Gamma^{V_{X}\times_{i}V_{X}}$ for all $0\leq i\leq k+s-1$.  
	 	\end{itemize}

	Let $\{X_{1},\dots,X_{\dim(G)}\}$ be the Mal'cev basis of $G/\Gamma$.
	For $\rho>0$, let $B_{\rho}$ denote the set of elements in $G$ of the form
	$$\exp\Bigl(\sum_{j=1}^{\dim(G)}t_{j}X_{j}\Bigr)$$
	for some $\vert t_{j}\vert\leq \rho^{s-i}$ whenever $1\leq i\leq s-1$ and $j\leq \dim(G)-\dim(G_{i})$.
	 For all $\e'',\rho>0$, let $\phi_{\rho,\e''}\colon G\to \R$ be a non-negative function which is supported on $B_{\rho}$ and equals to 1 on $B_{(1-\e'')\rho}$. We may further require that the Lipschitz norm of $\phi_{\rho,\e''}$ is $O_{C,\e'',\rho}(1)$ and that $\phi_{\rho',\e''}\leq \phi_{\rho,\e''}$ pointwise for all $0<\rho'<\rho\leq 1$ and $\e''>0$. Let $\Phi_{\rho,\e''}\colon (G/\Gamma)^{2}\to\R$ be the function given by
	$$\Phi_{\rho,\e''}(x,x'):=\sum_{h\in G\colon hx=x'}\phi_{\rho,\e''}(h).$$
	The function $\Phi_{\rho,\e''}$ is supported near the diagonal of $(G/\Gamma)^{2}$ and $\Phi_{\rho,\e''}(x,x')\neq 0$ only if $x'\in B_{\rho}x$. Moreover, if $x'\in B_{(1-\e'')\rho}x$, then $\Phi_{\rho,\e''}(x,x')=1$.
	By Lemma 6.3 of \cite{GT10b}, if $\rho_{0}$ is chosen sufficiently small depending on $C,\e'',R$, then we have the approximate shift-invariance 
	\begin{equation}\label{4:3e}
	\begin{split}
	\Phi_{(1-3\e'')\rho,\e''}(x,x')\leq \Phi_{\rho,\e''}(hx,hx')\leq \Phi_{(1+3\e'')\rho,\e''}(x,x')
	\end{split}
	\end{equation}
	for all $0<\rho<\rho_{0}$, $x,x'\in G/\Gamma$ and $h\in G$ such that $d_{G}(h,id_{G})\leq R$.

	 Let 
	 $$W_{\rho,\e''}:=\E_{(x_{0},\dots,x_{k+s-1})\in\Omega}\prod_{0\leq i<i'\leq k+s-1}\Phi_{\rho,\e''}(g(x_{i})\Gamma,g(x_{i'})\Gamma)$$
	 and
	\begin{equation}\nonumber
	\begin{split}
	\omega(x_{0},\dots,x_{k+s-1}):=\omega_{\rho,\e''}(x_{0},\dots,x_{k+s-1}):=\frac{1}{W_{\rho,\e''}}\prod_{0\leq i<i'\leq k+s-1}\Phi_{\rho,\e''}(g(x_{i})\Gamma,g(x_{i'})\Gamma)
	\end{split}
	\end{equation}
	for some $\rho$ and $\e''$  to be chosen later.	
	 By definition, it is clear that $\E_{\x\in\Omega}\omega(\x)=1$.
    So Condition (ii) holds.
	
	If $\omega(x_{0},\dots,x_{k+s-1})>0$, then for all $0\leq i<i'\leq k+s-1$,
	$$\vert f_{\str}(x_{i})-f_{\str}(x_{i'})\vert\leq \Vert F\Vert_{\Lip(G/\Gamma)}\cdot d_{G/\Gamma}(g(x_{i})\Gamma,g(x_{i'})\Gamma)\leq C\rho_{0}<\e'.$$
	So Condition (iv) is satisfied if we choose $\rho_{0}\leq \e' C^{-1}$.

	 Since $V_{X}$    is of complexity at most $O_{C_{0},d}(1)$, by Lemma \ref{4:vxprop} (iv), we have that $G^{V_{X}}/\Gamma^{V_{X}}$ is of complexity at most $O_{C_{0},C,d}(1)$. 
	Let $\tilde{\Phi}_{\rho,\e''}\colon (G/\Gamma)^{k+s}\to\R$ be the function given by
	\begin{equation}\nonumber
	\begin{split}
	\tilde{\Phi}_{\rho,\e''}(u_{0},\dots,u_{k+s-1}):=\prod_{0\leq i<i'\leq k+s-1}\Phi_{\rho,\e''}(u_{i},u_{i'})
	\end{split}
	\end{equation}	
	Clearly $\Vert\tilde{\Phi}_{\rho,\e''}\Vert_{\Lip((G/\Gamma)^{k+s})}=O_{C_{0},C,d,\e'',\rho}(1)$ and
	$\Vert\tilde{\Phi}_{\rho,\e''}\Vert_{\Lip(G^{V_{X}}/\Gamma^{V_{X}})}=O_{C_{0},C,d,\e'',\rho}(1).$

		Define
		$$S_{\rho,\e''}:=\int_{G^{V_{X}}/\Gamma^{V_{X}}}\tilde{\Phi}_{\rho,\e''}\,dm_{G^{V_{X}}/\Gamma^{V_{X}}},
		$$				
		where $m_{G^{V_{X}}/\Gamma^{V_{X}}}$ is the Haar measure of $G^{V_{X}}/\Gamma^{V_{X}}$. 
	Then 
		\begin{equation}\label{4:WS}
			\begin{split}
				\vert W_{\rho,\e''}-S_{\rho,\e''}\vert\leq 2\d\Vert\tilde{\Phi}_{\rho,\e''}\Vert_{\Lip(G^{V_{X}}/\Gamma^{V_{X}})}=O_{C_{0},C,d,\e'',\rho}(\d).
			\end{split}
		\end{equation}

  Since $\tilde{\Phi}_{\rho,\e''}$ equals to 1 on a ball of radius $\rho^{O_{C_{0},C,d}(1)}$ centered at the identity and is bounded by 1 throughout. We have that
   	\begin{equation}\label{4:WS2}
   		\begin{split}
   			O_{C_{0},C,d}(\rho^{O_{C_{0},C,d}(1)})\leq S_{\rho,\e''}\leq 1.
   		\end{split}
   	\end{equation}	
   Moreover, by the property of $\phi_{\rho,\e''}$, we have that $S_{\rho',\e''}\leq S_{\rho,\e''}$ for all $0<\rho'<\rho<\rho_{0}$. 
   
   \textbf{Claim.} If $\rho_{0}$ and $\e''$ sufficiently small depending only on $C_{0},C,d,\e'$, then there exists    
 $\rho_{0}>\rho\gg_{C_{0},C,d,\e',\e'',\rho_{0}} 1\gg_{C_{0},C,d,\e'} 1$ such that
	\begin{equation}\label{4:epe}
	\begin{split}
	(1-\e')S_{\rho,\e''}\leq S_{(1-3\e'')\rho,\e''}\leq S_{(1+3\e'')\rho,\e''}\leq (1+\e')S_{\rho,\e''}.
	\end{split}
	\end{equation}		
	
		To see this,
	 we say that $0<\rho<\rho_{0}$ is \emph{good} if 
	 $S_{\rho,\e''}<(1+\e')S_{(1-3\e'')\rho,\e''}.$ 
	 Denote $\rho_{n}:=(1-3\e'')^{n}\rho_{0}$. Suppose that for some $S\in\N_{+}$, there exist $n_{0}:=0<n_{1}<n_{2}<\dots<n_{S}$ with $1\leq n_{i}-n_{i-1}\leq 2$ such that $\rho_{n_{i}}$ is bad for all $1\leq i\leq S$.
	 Then 
	 $$S_{\rho_{n_{i}},\e''}-S_{\rho_{n_{i+1}},\e''}\geq S_{\rho_{n_{i}},\e''}-S_{(1-3\e'')\rho_{n_{i}},\e''}\geq\e' S_{(1-3\e'')\rho_{n_{i}},\e''}\geq \e'S_{\rho_{n_{i+1}},\e''}$$
	 for all $1\leq i\leq S-1$.
	 Summing over $i$, it follows from (\ref{4:WS2}) that
	 \begin{equation}\nonumber
	 \begin{split}
	&\quad 1\geq S_{\rho_{n_{1}},\e''}\geq \sum_{i=2}^{S-1}\e' S_{\rho_{n_{i}},\e''}\geq \sum_{i=2}^{S-1}\e' S_{(1-3\e'')^{2i+2}\rho_{0},\e''}
	 \\&=O_{C_{0},C,d,\e'}\Bigl(\rho_{0}^{O_{C_{0},C,d}(1)}\frac{(1-3\e'')^{4}(1-(1-3\e'')^{2S-2})}{1-(1-3\e'')^{2}}\Bigr).
	 \end{split}
	 \end{equation}
	 If we  choose $\rho_{0}$ and $\e''$ sufficiently small depending only on $C_{0},C,d,\e'$,  and then choose $S$ to be sufficiently large depending only on $\rho_{0},C_{0},C,d,\e',\e''$, we arrive at a contradiction.
	 In other words, there exist $S=O_{\rho_{0},C_{0},C,d,\e',\e''}(1)$ and some $0<n\leq 2S$ such that both $\rho_{n}$ and $\rho_{n+1}$ are good.
	 
	  Let $\rho:=\rho_{n+1}$. Then $\rho\gg_{\rho_{0},C_{0},C,d,\e',\e''} 1$.
	 Since $\rho$ and $(1-3\e'')^{-1}\rho$ are good, we have that
	 $$(1-\e')S_{\rho,\e''}<(1+\e') S_{(1-3\e'')\rho,\e''}-\e'S_{\rho,\e''}\leq S_{(1-3\e'')\rho,\e''}$$
	 and
	 $$S_{(1+3\e'')\rho,\e''}<S_{(1-3\e'')^{-1}\rho,\e''}<(1+\e')S_{\rho,\e''}.$$ 
	This completes the proof of the claim.

		\

	 We now fix such $\rho$ given by the claim. 	
	We remark that  $1\ll_{C_{0},C,d,\e'}\rho\leq O_{C_{0},C,d,\e'}(1)$ since eventually we will pick  $\rho_{0}$ to be dependent only on $C_{0},C,d,\e'$.
		It then follows from (\ref{4:WS}) and (\ref{4:WS2}) that 
	$$W_{\rho,\e''}\geq 	O_{C_{0},C,d}(\rho^{O_{C_{0},C,d}(1)})+O_{C_{0},C,d,\e'',\rho}(\d)\gg_{C_{0},C,d,\e'} 1$$	
	if $\d$ is sufficiently small depending on $C_{0},C,d,\e'',\rho$.	
		So Condition (i) holds.

	It remains to check Condition (iii). For convenience we only consider the case $i=0$ as the other cases are similar. By Lemma \ref{4:lemgg}, we may assume without loss of generality that $I=\{0,\dots,k\}$ is a generating set of $X$.
	Let $\tilde{\Phi}_{\rho,\e''}\times_{0}\tilde{\Phi}_{\rho,\e''}\colon (G/\Gamma)^{2k+2s-1}\to\R$ be the function given by
	\begin{equation}\nonumber
		\begin{split}
			\tilde{\Phi}_{\rho,\e''}\times_{0}\tilde{\Phi}_{\rho,\e''}(u_{0},\dots,u_{2k+2s-1}):=\tilde{\Phi}_{\rho,\e''}(u_{0},\dots,u_{k+s-1})\tilde{\Phi}_{\rho,\e''}(u_{k+s},\dots,u_{2k+2s-1})
		\end{split}
	\end{equation}	
	for all $u_{0},\dots,u_{2k+2s-1}\in G/\Gamma$.
	Clearly
	$\Vert\tilde{\Phi}_{\rho,\e''}\times_{0}\tilde{\Phi}_{\rho,\e''}\Vert_{\Lip(G^{V_{X}\times_{0}V_{X}}/\Gamma^{V_{X}\times_{0}V_{X}})}=O_{C_{0},C,d,\e'',\rho}(1).$
	Define
	$$S'_{\rho,\e''}:=\int_{G^{V_{X}\times_{0}V_{X}}/{\Gamma}^{V_{X}\times_{0}V_{X}}}\tilde{\Phi}_{\rho,\e''}\times_{0}\tilde{\Phi}_{\rho,\e''}\,dm_{G^{V_{X}\times_{0}V_{X}}/{\Gamma}^{V_{X}\times_{0}V_{X}}},$$
	where $m_{G^{V_{X}\times_{0}V_{X}}/{\Gamma}^{V_{X}\times_{0}V_{X}}}$ is the Haar measure of $G^{V_{X}\times_{0}V_{X}}/{\Gamma}^{V_{X}\times_{0}V_{X}}$.		
	 By  (\ref{4:WS}),
	\begin{equation}\label{4:WS3}
		\begin{split}
		&\quad \E_{\x\in \Omega\times_{i}\Omega}(\omega\times_{i}\omega)(\x)
		=\frac{1}{W_{\rho,\e''}^{2}}\E_{\x\in \Omega\times_{i}\Omega}\tilde{\Phi}_{\rho,\e''}\times_{0}\tilde{\Phi}_{\rho,\e''}(\x)
		\\&=\frac{1}{(S_{\rho,\e''}+O_{C_{0},C,d,\e'',\rho}(\d))^{2}}(S'_{\rho,\e''}+\d\Vert\tilde{\Phi}_{\rho,\e''}\times_{0}\tilde{\Phi}_{\rho,\e''}\Vert_{\Lip(G^{V_{X}\times_{0}V_{X}}/\Gamma^{V_{X}\times_{0}V_{X}})})
		\\&=(S'_{\rho,\e''}+O_{C_{0},C,d,\e'',\rho}(\d))(S_{\rho,\e''}+O_{C_{0},C,d,\e'',\rho}(\d))^{-2}.
		\end{split}
	\end{equation}

		For convenience we write $A=o_{c\to\infty;\mathcal{C}}(B)$ if for any $\e>0$, there exists $K>0$ depending only on $\e$ and  the parameters in $\mathcal{C}$ such that $\vert A\vert\leq \e\vert B\vert$ for all $c\geq K$.
If we can show that
	\begin{equation}\label{4:WS4}
		\begin{split}
		(1+O(\e'))(1+o_{R\to\infty;C_{0},C,d}(1))S_{\rho,\e''}^{2}=S'_{\rho,\e''},
		\end{split}
	\end{equation}
then by choosing first $R$ being sufficiently small depending on $C_{0},C,d,\e,\e''$, then $\rho$ being sufficiently small depending on $C_{0},C,d,\e,\e'',R$, and then $\d$ being sufficiently small depending on $C_{0},C,d,\e,\e'',R,\rho$, and finally $p$ being sufficiently large depending on $C_{0},C,d,$ $\e,\e'',\mathcal{F},R,\rho,\d$,
we deduce from (\ref{4:WS2}), (\ref{4:WS3})  and (\ref{4:WS4}) that Condition (iii) holds.

So it suffices to prove (\ref{4:WS4}). The argument is similar to  Claim 6.4 of \cite{GT10b}. For $i\in\N$,
Denote $H_{i}:=\{(g_{0},\dots,g_{2k+2s-1})\in G_{(i)}^{V_{X}}\times G_{(i)}^{V_{X}}\colon g_{0}=g_{k+s}\}$
and let $H:=H_{1}$. Then $(H_{i})_{i\in\N}$ is an $\N$-filtration of $H$.
It is not hard to see that 
there exists a natural $\N$-filtered isomorphism $\psi\colon G^{V_{X}\times_{0} V_{X}}/\Gamma^{V_{X}\times_{0} V_{X}}\to H/((\Gamma^{V_{X}}\times \Gamma^{V_{X}})\cap H)$ induced by 
$$\psi(g_{0},\dots,g_{2k+2s-2}):=(g_{0},\dots,g_{k+s-1},g_{0},g_{k+s},\dots,g_{2k+2s-2})$$
for all $(g_{0},\dots,g_{2k+2s-2})\in G^{V_{X}\times_{0} V_{X}}$.
For the rest of the proof, for convenience all the integrals are assumed to be taken with respect to the Haar measure of space.
Then we may rewrite $S'_{\rho,\e''}$ as 
$$S'_{\rho,\e''}=\int_{(\x,\x')\in (G^{V_{X}}\times G^{V_{X}})/(\Gamma^{V_{X}}\times \Gamma^{V_{X}}), x_{0}=x'_{0}}\tilde{\Phi}_{\rho,\e''}(\x)\tilde{\Phi}_{\rho,\e''}(\x'),$$
where we write $\x=(x_{0},\dots,x_{k+s-1})$ and $\x'=(x'_{0},\dots,x'_{k+s-1})$.
For any $h$ in $B_{R}:=\{h\in G\colon d_{G}(h,id_{G})\leq R\}$, by (\ref{4:3e}),  we have that 
 $$\int_{(\x,\x')\in (G^{V_{X}}\times G^{V_{X}})/(\Gamma^{V_{X}}\times \Gamma^{V_{X}}), x_{0}=hx'_{0}}\tilde{\Phi}_{(1+3\e'')\rho,\e''}(\x)\tilde{\Phi}_{(1+3\e'')\rho,\e''}(\x')\geq S'_{\rho',\e''}.$$ 
Integrating both side over $B_{R}$, we deduce that
$$\int_{(\x,\x')\in (G^{V_{X}}\times G^{V_{X}})/(\Gamma^{V_{X}}\times \Gamma^{V_{X}})}\lambda(\x,\x')\tilde{\Phi}_{(1+3\e'')\rho,\e''}(\x)\tilde{\Phi}_{(1+3\e'')\rho,\e''}(\x')\geq m_{G}(B_{R})S'_{\rho',\e''},$$
where  $\lambda(\x,\x')$ is the number of $h\in B_{R}$ such that $x_{0}\Gamma=hx'_{0}\Gamma$ and $m_{G}$ is the Haar measure of $G$. In other words, $\lambda(\x,\x')$ equals to the cardinality of the set $\Gamma\cap x_{0}^{-1}B_{R}x'_{0}$. Choose any representatives of $x_{0}$ and $x_{0}'$ in some fundamental domain of complexity $O_{C_{0},C,d}(1)$, and use a  volume-packing argument and some simple geometry, one can show that 
$$\lambda(\x,\x')=m_{G}(B_{R})(1+o_{R\to\infty;C_{0},C,d}(1)).$$
Comparing with the above we have
$$(1+o_{R\to\infty;C_{0},C,k,s}(1))S_{(1+3\e'')\rho,\e''}^{2}\geq S'_{\rho,\e''}.$$
So by (\ref{4:epe}),
we have that 
$$(1+\e')^{2}(1+o_{R\to\infty;C_{0},C,d}(1))S_{\rho,\e''}^{2}\geq S'_{\rho,\e''}.$$
This proves one side of (\ref{4:WS4}), and the proof for the other side is similar.
We are done.
\end{proof}

We now continue to prove Theorem \ref{4:mainmain2}.
First note that 
\begin{equation}\label{4:mm0}
\begin{split}
&\quad\Vert f_{\str}\Vert_{L^{1}(V(M-r))}\geq \Vert \bold{1}_{E}\Vert_{L^{1}(V(M-r))}-\Vert f_{\uni}\Vert_{L^{1}(V(M-r))}-\Vert f_{\err}\Vert_{L^{1}(V(M-r))}
\\&\geq \Vert \bold{1}_{E}\Vert_{L^{1}(V(M-r))}-\Vert f_{\uni}\Vert_{U^{s}(V(M-r))}-\Vert f_{\err}\Vert_{L^{2}(V(M-r))}\geq \e-\mathcal{F}(C)^{-1}-\e'.
\end{split}
\end{equation}
Let $f':=f_{\str}+f_{\err}$.
By Theorem \ref{4:vdcc} and the triangle inequality, if $d\geq d_{0}(X)$,   then 
\begin{equation}\label{4:mm3}
\begin{split}
\vert\E_{(x_{0},\dots,x_{k+s-1})\in \Omega}\bold{1}_{E}^{\otimes (k+s)}(x_{0},\dots,x_{k+s-1})-\E_{(x_{0},\dots,x_{k+s-1})\in \Omega}{f'}^{\otimes (k+s)}(x_{0},\dots,x_{k+s-1})\vert
 \leq 2^{k+s}\mathcal{F}(C)^{-1}
\end{split}
\end{equation} 
provided that $p\gg_{C,d,\mathcal{F}} 1$. 

By Proposition \ref{4:recset}, if $\mathcal{F}$ grows sufficiently fast depending only on $C_{0},d,\e'$ (which will be eventually dependent only on $C_{0},d,\e$ since we will choose $\e'$ to be dependent only on these quantities) and  $p\gg_{C_{0},C,d,\e',\mathcal{F}} 1$, then there exists  a function $\omega\colon \Omega\to \mathbb{R}$ such that the following holds:
	\begin{enumerate}[(i)]
		\item $0\leq\omega(\x)\leq O_{C_{0},C,d,\e'}(1)$ for all $\x\in\Omega$;
		\item $\E_{\x\in\Omega}\omega(\x)=1$; 
		\item $\vert\E_{\x\in \Omega\times_{i}\Omega}(\omega\times_{i}\omega)(\x)-1\vert\leq O_{C_{0},d}(\e')$ for all $0\leq i\leq k+s-1$; 
		\item 	$\vert f_{\str}(x_{i})-f_{\str}(x_{i'})\vert\leq \e'$ for all $0\leq i,i'\leq k+s-1$ if $\omega(x_{0},\dots,x_{k+s-1})>0$.
	\end{enumerate}
Since $0\leq f'\leq 1$ and $0\leq\omega(\x)\leq O_{C_{0},C,d,\e'}(1)$ for all $\x\in\Omega$,
\begin{equation}\label{4:mm4}
\begin{split}
\E_{\x\in \Omega}{f'}^{\otimes (k+s)}(\x)
\geq O_{C_{0},C,d,\e'}(1)^{-1}\E_{\x\in \Omega}{f'}^{\otimes (k+s)}(\x)\omega(\x).
\end{split}
\end{equation}

We need some definitions in order the apply Theorem \ref{4:ct} to the sets $\Omega$ and $\Omega\times_{i}\Omega$.
For any generating set $I$ of $X$,
	by Lemma \ref{4:kk2d} (ii),   $\Omega_{I}$  is a nice and consistent $M$-set of total co-dimension   $\frac{(k+2)(k+1)}{2}$. 
So we may write $\Omega_{I}=V(\mathcal{J}_{I})$ for some consistent $(M,k+1)$-family $\mathcal{J}_{I}\subseteq \F_{p}[(x_{i})_{i\in I}]$ of total dimension $\frac{(k+2)(k+1)}{2}$. For any $i\in I$, let $(\mathcal{J}'_{I,i},\mathcal{J}''_{I,i})$ be an  $\{x_{i}\}$-decomposition of $\mathcal{J}_{I}$. It is clear that  
\begin{equation}\label{4:mmm1}
\begin{split}
\text{the set of $x_{i}\in\V$  such that $(x_{i})_{i\in I}\in V(\mathcal{J}'_{I,i})$ for all $(x_{i'})_{i'\in I\backslash\{i\}}\in (\V)^{k}$ is equal to $V(M-r)$.}
\end{split}
\end{equation} 

Similarly, for any generating set $I$ of $X$ and any $i\in I$,
	by Lemma \ref{4:kk3d} (ii),   $\Omega_{I}\times_{i}\Omega_{I}$  is a nice and consistent $M$-set of total co-dimension   $k^{2}+3k+1$. 
So we may write $\Omega_{I}\times_{i}\Omega_{I}=V(\tilde{\mathcal{J}}_{I,i})$ for some consistent $(M,2k+1)$-family $\tilde{\mathcal{J}}_{I,i}\subseteq \F_{p}[x_{i},(x^{+}_{i'})_{i'\in I\backslash\{i\}},(x^{-}_{i'})_{i'\in I\backslash\{i\}}]$ of total dimension $k^{2}+3k+1$. Let $(\tilde{\mathcal{J}}'_{I,i},\tilde{\mathcal{J}}''_{I,i})$ be an  $\{x_{i}\}$-decomposition of $\tilde{\mathcal{J}}_{I,i}$. It is clear that  
\begin{equation}\label{4:mmm2}
\begin{split}
\text{the set of $x_{i}\in\V$  such that $(x_{i},\x^{+},\x^{-})\in V(\tilde{\mathcal{J}}'_{I,i})$ for all $\x^{+},\x^{-}\in (\V)^{k}$ is equal to $V(M-r)$,}
\end{split}
\end{equation}
and that for all $x_{i}\in V(M-r)$, $\x^{+},\x^{-}\in (\V)^{k}$,
\begin{equation}\label{4:mmm3}
\begin{split}
(x,\x^{+},\x^{-})\in V(\tilde{\mathcal{J}}''_{I,i}) \Leftrightarrow (x\da_{i\ca I}\x^{+}),(x\da_{i\ca I}\x^{-})\in V(\mathcal{J}''_{I,i}).
\end{split}
\end{equation}

For any functions $f_{0},\dots,f_{k+s-1}$ taking values in $[0,1]$ with $f_{i}=f_{\err}$ for some $0\leq i\leq k+s-1$,  by Theorem \ref{4:ct}, Lemmas \ref{4:kk2d} and \ref{4:kk3d},  relations (\ref{4:mmm1}), (\ref{4:mmm2}), (\ref{4:mmm3}), and the Cauchy-Schwartz inequality, we have that (taking $I$ to be any generating set containing $i$)
\begin{equation}\label{4:mm5}
\begin{split}
&\quad\vert\E_{(x_{0},\dots,x_{k+s-1})\in \Omega}f_{0}(x_{0})f_{1}(x_{1})\dots f_{k+s-1}(x_{k+s-1})\omega(x_{0},\dots,x_{k+s-1})\vert
\\&\leq \E_{(x_{0},\dots,x_{k+s-1})\in \Omega}\vert f_{\err}(x_{i})\vert\cdot\omega(x_{0},\dots,x_{k+s-1})
\\&=\E_{(x_{i'})_{i'\in I}\in \Omega_{I}}\vert f_{\err}(x_{i})\vert\cdot\omega(\L_{I}((x_{i'})_{i'\in I}))
\\&=\E_{x_{i}\in V(M-r)}\vert f_{\err}(x_{i})\vert\cdot \E_{\x'\in (\V)^{k}\colon (x_{i}\da_{i\ca I}\x')\in V(\mathcal{J}''_{I,i})}\omega(\L_{I}(x_{i}\da_{i\ca I}\x'))+O_{C_{0},C,d,\e'}(p^{-1/2})
\\&\leq (\vert\E_{x_{i}\in V(M-r)}\vert f_{\err}(x_{i})\vert^{2})^{1/2}\cdot(\vert\E_{x_{i}\in V(M-r)}\vert\E_{\x'\in (\V)^{k}\colon (x_{i}\da_{i\ca I}\x')\in V(\mathcal{J}''_{I,i})}\omega(\L_{I}(x_{i}\da_{i\ca I}\x'))\vert^{2})^{1/2}+\e'
\\&\leq \e'(\E_{x_{i}\in V(M-r)}\E_{\x',\x''\in (\V)^{k}\colon (x_{i}\da_{i\ca I}\x'),(x_{i}\da_{i\ca I}\x'')\in V(\mathcal{J}''_{I,i})}\omega(\L_{I}(x_{i}\da_{i\ca I}\x'))\omega(\L_{I}(x_{i}\da_{i\ca I}\x'')))^{1/2}+\e'
\\&=\e' (\E_{\x\in \Omega_{I}\times_{i}\Omega_{I}}(\omega\times_{i}\omega)(\L^{\pm}_{i\ca I}(\x))+O_{C_{0},C,d,\e'}(p^{-1/2}))^{1/2}+\e'
\\&=\e' (\E_{\x\in \Omega\times_{i}\Omega}(\omega\times_{i}\omega)(\x)+O_{C_{0},C,d,\e'}(p^{-1/2}))^{1/2}+\e'
\\&\leq \e' (1+O_{C_{0},d}(\e')+O_{C_{0},C,d,\e'}(p^{-1/2}))^{1/2}+\e'\leq 3\e'
\end{split}
\end{equation} 
provided that $\e'$ is sufficiently small depending on $C_{0},d$ and that $p\gg_{C_{0},C,d,\e'} 1$. 

Let $I$ be any generating set of $X$ containing 0. From now on we assume without loss of generality that $I=\{0,\dots,k\}$.
Since $\omega(\x)>0$ implies that $\vert f_{\str}(x_{i})-f_{\str}(x_{i'})\vert\leq \e'$ for all $0\leq i,i'\leq k+s-1$,  it follows from Theorem \ref{4:ct}, Lemma \ref{4:kk2d} and (\ref{4:mmm1}) that 
\begin{equation}\label{4:mm6}
\begin{split}
&\quad\vert\E_{(x_{0},\dots,x_{k+s-1})\in \Omega}f_{\str}(x_{0})f_{\str}(x_{1})\dots f_{\str}(x_{k+s-1})\omega(x_{0},\dots,x_{k+s-1})\vert
\\&\geq  \E_{(x_{0},\dots,x_{k+s-1})\in \Omega}(f_{\str}(x_{0})^{k+s}-(k+s)\e')\omega(x_{0},\dots,x_{k+s-1})
\\&=\E_{(x_{0},\dots,x_{k})\in \Omega_{I}}f_{\str}(x_{0})^{k+s}\omega(\L_{I}(x_{0},\dots,x_{k}))-(k+s)\e'
\\&=\E_{x_{0}\in V(M-r)}f_{\str}(x_{0})^{k+s}\E_{\x'\in(\V)^{k}\colon (x_{0},\x')\in V(\mathcal{J}''_{I,0})}\omega(\L_{I}(x_{0},\x'))-(k+s)\e'+O_{C_{0},C,d,\e'}(p^{-1/2}).
\end{split}
\end{equation}

By the construction of $\omega$, Theorem \ref{4:ct} and Lemma \ref{4:kk2d}, we have
\begin{equation}\label{4:tteemmm1}
\begin{split}
&\quad\E_{x_{0}\in V(M-r)} \E_{\x'\in(\V)^{k}\colon (x_{0},\x')\in V(\mathcal{J}''_{I,0})}\omega(\L_{I}(x_{0},\x'))
 =\E_{\x\in\Omega_{I}}\omega(\L_{I}(\x))+O_{C_{0},C,d,\e'}(p^{-1/2})
\\&=\E_{\x\in\Omega}\omega(\x)+O_{C_{0},C,d,\e'}(p^{-1/2})
=1+O_{C_{0},C,d,\e'}(p^{-1/2}).
\end{split}
\end{equation} 
So by the construction of $\omega$, Theorem \ref{4:ct}, Lemmas \ref{4:kk2d} and \ref{4:kk3d}, and relations (\ref{4:mmm1}), (\ref{4:mmm2}), (\ref{4:mmm3}), (\ref{4:tteemmm1}), we have that 
\begin{equation}\nonumber
\begin{split}
&\quad\E_{x_{0}\in V(M-r)}\vert\E_{\x'\in(\V)^{k}\colon (x_{0},\x')\in V(\mathcal{J}''_{I,0})}\omega(\L_{I}(x_{0},\x'))-1\vert^{2}
\\&=\E_{x_{0}\in V(M-r)}\vert\E_{\x'\in(\V)^{k}\colon (x_{0},\x')\in V(\mathcal{J}''_{I,0})}\omega(\L_{I}(x_{0},\x'))\vert^{2}+1-2\E_{x_{0}\in V(M-r)} \E_{\x'\colon (x_{0},\x')\in V(\mathcal{J}''_{I,0})}\omega(\L_{I}(x_{0},\x'))
\\&=\E_{x_{0}\in V(M-r)}\E_{\x',\x''\in(\V)^{k}\colon (x_{0},\x')(x_{0},\x'')\in V(\mathcal{J}''_{I,0})}\omega(\L_{I}(x_{0},\x'))\omega(\L_{I}(x_{0},\x''))-1+O_{C_{0},C,d,\e'}(p^{-1/2})
\\&=\E_{(x_{0},\x',\x'')\in \Omega_{I}\times_{0}\Omega_{I}}(\omega\times_{0}\omega)(\L^{\pm}_{0\ca I}(x_{0},\x',\x''))-1+O_{C_{0},C,d,\e'}(p^{-1/2})
\\&=\E_{\x\in \Omega\times_{0}\Omega}(\omega\times_{0}\omega)(\x)-1+O_{C_{0},C,d,\e'}(p^{-1/2})
\\&=O_{C_{0},d}(\e')+O_{C_{0},C,d,\e'}(p^{-1/2})=O_{C_{0},d}(\e').
\end{split}
\end{equation} 
So by the Pigeonhole Principle, there exists a subset $U$ of $V(M-r)$ with $\vert U\vert=O_{C_{0},d}({\e'}^{1/3}\vert V(M-r)\vert)$ such that for all $x_{0}\in V(M-r)\backslash U$, 
\begin{equation}\label{4:mm65s}
\begin{split}
\vert \E_{\x'\in (\V)^{k}\colon (x_{0},\x')\in V(\mathcal{J}''_{I,0})}\omega(\L_{I}(x_{0},\x'))-1\vert=O_{C_{0},d}({\e'}^{1/3}).
\end{split}
\end{equation} 
  It then follows from (\ref{4:mm0}), (\ref{4:mm6}), (\ref{4:mm65s}) and the Cauchy-Schwartz inequality that
\begin{equation}\label{4:mm65}
\begin{split}
&\quad\vert\E_{(x_{0},\dots,x_{k+s-1})\in \Omega}f_{\str}(x_{0})f_{\str}(x_{1})\dots f_{\str}(x_{k+s-1})\omega(x_{0},\dots,x_{k+s-1})\vert
\\&\geq  \E_{x_{0}\in V(M-r)}\bold{1}_{V(M-r)\backslash U}(x_{0})f_{\str}(x_{0})^{k+s}(1+O_{C_{0},d}({\e'}^{1/3}))+O_{C_{0},d}({\e'})
\\&=\E_{x_{0}\in V(M-r)}f_{\str}(x_{0})^{k+s}+O_{C_{0},d}({\e'}^{1/3})
\\&\geq (\E_{x_{0}\in V(M-r)}f_{\str}(x_{0}))^{k+s}+O_{C_{0},d}({\e'}^{1/3})
\\&\geq (\e-\mathcal{F}(C)^{-1}-\e')^{k+s}+O_{C_{0},d}({\e'}^{1/3})
\end{split}
\end{equation} 

Combining (\ref{4:mm3}), (\ref{4:mm4}), (\ref{4:mm5}) and (\ref{4:mm65}), we have that 
\begin{equation}\nonumber
\begin{split}
&\quad\E_{(x_{0},\dots,x_{k+s-1})\in \Omega}\bold{1}_{E}(x_{0})\bold{1}_{E}(x_{1})\dots \bold{1}_{E}(x_{k+s-1})
\\&\geq
O_{C_{0},C,d,\e'}(1)^{-1}\Bigl((\e-\mathcal{F}(C)^{-1}-\e')^{k+s}+O_{C_{0},d}({\e'}^{1/3})\Bigr)
-2^{k+s}\mathcal{F}(C)^{-1}.
\end{split}
\end{equation} 
If we first chose $\e'$ to be sufficiently small depending on $C_{0},d,\e$, then $\mathcal{F}$ grow sufficiently fast depending on $C_{0},d,\e$, and 
 then $p$ sufficiently large depending on $C_{0},C,d,\e,\mathcal{F}$,  we will derive a lower bound
\begin{equation}\nonumber
\begin{split}
\E_{(x_{0},\dots,x_{k+s-1})\in \Omega}\bold{1}_{E}(x_{0})\bold{1}_{E}(x_{1})\dots \bold{1}_{E}(x_{k+s-1})>\d(C_{0},C,d,\e)>0.
\end{split}
\end{equation} 
This proves (\ref{4:mmm1}) since $C$ is dependent on $d,\e',\mathcal{F}$ and thus on $C_{0},d,\e$. This finishes the proof of Theorem \ref{4:mainmain2}.

\appendix

\section{Background materials for polynomials and nilmanifolds}\label{4:s:AppA}

In this appendix, we recall the definitions in \cite{SunA} on polynomials and nilmanifolds.

 \subsection{Polynomials}\label{4:s:AppA1}

\begin{defn}[Polynomials in finite field]
	Let $\poly(\V\to\F^{d'}_{p})$ be the collection of all functions $f\colon \V\to\F_{p}$ of the form
	\begin{equation}\label{1:p1}
	f(n_{1},\dots,n_{d})=\sum_{0\leq a_{1},\dots,a_{d}\leq p-1, a_{i}\in\N}C_{a_{1},\dots,a_{d}}n^{a_{1}}_{1}\dots n^{a_{d-1}}_{d}
	\end{equation}	
	for some $C_{a_{1},\dots,a_{d}}\in \F_{p}^{d'}$. 
	Let $f\in\poly(\V\to\F^{d'}_{p})$ be a function which is not constant zero.
	The \emph{degree of $f$} (denoted as $\deg(f)$) is the largest $r\in\N$ such that $C_{a_{1},\dots,a_{d}}\neq \bold{0}$ for some $a_{1}+\dots+a_{d}=r$. We say that $f$ is \emph{homogeneous of degree $r$} if $C_{a_{1},\dots,a_{d}}\neq \bold{0}$ implies that $a_{1}+\dots+a_{d}=r$.	 We say that $f$ is a \emph{linear transformation} if $f$ is homogeneous of degree 1.
\end{defn}	

\begin{conv}\label{1:c00}
       For convenience, the \emph{degree} of the constant zero function is allowed to take any integer value. For example, 0 can be regarded as a homogeneous polynomial of degree 10, as a linear transformation, or as a polynomial of degree -1. Here we also adopt the convention the   only  polynomial of negative degree is 0.  
\end{conv}

 \begin{defn}[Partially $p$-periodic polynomials]\label{1:polygeneral}
     For any subset $R$ of $\R$,
     let $\poly(\Z^{d}\to R)$  be the collection of all polynomials in $\Z^{d}$ taking values in $R$. For any prime $p$, subset $\Omega$ of $\Z^{d}$, and sets  $R'\subseteq R\subseteq \R$, let $\poly(\Omega\to R\vert R')$ denote the set of all $f\in\poly(\Z^{d}\to R)$ such that $f(n)\in R$ for all $n\in\Z^{d}$ and that $f(n)\in R'$ for all $n\in\Omega+p\Z^{d}$.
     Let $\poly_{p}(\Omega\to R\vert R')$ denote the set of all $f\in\poly(\Z^{d}\to R)$ such that $f(n)\in R$ for all $n\in\Z^{d}$ and that $f(n+pm)-f(n)\in R'$ for all $n\in\Omega+p\Z^{d}$ and $m\in\Z^{d}$.  
     
     A polynomial $f$ in $\poly(\Omega\to \R\vert\Z)$ is called a \emph{partially $p$-periodic polynomial on $\Omega$}.
 \end{defn}

\subsection{Nilmanifolds, filtrations and polynomial sequences}\label{4:s:AppA2}

Let $G$ be a group and $g,h\in G$. Denote $[g,h]:=g^{-1}h^{-1}gh$. For subgroups $H,H'$ of $G$, let $[H,H']$ denote the group generated by $[h,h']$ for all $h\in H$ and $h'\in H'$.

\begin{defn}[Filtered group]
	Let $G$ be a group. An \emph{$\N$-filtration} on $G$ is a collection $G_{\N}=(G_{i})_{i\in \N}$ of subgroups of $G$ indexed by $\N$ with $G_{0}=G$  such that  the following holds:
	\begin{enumerate}[(i)]
		\item for all $i,j\in \N$ with $i\leq j$, we have that $G_{i}\supseteq G_{j}$;
		\item  for all $i,j\in \N$, we have $[G_{i},G_{j}]\subseteq G_{i+j}$.
	\end{enumerate}	
	For $s\in \N$, we say that $G$ is an \emph{($\N$-filtered) nilpotent group} of \emph{degree} at most $s$ (or of \emph{degree} $\leq s$) with respect to some $\N$-filtration  $(G_{i})_{i\in \N}$ if $G_{i}$ is trivial whenever $i>s$.\footnote{Unlike the convention in literature, for our convenience, we do not require $G_{0}=G$ to be the same as $G_{1}$.}
 \end{defn}

\begin{defn}[Nilmanifold]
	Let $\Gamma$ be a discrete and cocompact subgroup of a connected, simply-connected nilpotent Lie group $G$ with filtration $G_{\N}=(G_{i})_{i\in \N}$ such that $\Gamma_{i}:=\Gamma\cap G_{i}$ is a cocompact subgroup of $G_{i}$\footnote{In some papers, such $\Gamma_{i}$ is called a \emph{rational} subgroup of $G$.} for all $i\in \N$.
	Then we say that $G/\Gamma$ is an \emph{($\N$-filtered) nilmanifold}, and we use $(G/\Gamma)_{\N}$ to denote the collection $(G_{i}/\Gamma_{i})_{i\in \N}$ (which is called the \emph{$\N$-filtration} of $G/\Gamma$). We say that $G/\Gamma$ has degree $\leq s$  with respect to $(G/\Gamma)_{\N}$ if $G$ has degree $\leq s$  with respect to $G_{\N}$.
\end{defn}

\begin{defn}[Sub-nilmanifold]\label{1:id1}
	Let $G/\Gamma$ be an $\N$-filtered nilmanifold of degree $\leq s$ with filtration $G_{\N}$ and $H$ be a rational subgroup of $G$. Then $H/(H\cap \Gamma)$ is also  an $\N$-filtered nilmanifold of degree $\leq s$ with the filtration $H_{\N}$ given by $H_{i}:=G_{i}\cap H$ for all $i\in \N$ (see Example 6.14 of \cite{GTZ12}). We say that  $H/(H\cap \Gamma)$ is a \emph{sub-nilmanifold} of $G/\Gamma$, $H_{\N}$ (or $(H/(H\cap \Gamma))_{\N}$) is the filtration \emph{induced by} $G_{\N}$ (or $(G/\Gamma)_{\N}$).
\end{defn}	

\begin{defn}[Quotient nilmanifold]\label{1:id2}
	Let $G/\Gamma$ be an $\N$-filtered nilmanifold of degree $\leq s$ with filtration $G_{\N}$ and $H$ be a normal subgroup of $G$. Then $(G/H)/(\Gamma/(\Gamma\cap H))$ is also  an $\N$-filtered nilmanifold of most $\leq s$ with the filtration $(G/H)_{\N}$ given by $(G/H)_{i}:=G_{i}/(H\cap G_{i})$ for all $i\in \N$. We say that  $(G/H)/(\Gamma/(\Gamma\cap H))$ is the \emph{quotient nilmanifold} of $G/\Gamma$ by $H$ and that $(G/H)_{\N}$ is the filtration \emph{induced by} $G_{\N}$.
\end{defn}	

\begin{defn}[Product nilmanifold]\label{1:id4}
	Let $G/\Gamma$ and  $G'/\Gamma'$ be  $\N$-filtered nilmanifolds of degree $\leq s$ with filtration $G_{\N}$ and $G'_{\N}$. Then $G\times G'/\Gamma\times\Gamma'$ is also  an $\N$-filtered nilmanifold of most $\subseteq J$ with the filtration $(G\times G'/\Gamma\times\Gamma')_{\N}$ given by $(G\times G'/\Gamma\times\Gamma')_{i}:=G_{i}\times G'_{i}/\Gamma_{i}\times \Gamma'_{i}$ for all $i\in \N$. We say that  $G\times G'/\Gamma\times\Gamma'$ is the \emph{product nilmanifold} of $G/\Gamma$ and $G'/\Gamma'$ and that $(G\times G'/\Gamma\times\Gamma')_{\N}$ is the filtration \emph{induced by} $G_{\N}$ and $G'_{\N}$.
\end{defn}

Every nilmanifold has an explicit algebraic description by using the Mal'cev basis:

\begin{defn} [Mal'cev basis]\label{1:Mal}
	Let $s\in\N_{+}$, $G/\Gamma$ be a  nilmanifold  of step at most $s$ with the $\N$-filtration  $(G_{i})_{i\in\N}$. Let $\dim(G)=m$ and $\dim(G_{i})=m_{i}$ for all $0\leq i\leq s$. A basis $\mathcal{X}:=\{X_{1},\dots,X_{m}\}$ for the Lie algebra $\log G$ of $G$ (over $\mathbb{R}$) is a \emph{Mal'cev basis} for $G/\Gamma$ adapted to the filtration $G_{\N}$ if
	\begin{itemize}
		\item for all $0\leq j\leq m-1$, $\log H_{j}:=\text{Span}_{\mathbb{R}}\{\xi_{j+1},\dots,\xi_{m}\}$ is a Lie algebra ideal of $\log G$ and so $H_{j}:=\exp(\log H_{j})$  is a normal Lie subgroup of $G;$
		\item $G_{i}=H_{m-m_{i}}$ for all $0\leq i\leq s$;
		\item the map $\psi^{-1}\colon \mathbb{R}^{m}\to G$ given by
		\begin{equation}\nonumber
		\psi^{-1}(t_{1},\dots,t_{m})=\exp(t_{1}X_{1})\dots\exp(t_{m}X_{m})
		\end{equation}	
		is a bijection;
		\item $\Gamma=\psi^{-1}(\Z^{m})$.
	\end{itemize}	
	We call $\psi$ the  \emph{Mal'cev coordinate map} with respect to the Mal'cev basis $\mathcal{X}$.
	If $g=\psi^{-1}(t_{1},\dots,t_{m})$, we say that $(t_{1},\dots,t_{m})$ are the \emph{Mal'cev coordinates} of $g$ with respect to $\mathcal{X}$. 
	
	We say that the Mal'cev basis $\mathcal{X}$  is \emph{$C$-rational} (or of \emph{complexity} at most $C$)  if all the structure constants $c_{i,j,k}$ in the relations
	   	$$[X_{i},X_{j}]=\sum_{k}c_{i,j,k}X_{k}$$
	   	are rational with complexity at most $C$.
\end{defn}

It is known that for every filtration $G_{\bullet}$ which is rational for $\Gamma$, there exists a Mal'cev basis adapted to it. See for example the discussion on pages 11--12 of \cite{GT12b}.

We use the following quantities to describe the complexities of the objected defined above.

 \begin{defn}[Notions of complexities for nilmanifolds]
Let $G/\Gamma$ be a nilmanifold with filtration $G_{I}$ and a Mal'cev basis $\mathcal{X}=\{X_{1},\dots,X_{D}\}$ adapted to it. 
We say that  $G/\Gamma$ is of \emph{complexity} at most $C$ if the  Mal'cev basis $\mathcal{X}$ is $C$-rational and $\dim(G)\leq C$. 

 An element $g\in G$ is of \emph{complexity} at most $C$ (with respect to the Mal'cev coordinate map $\psi\colon G/\Gamma\to\R^{m}$) if $\psi(g)\in [-C,C]^{m}$.

Let $G'/\Gamma'$ be a nilmanifold  endowed with the Mal'cev basis  $\mathcal{X}'=\{X'_{1},\dots,X'_{D'}\}$ respectively. Let $\phi\colon G/\Gamma\to G'/\Gamma'$ be a filtered homomorphism, we say that $\phi$ is of  \emph{complexity} at most $C$ if the map $X_{i}\to\sum_{j}a_{i,j}X'_{j}$ induced by $\phi$ is such that all $a_{i,j}$ are of complexity at most $C$.

 Let $G'\subseteq G$  be a closed connected subgroup. We say that  $G'$ is \emph{$C$-rational} (or of \emph{complexity} at most $C$) relative to $\mathcal{X}$ if the Lie algebra $\log G$ has a basis consisting of linear combinations $\sum_{i}a_{i}X_{i}$ such that $a_{i}$ are rational numbers of complexity at most $C$.
\end{defn}

\begin{conv}
 Throughout this paper, all nilmanifolds are assumed to have a fixed filtration, Mal'cev basis and a smooth Riemannian metric induced by the Mal'cev basis. Therefore, we will simply say that a nilmanifold, Lipschitz function, sub-nilmanifold etc. is of complexity $C$ without mentioning the reference filtration and Mal'cev basis.
\end{conv}

\begin{defn}[Polynomial sequences]
	Let $d, k\in\N_{+}$ and $G$ be a connected simply-connected nilpotent Lie group. Let    $(G_{i})_{i\in\N}$ be an $\N$-filtration of $G$. A map $g\colon \Z^{d}\to G$ is a \emph{($\N$-filtered) $d$-integral polynomial sequence} if
		$$\Delta_{h_{m}}\dots \Delta_{h_{1}} g(n)\in G_{m}$$
		 for all $m\in\N$ and $n,h_{1},\dots,h_{m}\in \Z^{d}$.\footnote{Recall that $\Delta_{h}g(n):=g(n+h)g(n)^{-1}$
 for all $n, h\in H$.}	
	
    The set of all $\N$-filtered polynomial sequences is denoted by $\poly(\Z^{d}\to G_{\N})$.
  \end{defn}

By Corollary B.4 of \cite{GTZ12}, $\poly(\Z^{k}\to G_{\N})$ is a  group with respect to the pointwise multiplicative operation.
We refer the readers to Appendix B of \cite{GTZ12} for more properties for   polynomial sequences.

\begin{defn}[Partially periodic polynomial sequences]
	Let  $k\in\N_{+}$, $G/\Gamma$ be an $\N$-filtered nilmanifold, $p$ be a prime, and $\Omega$ be a subset of $\Z^{k}$.
	Let $\poly(\Omega\to G_{\N}\vert\Gamma)$ denote the set of all $g\in \poly(\Z^{k}\to G_{\N})$ such that $g(n)\in \Gamma$ for all $n\in\Omega$, and $\poly_{p}(\Omega\to G_{\N}\vert\Gamma)$ denote the set of all $g\in \poly(\Z^{k}\to G_{\N})$ such that $g(n+pm)^{-1}g(n)\in \Gamma$ for all $n\in\Omega$ and $m\in\Z^{k}$.

Let
 $\poly(\F_{p}^{k}\to G_{\N})$ denote the set of all functions of the form $f\circ\tau$ for some $f\in \poly(\Z^{k}\to G_{\N})$.
 For a subset $\Omega$ of $\F_{p}^{k}$, let
 $\poly_{p}(\Omega\to G_{\N}\vert\Gamma)$ denote the set of all functions of the form $f\circ\tau$ for some $f\in \poly_{p}(\iota^{-1}(\Omega)\to G_{\N}\vert\Gamma)$. 
 We define $\poly(\Omega\to G_{\N}\vert\Gamma)$  similarly.

 For $\Omega\subseteq\Z^{k}$
 or $\F_{p}^{k}$, we call functions in $\poly_{p}(\Omega\to G_{\N}\vert\Gamma)$ \emph{partially $p$-periodic polynomial sequences on $\Omega$.}
 \end{defn}

Recall from \cite{SunA} that $\poly_{p}(\Omega\to G_{\N}\vert \Gamma)$ is not closed under multiplication. However, we have

\begin{lem}[Lemma 
3.17 of \cite{SunA}]\label{4:grg}
	Let  $k\in\N_{+}$, $G/\Gamma$ be an $\N$-filtered nilmanifold, $p$ be a prime, and  $\Omega$ be a subset of $\Z^{k}$. For all $f\in\poly_{p}(\Omega\to G_{\N}\vert \Gamma)$ and $g\in\poly(\Omega\to G_{\N}\vert \Gamma)$, we have that $fg\in \poly_{p}(\Omega\to G_{\N}\vert \Gamma)$.
\end{lem}

\subsection{The Baker-Campbell-Hausdorff formula}\label{4:s:AppA3}	

The material of this section comes from Appendix C of \cite{GT10b}. We write it down for completeness.

Let $G$ be a group, $t\in\N_{+}$ and $g_{1},\dots,g_{t}\in G$. The \emph{iterated commutator} of $g_{1}$ is defined to be $g_{1}$ itself. Iteratively, we define an \emph{iterated commutator} of $g_{1},\dots,g_{t}$ to be an element of the form $[w,w']$, where $w$ is an iterated commutator of $g_{i_{1}},\dots,g_{i_{r}}$, $w'$ is an iterated commutator of $g_{i'_{1}},\dots,g_{i'_{r'}}$ for some $1\leq r,r'\leq t-1$ with $r+r'=t$ and $\{i_{1},\dots,i_{r}\}\cup \{i'_{1},\dots,i'_{r'}\}=\{1,\dots,t\}$.

Similarly, let $X_{1},\dots,X_{t}$ be elements of a Lie algebra. The \emph{iterated Lie bracket} of $X_{1}$ is defined to be $X_{1}$ itself. Iteratively, we define an \emph{iterated Lie bracket} of $X_{1},\dots,X_{t}$ to be an element of the form $[w,w']$, where $w$ is an iterated Lie bracket of $X_{i_{1}},\dots,X_{i_{r}}$, $w'$ is an iterated Lie bracket of $X_{i'_{1}},\dots,X_{i'_{r'}}$ for some $1\leq r,r'\leq t-1$ with $r+r'=t$ and $\{i_{1},\dots,i_{r}\}\cup \{i'_{1},\dots,i'_{r'}\}=\{1,\dots,t\}$.

Let $G$ be a connected and simply connected nilpotent Lie group. The \emph{Baker-Campbell-Hausdorff formula} asserts that for all $X_{1},X_{2}\in\log G$, we have
$$\exp(X_{1})\exp(X_{2})=\exp\Bigl(X_{1}+X_{2}+\frac{1}{2}[X_{1},X_{2}]+\prod_{\alpha}c_{\alpha}X_{\alpha}\Bigr),$$	
where $\alpha$ is a finite set of labels, $c_{\alpha}$ are real constants, and $X_{\alpha}$ are iterated Lie brackets with $k_{1,\alpha}$ copies of $X_{1}$  and $k_{2,\alpha}$ copies of $X_{2}$ for some $k_{1,\alpha},k_{2,\alpha}\geq 1$ and $k_{1,\alpha}+k_{2,\alpha}\geq 3$. One may use this formula to show that for all $g_{1},g_{2}\in G$ and $x\in \R$, we have that
\begin{equation}\nonumber\label{1:C1}
(g_{1}g_{2})^{x}=g_{1}^{x}g_{2}^{x}\prod_{\alpha}g_{\alpha}^{Q_{\alpha}(x)},
\end{equation}	
where $\alpha$ is a finite set of labels,   $g_{\alpha}$ are iterated commutators with $k_{1,\alpha}$ copies of $g_{1}$  and $k_{2,\alpha}$ copies of $X_{2}$ for some $k_{1,\alpha},k_{2,\alpha}\geq 1$, and $Q_{\alpha}\colon\R\to\R$ are polynomials of degrees at most $k_{1,\alpha}+k_{2,\alpha}$ without constant terms.

Similarly, one can show that for any $g_{1},g_{2}\in G$ and $x_{1},x_{2}\in \R$, we have that	
\begin{equation}\label{4:C2}
[g_{1}^{x_{1}},g_{2}^{x_{2}}]=[g_{1},g_{2}]^{x_{1}x_{2}}\prod_{\alpha}g_{\alpha}^{P_{\alpha}(x_{1},x_{2})},
\end{equation}	
where $\alpha$ is a finite set of labels,   $g_{\alpha}$ are iterated commutators with $k_{1,\alpha}$ copies of $g_{1}$  and $g_{2,\alpha}$ copies of $X_{2}$ for some $k_{1,\alpha},k_{2,\alpha}\geq 1$, $k_{1,\alpha}+k_{2,\alpha}\geq 3$, and $P_{\alpha}\colon\R^{2}\to\R$ are polynomials of degrees at most $k_{1,\alpha}$ in $x_{1}$ and at most $k_{2,\alpha}$ in $x_{2}$ which vanishes when $x_{1}x_{2}=0$.
	
\subsection{Type-I horizontal  torus and character}

\begin{defn}[Type-I horizontal torus and character]
Let $G/\Gamma$ be a nilmanifold endowed with a Mal'cev basis $\mathcal{X}$.  The \emph{type-I horizontal torus} of $G/\Gamma$ is $G/[G,G]\Gamma$.
	A \emph{type-I horizontal character} is a continuous homomorphism $\eta\colon G\to \R$ such that $\eta(\Gamma)\subseteq \Z$.
	When written in the coordinates relative to $\mathcal{X}$, we may write $\eta(g)=k\cdot \psi(g)$ for some unique $k=(k_{1},\dots,k_{m})\in\Z^{m}$, where $\psi\colon G\to \R^{m}$ is the coordinate map with respect to the Mal'cev basis  $\mathcal{X}$. We call the quantity $\Vert\eta\Vert:=\vert k\vert=\vert k_{1}\vert+\dots+\vert k_{m}\vert$ the \emph{complexity} of $\eta$ (with respect to $\mathcal{X}$). 
\end{defn}

It is not hard to see that any type-I horizontal character mod $\Z$ vanishes on $[G,G]\Gamma$ and thus descent to a continuous homomorphism between the type-I horizontal torus $G/[G,G]\Gamma$ and $\R/\Z$. Moreover, $\eta \mod \Z$ is a well defined map from $G/\Gamma$ to $\R/\Z$. 

It follows from Lemma 
3.21 of \cite{SunA} that 
  a function $g\colon \Z^{k}\to G$ belongs to $\poly(\Z^{k}\to G_{\N})$  if and only if 
\begin{equation}\label{4:tl}
g(n)=\prod_{i\in\N^{k}, \vert i\vert\leq s}g_{i}^{\binom{n}{i}}
\end{equation}
for some $g_{i}\in G_{\vert i\vert}$ (with any fixed order in $\N^{k}$), where $s$ is the degree of $G$.
We call $g_{i}$ the \emph{($i$-th) type-I Taylor  coefficient} of $g$ (which is uniquely determined by $g$ if $g\in\poly(\Z^{k}\to G_{\N})$ thanks to Lemma 
3.21 of \cite{SunA}), and (\ref{4:tl}) the \emph{type-I Taylor expansion} of $g$.

  
\begin{lem}[Lemma 
   3.28 of \cite{SunA}]\label{4:goodcoordinates}
Let $d,s\in\N_{+}$, $p\gg_{d,s} 1$ be a prime, $\Omega$ be a non-empty $p$-periodic subset of $\Z^{d}$, and $\eta$ be a type-I horizontal character on some $\N$-filtered degree $\leq s$ nilmanifold $G/\Gamma$.
Then for all $g\in\poly_{p}(\Omega\to G_{\N}\vert\Gamma)$, $p^{s}\eta\circ g(n)$  takes values in $\Z+C$ for some $C\in\R$.
\end{lem}

\section{Results for $M$-sets from previous papers}\label{4:s:AppB}

In this appendix, we collect some results proved in previous parts of the series \cite{SunA,SunC} which are used in this paper. We refer the readers to Section \ref{4:s:defn} for the notation used in this section.

\subsection{The rank of quadratic forms}\label{4:s:AppB1}

 \begin{lem}[Lemma 
 4.2 of \cite{SunA}]\label{4:or}
Let $M\colon\V\to\F_{p}$ be a quadratic form associated with the matrix $A$, and $x,y,z\in \V$. Suppose that $M(x)=M(x+y)=M(x+z)=0$. Then $M(x+y+z)=0$ if and only if $(yA)\cdot z=0$.
\end{lem}

Let $M\colon\V\to\F_{p}$ be a quadratic form associated with the matrix $A$ and $V$ be a subspace of $\V$. Let $V^{\pp}$ denote the set of $\{n\in\V\colon (mA)\cdot n=0 \text{ for all } m\in V\}$.
		A subspace $V$ of $\V$ is \emph{$M$-isotropic} if $V\cap V^{\pp}\neq\{\bold{0}\}$. 
	We say that a tuple $(h_{1},\dots,h_{k})$ of vectors in $\V$ is \emph{$M$-isotropic} if the span of $h_{1},\dots,h_{k}$ is an $M$-isotropic subspace.
	We say that a subspace or tuple of vectors is \emph{$M$-non-isotropic} if it is not $M$-isotropic.
 
\begin{prop}[Proposition 
4.8 of \cite{SunA}]\label{4:iissoo}
Let $M\colon \V\to\F_{p}$ be a quadratic form and $V$ be a subspace of $\V$ of co-dimension $r$, and $c\in\V$.
	\begin{enumerate}[(i)]
		\item We have $\dim(V\cap V^{\pp})\leq \min\{d-\rank(M)+r,d-r\}$.
		\item The rank of $M\vert_{V+c}$ equals to $d-r-\dim(V\cap V^{\pp})$ (i.e. $\dim(V)-\dim(V\cap V^{\pp})$). 
		\item The rank of $M\vert_{V+c}$ is at most $d-r$ and at least $\rank(M)-2r$.
		\item $M\vert_{V+c}$ is non-degenerate (i.e. $\rank(M\vert_{V+c})=d-r$) if and only if $V$ is not an $M$-isotropic subspace.
	\end{enumerate}	
\end{prop}

\begin{lem}[Lemma 
4.11 of \cite{SunA}]\label{4:iiddpp}
Let $d,k\in\N_{+}$, $p$ be a prime, and $M\colon\V\to\F_{p}$ be a non-degenerate quadratic form.
	\begin{enumerate}[(i)]
		\item The number of tuples $(h_{1},\dots,h_{k})\in (\V)^{k}$ such that $h_{1},\dots,h_{k}$ are linearly  dependent is at most $kp^{(d+1)(k-1)}$.
		\item The number of $M$-isotropic tuples $(h_{1},\dots,h_{k})\in (\V)^{k}$ is at most $O_{d,k}(p^{kd-1})$.
	\end{enumerate}	
\end{lem}

\subsection{Some basic counting properties}

 \begin{lem}[Lemma 
 4.10 of \cite{SunA}]\label{4:ns}
Let $P\in\poly(\V\to\F_{p})$ be  of degree at most $r$.
	Then $\vert V(P)\vert\leq O_{d,r}(p^{d-1})$ unless $P\equiv 0$.
\end{lem}

 \begin{lem}[Corollary 
 4.14 of \cite{SunA}]\label{4:counting01}
Let $d\in\N_{+},r\in\N$ and $p$ be a prime number. Let  $M\colon\V\to\F_{p}$ be a   quadratic form   and $V+c$ be an affine subspace of $\V$ of co-dimension $r$.   If either $s:=\rank(M\vert_{V+c})\geq 3$ or $\rank(M)-2r\geq 3$, then $$\vert V(M)\cap (V+c)\vert=p^{d-r-1}(1+O(p^{-\frac{1}{2}})).$$	
	\end{lem}

 \begin{lem}[Corollary 
 4.16 of \cite{SunA}]\label{4:counting02}
Let $d,r\in\N_{+}$ and $p$ be a prime number. Let  $M\colon\V\to\F_{p}$ be a non-degenerate   quadratic form  and $h_{1},\dots,h_{r}$ be linearly independent vectors.  
		If $d-2r\geq 3$, then $$\vert V(M)^{h_{1},\dots,h_{r}}\vert=p^{d-r-1}(1+O(p^{-\frac{1}{2}})).$$	
\end{lem}

 \begin{lem}[Corollary 
 C.7 of \cite{SunA}]\label{4:countingh}
Let $d\in \N_{+}, r,s\in\N$ and $p$ be a prime. Let $M\colon\V\to\F_{p}$ be a quadratic form and $V+c$ be an affine subspace of $\V$ of co-dimension $r$. If either $\rank(M\vert_{V+c})$ or $\rank(M)-2r$ is at least $s^{2}+s+3$, then
	$$\vert \Gow_{s}(V(M)\cap (V+c))\vert=p^{(s+1)(d-r)-(\frac{s(s+1)}{2}+1)}(1+O_{s}(p^{-1/2})).$$
\end{lem}

 \begin{lem}[Lemma 
 4.18 of \cite{SunA}]\label{4:changeh}
Let $s\in\N$, $M\colon\V\to\F_{p}$ be a quadratic form associated with the matrix $A$, and $V+c$ be an affine subspace of $\V$  of dimension $r$. For $n,h_{1},\dots,h_{s}\in\V$, we have that $(n,h_{1},\dots,h_{s})\in \Gow_{s}(V(M)\cap (V+c))$ if and only if 
	\begin{itemize}
		\item $n\in V+c$, $h_{1},\dots,h_{s}\in V$;
		\item $n\in V(M)^{h_{1},\dots,h_{s}}$;
		\item $(h_{i}A)\cdot h_{j}=0$ for all $1\leq i,j\leq s, i\neq j$.
	\end{itemize}	
	
	In particular, let $\phi\colon\F_{p}^{r}\to V$ be any bijective linear transformation, $M'(m):=M(\phi(m)+c)$. We have that $(n,h_{1},\dots,h_{s})\in \Gow_{s}(V(M)\cap (V+c))$ if and only if 
	$(n,h_{1},\dots,h_{s})=(\phi(n')+c,\phi(h'_{1}),\dots,\phi(h'_{s}))$
	for some $(n',h'_{1},\dots,h'_{s})\in \Gow_{s}(V(M'))$.
\end{lem}

 \subsection{$M$-families and $M$-sets}\label{4:s:AppB3}

We say that a linear transformation  $L\colon(\V)^{k}\to (\V)^{k'}$ is \emph{$d$-integral} if there exist $a_{i,j}\in\F_{p}$ for $1\leq i\leq k$ and $1\leq j\leq k'$ such that 
$$L(n_{1},\dots,n_{k})=\Bigl(\sum_{i=1}^{k}a_{i,1}n_{i},\dots,\sum_{i=1}^{k}a_{i,k'}n_{i}\Bigr)$$
for all $n_{1},\dots,n_{k}\in \V$. 
Let $L\colon(\V)^{k}\to \V$ be a $d$-integral linear transformation given by
$L(n_{1},\dots,n_{k})=\sum_{i=1}^{k}a_{i}n_{i}$ for some  $a_{i}\in\F_{p},1\leq i\leq k$. We say that $L$ is the $d$-integral linear transformation \emph{induced} by $(a_{1},\dots,a_{k})\in\F_{p}^{k}$.

Similarly, we say that a linear transformation  $L\colon(\Z^{d})^{k}\to (\frac{1}{p}\Z^{d})^{k'}$ is \emph{$d$-integral} if there exist $a_{i,j}\in\Z/p$ for $1\leq i\leq k$ and $1\leq j\leq k'$ such that 
$$L(n_{1},\dots,n_{k})=\Bigl(\sum_{i=1}^{k}a_{i,1}n_{i},\dots,\sum_{i=1}^{k}a_{i,k'}n_{i}\Bigr)$$
for all $n_{1},\dots,n_{k}\in \Z^{d}$.

Let $d\in\N_{+}$, $p$  be a prime, $M\colon\V\to\F_{p}$ be a quadratic form with $A$ being the associated matrix.  We say that a function $F\colon(\V)^{k}\to\F_{p}$ is an \emph{$(M,k)$-integral quadratic function} 
if 
\begin{equation}\label{4:thisisf2}
F(n_{1},\dots,n_{k})=\sum_{1\leq i\leq j\leq k}b_{i,j}(n_{i}A)\cdot n_{j}+\sum_{1\leq i\leq k} v_{i}\cdot n_{i}+u
\end{equation}
for some $b_{i,j}, u\in \F_{p}$ and $v_{i}\in\F_{p}^{d}$.
We say that an $(M,k)$-integral quadratic function $F$ is \emph{pure}
if $F$ can be written in the form of (\ref{4:thisisf2}) with $v_{1}=\dots=v_{k}=\bold{0}$. 
We say that  an $(M,k)$-integral quadratic function $F\colon(\V)^{k}\to\F_{p}$ is \emph{nice}
if 
\begin{equation}\nonumber
F(n_{1},\dots,n_{k})=\sum_{1\leq i\leq k'}b_{i}(n_{k'}A)\cdot n_{i}+u
\end{equation}
for some $0\leq k'\leq k$, $b_{i}, u\in \F_{p}$.
 
For $F$ given in (\ref{4:thisisf2}), denote 
$$v_{M}(F):=(b_{k,k},b_{k,k-1},\dots,b_{k,1},v_{k},b_{k-1,k-1},\dots,b_{k-1,1},v_{k-1},\dots,b_{1,1},v_{1},u)\in\F_{p}^{\binom{k+1}{2}+kd+1},$$
and
$$v'_{M}(F):=(b_{k,k},b_{k,k-1},\dots,b_{k,1},v_{k},b_{k-1,k-1},\dots,b_{k-1,1},v_{k-1},\dots,b_{1,1},v_{1})\in\F_{p}^{\binom{k+1}{2}+kd}.$$
Informally, we say that $b_{i,j}$ is the $n_{i}n_{j}$-coefficient, $v_{i}$ is the $n_{i}$-coefficient, and $u$ is the constant term coefficient for these vectors.

An \emph{$(M,k)$-family} is a collections of $(M,k)$-integral quadratic functions.
Let $\mathcal{J}=\{F_{1},\dots$, $F_{r}\}$ be an $(M,k)$-family.
\begin{itemize}
    \item We say that $\mathcal{J}$ is \emph{pure} if all of $F_{1},\dots,F_{r}$ are pure.
    \item We say that $\mathcal{J}$ is \emph{consistent} if $(0,\dots,0,1)$ does not belong to the span of $v_{M}(F_{1}),$ $\dots,$ $v_{M}(F_{r})$, or equivalently, there is no linear combination of $F_{1},\dots,F_{r}$ which is a constant nonzero function, or equivalently, for all $c_{1},\dots,c_{r}\in\F_{p}$, we have
    $$c_{1}v'_{M}(F_{1})+\dots+c_{r}v'_{M}(F_{r})=\bold{0}\Rightarrow c_{1}v_{M}(F_{1})+\dots+c_{r}v_{M}(F_{r})=\bold{0}.$$
    \item We say that $\mathcal{J}$ is \emph{independent} if $v'_{M}(F_{1}),\dots,v'_{M}(F_{r})$ are linearly independent, or equivalently, there is no nontrivial linear combination of $F_{1},\dots,F_{r}$ which is a constant function, or equivalently, for all $c_{1},\dots,c_{r}\in\F_{p}$, we have
    $$c_{1}v'_{M}(F_{1})+\dots+c_{r}v'_{M}(F_{r})=\bold{0}\Rightarrow c_{1}=\dots=c_{r}=0.$$
\item We say that $\mathcal{J}$ is \emph{nice} if there exist some bijective $d$-integral  linear transformation $L\colon(\V)^{k}\to(\V)^{k}$ and some $v\in(\V)^{k}$ such that $F_{i}(L(\cdot)+v)$ is nice for all $1\leq i\leq r$.
\end{itemize}

The dimension of the span of $v'_{M}(F_{1}),\dots,v'_{M}(F_{r})$ is called the \emph{dimension} of an $(M,k)$-family $\{F_{1},\dots,F_{r}\}$.

When there is no confusion, we call an $(M,k)$-family to be an \emph{$M$-family} for short.

We say that a subset $\Omega$ of $(\V)^{k}$ is an  \emph{$M$-set} if there exists an $(M,k)$-family $\{F_{i}\colon (\V)^{k}\to\V\colon 1\leq i\leq r\}$
 such that $\Omega=\cap_{i=1}^{r}V(F_{i})$. 
 We call either $\{F_{1},\dots,F_{r}\}$ or the ordered set $(F_{1},\dots,F_{r})$ an \emph{$M$-representation} of $\Omega$. 
 
  Let $\P\in\CP$.
We say that   $\Omega$  is $\P$  if one can choose the  $M$-family $\{F_{1},\dots,F_{r}\}$ to be $\P$.
 We say that the $M$-representation $(F_{1},\dots,F_{r})$ is $\P$ if $\{F_{i}\colon 1\leq i\leq r\}$ is $\P$.
  We say that $r$ is the \emph{dimension} of the $M$-representation $(F_{1},\dots,F_{r})$.
  The \emph{total co-dimension} of a consistent $M$-set $\Omega$, denoted by $r_{M}(\Omega)$, is the minimum of the dimension of the independent $M$-representations of $\Omega$. It was shown in Proposition 
  C.8 of \cite{SunA} that $r_{M}(\Omega)$ is independent of the choice of the independent $M$-representation if $d$ and $p$ are sufficiently large.

  	We say that an independent $M$-representation $(F_{1},\dots,F_{r})$ of $\Omega$ is \emph{standard} if
	the matrix 
		$\begin{bmatrix}
		v_{M}(F_{1})\\
		\dots \\
		v_{M}(F_{r})
		\end{bmatrix}$ is in the reduced row echelon form (or equivalently, the matrix 
		$\begin{bmatrix}
		v'_{M}(F_{1})\\
		\dots \\
		v'_{M}(F_{r})
		\end{bmatrix}$ is in the reduced row echelon form).
	If $(F_{1},\dots,F_{r})$ is a standard $M$-representation of $\Omega$, then we may relabeling $(F_{1},\dots,F_{r})$ as $$(F_{k,1},\dots,F_{k,r_{k}},F_{k-1,1},\dots,F_{k-1,r_{k-1}},\dots,F_{1,1},\dots,F_{1,r_{1}})$$
  for some $r_{1},\dots,r_{k}\in \N$ such that $F_{i,j}$ is non-constant with respect to $n_{i}$ and is independent of $n_{i+1},\dots,n_{k}$.
	We also call $(F_{i,j}\colon 1\leq i\leq k, 1\leq j\leq r_{i})$	a \emph{standard $M$-representation} of $\Omega$. The vector $(r_{1},\dots,r_{k})$ is called the \emph{dimension vector} of this representation. 
	
	\begin{conv}\label{1:fpism}
	We allow $(M,k)$-families, $M$-representations and   dimension vectors to be empty. In particular, $(\V)^{k}$ is considered as a nice and consistent $M$-set with total co-dimension zero. 
	\end{conv}

\begin{prop}[Proposition 
B.3 of \cite{SunA}]\label{4:yy3}
Let $d,k,r\in\N_{+}$, $p$ be a prime, $M\colon\V\to\F_{p}$ be a quadratic form and $\mathcal{J}=\{F_{1},\dots,F_{r}\}, F_{i}\colon(\V)^{k}\to \F_{p}, 1\leq i\leq r$ be an $(M,k)$-family.
	\begin{enumerate}[(i)]
	\item For any subset $I\subseteq \{1,\dots,r\}$, $\{F_{i}\colon i\in I\}$ is an    $(M,k)$-family.
	\item Let $I\subseteq \{1,\dots,r\}$ be a subset such that all of $F_{i},i\in I$ are independent of the last $d$-variables. Then writing $G_{i}(n_{1},\dots,n_{k-1}):=F_{i}(n_{1},\dots,n_{k-1})$, we have that $\{G_{i}\colon i\in I\}$ is an $(M,k-1)$-family.  
	\item For any bijective $d$-integral linear transformation  $L\colon (\V)^{k}\to(\V)^{k}$ and any $v\in(\V)^{k}$, $\{F_{1}(L(\cdot)+v),\dots,F_{r}(L(\cdot)+v)\}, F_{i}\colon(\V)^{k}\to \F_{p}, 1\leq i\leq r$ is an $(M,k)$-family.
		\item Let $1\leq k'\leq k$ and $1\leq r'\leq r$ be such that $F_{1},\dots,F_{r'}$ are independent of $n_{k-k'+1},\dots,n_{r}$ and that every nontrival linear combination of $F_{r'+1},\dots,F_{r}$ dependents nontrivially on some of $n_{k-k'+1},\dots,n_{r}$. 
		Define $G_{i},H_{i}\colon (\V)^{k+k'}\to \F_{p}, 1\leq i\leq r$ by 
		$$\text{$G_{i}(n_{1},\dots,n_{k+k'}):=F_{i}(n_{1},\dots,n_{k})$ and $H_{i}(n_{1},\dots,n_{k+k'}):=F_{i}(n_{1},\dots,n_{k-k'},n_{k+1},\dots,n_{k+k'})$}.$$ Then $\{F_{1},\dots,F_{r'},G_{r'+1},\dots,G_{r},H_{r'+1},\dots,H_{r}\}$\footnote{Here we regard $F_{1},\dots,F_{r'}$ as functions of $n_{1},\dots,n_{k+k'}$, which acutally depends only on $n_{1},\dots,n_{k-k'}$.} is an $(M,k+k')$-family. 
	\end{enumerate}
Moreover, if $\mathcal{J}$ is nice, then so are the sets mentioned in Parts (i)-(iii);
if $\mathcal{J}$ is 
consistent/independent, then so are the sets mentioned in Parts (i)-(iv);  
if $\mathcal{J}$ is pure, then so are the sets mentioned in Parts (i), (ii), (iv), and so is the set mentioned in Part (iii) when $v=\bold{0}$.
\end{prop}

\begin{prop}[Proposition 
C.8 of \cite{SunA}]\label{4:yy33}
Let $d,k,r\in\N_{+}$, $p$ be a prime, $M\colon\V\to\F_{p}$ be a quadratic form and $\Omega\subseteq (\V)^{k}$ be a consistent $M$-set. Suppose that $\Omega=\cap_{i=1}^{r}V(F_{i})$ for some consistent $(M,k)$-family $\{F_{1},\dots,F_{r}\}$. 
Let $1\leq k'\leq k$.
Suppose that $\rank(M)\geq 2r+1$ and $p\gg_{k,r} 1$.
\begin{enumerate}[(i)]
\item The dimension of all independent $M$-representations of $\Omega$ equals to $r_{M}(\Omega)$. 
\item We have that $r_{M}(\Omega)\leq r$, and that $r_{M}(\Omega)=r$ if
the $(M,k)$-family $\{F_{1},\dots,F_{r}\}$ is independent.
\item For all $I\subseteq\{1,\dots,r\}$, the set $\Omega':=\cap_{i\in I}V(F_{i})$ is a consistent  $M$-set such that $r_{M}(\Omega')\leq r_{M}(\Omega)$. Moreover, if  the $(M,k)$-family $\{F_{1},\dots,F_{r}\}$ is independent, then $r_{M}(\Omega')=\vert I\vert$.
\item For any bijective $d$-integral linear transformation $L\colon(\V)^{k}\to(\V)^{k}$ and $v\in(\V)^{k}$, we have that $L(\Omega)+v$ is a consistent $M$-set and that $r_{M}(L(\Omega)+v)=r_{M}(\Omega)$.   
\item Assume that $\Omega$ admits a standard $M$-representation with dimension vector $(r_{1},\dots,$ $r_{k})$, Then for all $1\leq k'\leq k$, the set $$\{(n_{1},\dots,n_{k+k'})\in(\V)^{k+k'}\colon (n_{1},\dots,n_{k}), (n_{1},\dots,n_{k-k'},n_{k+1},\dots,n_{k+k'})\in\Omega\}$$ admits a standard $M$-representation with dimension vector $(r_{1},\dots,r_{k'},r_{k+1},\dots,r_{k'})$. 
\item Assume that $\Omega$ admits a standard $M$-representation with dimension vector $(r_{1},\dots,$ $r_{k})$.
For $I=\{1,\dots,k'\}$ for some $1\leq k'\leq k$, the $I$-projection $\Omega_{I}$ of $\Omega$ 
admits a standard $M$-representation with dimension vector $(r_{1},\dots,r_{k'})$. 
\item If $\Omega$ is a nice and consistent $M$-set or a  pure and consistent $M$-set, then $r_{M}(\Omega)\leq \binom{k+1}{2}$.
\end{enumerate}
\end{prop}

\subsection{Fubini's theorem for $M$-sets}\label{4:s:AppB4}		
Let $n_{i}=(n_{i,1},\dots,n_{i,d})\in\V$ denote a $d$-dimensional variable for $1\leq i\leq k$. For convenience we denote $\F^{d}_{p}[n_{1},\dots,n_{k}]$ to be the polynomial ring $\F_{p}[n_{1,1},\dots,n_{1,d},\dots,n_{k,1},\dots,n_{k,d}]$.
Let $J,J',J''$ be  finitely generated ideals of the polynomial ring $\F^{d}_{p}[n_{1},\dots,n_{k}]$ and $I\subseteq \{n_{1},\dots,n_{k}\}$. 
Suppose that $J=J'+J''$, $J'\cap J''=\{0\}$, all the polynomials in $J'$ are independent of  $\{n_{1},\dots,n_{k}\}\backslash I$, and all the non-zero polynomials in $J''$ depend nontrivially on $\{n_{1},\dots,n_{k}\}\backslash I$. Then we say that $J'$ is an  \emph{$I$-projection} of $J$ and that $(J',J'')$ is an \emph{$I$-decomposition} of $J$.
 It was proved in Proposition 
 C.1 of \cite{SunA} that the $I$-projection of  $J$ exists and is unique.

Let $\mathcal{J}$, $\mathcal{J}'$ and $\mathcal{J}''$ be  finite subsets of $\F^{d}_{p}[n_{1},\dots,n_{k}]$, and $I\subseteq \{n_{1},\dots,n_{k}\}$. 
Let $J,J'$, and $J''$ be the ideals generated by $\mathcal{J}$, $\mathcal{J}'$ and $\mathcal{J}''$ respectively. If $J'$ is an $I$-projection of $J$ with $(J',J'')$ being an $I$-decomposition of $J$, then we say that $\mathcal{J}'$ is an \emph{$I$-projection} of $\mathcal{J}$ and that $(\mathcal{J}',\mathcal{J}'')$ is an \emph{$I$-decomposition} of $\mathcal{J}$.
Note that the $I$-projection $\mathcal{J}'$ of $\mathcal{J}$ is not unique. However, by Proposition 
C.1 of \cite{SunA}, the ideal generated by $\mathcal{J}'$ and the set of zeros $V(\mathcal{J}')$ are unique.

If $I$ is a subset of $\{1,\dots,k\}$, for convenience we say that $J'$ is an \emph{$I$-projection} of $J$ if $J'$ is an $\{n_{i}\colon i\in I\}$-projection of $J$. Similarly, we say that $(J',J'')$ is an \emph{$I$-decomposition} of $J$ if $(J',J'')$ is an $\{n_{i}\colon i\in I\}$-decomposition of $J$. Here $J,J',J''$ are either ideals or finite subsets of the polynomial ring.

We now define the projections of $M$-sets.
Let $\mathcal{J}\subseteq \F_{p}[n_{1},\dots,n_{k}]$ be a consistent $(M,k)$-family and 
let $\Omega=V(\mathcal{J})\subseteq(\V)^{k}$.
Let $I\cup J$ be a partition of $\{1,\dots,k\}$ (where $I$ and $J$ are non-empty), and $(\mathcal{J}_{I},\mathcal{J}'_{I})$ be an $I$-decomposition of $\mathcal{J}$.
Let $\Omega_{I}$ denote the set of $(n_{i})_{i\in I}\in(\V)^{\vert I\vert}$ such that $f(n_{1},\dots,n_{k})=0$ for all $f\in \mathcal{J}_{I}$ and $(n_{j})_{j\in J}\in(\V)^{\vert J\vert}$. 
Note that all $f\in \mathcal{J}_{I}$ are independent of $(n_{j})_{j\in J}$, and that $\Omega_{I}$ is independent of the choice of the $I$-decomposition.
We say that $\Omega_{I}$ is an \emph{$I$-projection} of $\Omega$.

 For $(n_{i})_{i\in I}\in (\V)^{\vert I\vert}$, let $\Omega_{I}((n_{i})_{i\in I})$ denote the set of $(n_{j})_{j\in J}\in(\V)^{\vert J\vert}$ such that $f(n_{1},\dots,n_{k})=0$ for all $f\in \mathcal{J}'_{I}$. By construction, for any $(n_{i})_{i\in I}\in \Omega_{I}$, we have that $(n_{j})_{j\in J}\in\Omega_{I}((n_{i})_{i\in I})$ if and only if $(n_{1},\dots,n_{k})\in\Omega$. So for all  $(n_{i})_{i\in I}\in \Omega_{I}$, $\Omega_{I}((n_{i})_{i\in I})$ is independent of the choice of the $I$-decomposition.

\begin{thm}[Theorem 
C.3 of \cite{SunA}]\label{4:ct}
	Let $d,k\in\N_{+}$, $r_{1},\dots,r_{k}\in\N$ and $p$ be a prime.  Set $r:=r_{1}+\dots+r_{k}$.  
	Let  $M\colon\V\to\F_{p}$ be a  quadratic form with $\rank(M)\geq 2r+1$, and $\Omega\subseteq (\V)^{k}$ be a consistent $M$-set admitting a  standard $M$-representation of $\Omega$ with dimension vector $(r_{1},\dots,r_{k})$. 
		\begin{enumerate}[(i)]
		\item We have $\vert\Omega\vert=p^{dk-r}(1+O_{k,r}(p^{-1/2}))$;
		\item If $I=\{1,\dots,k'\}$ for some $1\leq k'\leq k-1$, then 
		$\Omega_{I}$ is a consistent $M$-set admitting a  standard  $M$-representation with dimension vector $(r_{1},\dots,r_{k'})$. Moreover,
		for all but at most $(k-k')r'_{k'}p^{dk'+r'_{k'}-1-\rank(M)}$ many $(n_{i})_{i\in I}\in(\V)^{k'}$, we have that $\Omega_{I}((n_{i})_{i\in I})$ has a standard $M$-representation with dimension vector $(r_{k'+1},\dots,r_{k})$  and that $\vert\Omega_{I}((n_{i})_{i\in I})\vert=p^{d(k-k')-(r_{k'+1}+\dots+r_{k})}(1+O_{k,r}(p^{-1/2}))$, where $r'_{k'}:=\max_{k'+1\leq i\leq k}r_{i}$. 
		\item 
		If $k\geq 2$, and $I\cup J$ is a partition of $\{1,\dots,k\}$ (where  $I$ and $J$ are non-empty), then for any function $f\colon \Omega\to\C$ with norm bounded by 1,   we have that
		\begin{equation}\nonumber
		\E_{(n_{1},\dots,n_{k})\in\Omega}f(n_{1},\dots,n_{k})=\E_{(n_{i})_{i\in I}\in\Omega_{I}}\E_{(n_{j})_{j\in J}\in\Omega_{I}((n_{i})_{i\in I})}f(n_{1},\dots,n_{k})+O_{k,r}(p^{-1/2}).
		\end{equation}
	\end{enumerate} 
\end{thm}

\subsection{Intrinsic definitions for polynomials}\label{4:s:ipp}

\begin{prop}[A special case of Proposition 
7.12 of \cite{SunA}]\label{4:att3}
 	Let $d,s\in\N_{+}$, $p\gg_{d} 1$ be a prime, and $M\colon\Z^{d}\to\Z/p$ be a  quadratic form associated with the matrix $A$ of $p$-rank at least $s^{2}+s+3$.
 	Denote $M'\colon\Z^{d}\to\Z/p$, $M'(n):=\frac{1}{p}((nA)\cdot n)$.
	Then for any  $g\in\poly(\Z^{d}\to \Q)$  of degree at most $s$ taking values in $\Q$, the following are equivalent:
 	\begin{enumerate}[(i)]
 		\item $\Delta_{h_{s}}\dots\Delta_{h_{1}}g(n)\in \frac{1}{q}\Z$ for some $q\in\N,p\nmid q$ for all $(n,h_{1},\dots,h_{s})\in \Gow_{p,s}(V_{p}(M))$;
 		\item  we have $$g=\frac{1}{q}g_{1}+g_{2}$$ for some $q\in\N,p\nmid q$,
 		some $g_{1}\in \poly(V_{p}(M)\to \R\vert\Z)$ of degree at most $s$ and some $g_{2}\in\poly(\Z^{d}\to \R)$ of degree at most $s-1$;
 		\item we have $$g=\frac{1}{q}g_{1}+g_{2}$$ for some $q\in\N,p\nmid q$,
 		some $g_{1}\in \poly(V_{p}(M')\to \R\vert\Z)$ of degree at most $s$, and some $g_{2}\in\poly(\Z^{d}\to \R)$ of degree at most $s-1$.
 	\end{enumerate}	 
 \end{prop}

\subsection{An irreducible property for $M$-sets}\label{4:s:ir}


Let $k\in\N_{+},s\in\N, \d>0$ and $p$ be a prime.  
	We say that a $p$-periodic subset   $\Omega$ of $\Z^{k}$  is   \emph{strongly $(\d,p)$-irreducible up to degree $s$} if for all $P\in\poly(\Z^{k}\to\Z/p^{s})$ of degree at most $s$, either
	$\vert V_{p}(P)\cap \Omega\cap [p]^{d}\vert\leq \d\vert\Omega\cap [p]^{d}\vert$ or $\Omega\subseteq V_{p}(P).$

\begin{thm}[A special case of Theorem 
9.14 of \cite{SunA}]\label{4:irrr}
Let $d,K\in\N_{+},s\in\N$, $p$ be a prime, and $\d>0$. Let $M\colon\V\to \F_{p}$ be a  non-degenerate quadratic form and  $\Omega\subseteq (\V)^{K}$ be a  nice and consistent $M$-set. If 	$d\geq \max\{2r_{M}(\Omega)+1,4K-1\}$  and $p\gg_{d,s} \d^{-O_{d,s}(1)}$, then 
	 $\iota^{-1}(\Omega)$ is strongly $(\d,p)$-irreducible up to degree $s$.
\end{thm}

\subsection{Leibman dichotomy for $M$-sets}\label{4:s:vr}

Let $d\in\N_{+},s\in\N$, $C,\d,K>0$,   $p$ be a prime, and $\Omega$ be a non-empty subset of $\V$. We say that $\Omega$ admits a \emph{partially periodic $(\d,K)$-Leibman dichotomy up to step $s$ and complexity $C$}  
if 	for any   $\N$-filtered nilmanifold   $G/\Gamma$  of degree at most $s$ and complexity at most $C$, and any $g(n)\in \poly_{p}(\Omega\to G_{\N}\vert\Gamma)$, either $(g(n)\Gamma)_{n\in \Omega}$ is  $\d$-equidistributed on $G/\Gamma$, or  there exists a nontrivial type-I horizontal character $\eta$ with $0<\Vert\eta\Vert\leq K$  such that $\eta\circ g\vert_{\Omega}\colon \Omega\to\R$  is a constant mod $\Z$. 

\begin{thm}[Theorem 
10.6 of \cite{SunA}]\label{4:veryr}
	Let $d,k\in\N_{+},s,r\in\N$ with $d\geq \max\{4r+1,4k+3,2k+s+11\}$, $C>0$ and $p$ be a prime. 
	There exists  $K:=O_{C,d}(1)$ such that for any  non-degenerate quadratic form $M\colon\V\to \F_{p}$, any nice and consistent $M$-set $\Omega\subseteq (\V)^{k}$ of total co-dimension $r$, and any $0<\d<1/2$,  if $p\gg_{C,d} \d^{-O_{C,d}(1)}$, 
	then  $\Omega$ admits a partially periodic $(\d,K\d^{-K})$-Leibman dichotomy up to step $s$ and complexity $C$.	
\end{thm}

\bibliographystyle{plain}
\bibliography{swb}
\end{document}